\documentclass[12pt]{amsart}
\usepackage{amsmath,amsthm,amsfonts,amssymb}
\usepackage{graphicx}
\usepackage{mathrsfs}
\usepackage{mathtools}
\usepackage{enumitem}
\usepackage{etoolbox}
\usepackage{xcolor}
\apptocmd{\sloppy}{\hbadness 10000\relax}{}{}
\apptocmd{\sloppy}{\vbadness 10000\relax}{}{}
\usepackage[pdfpagelabels]{hyperref}
\usepackage[letterpaper,margin=1.1in]{geometry}

\allowdisplaybreaks

\numberwithin{equation}{section}
\theoremstyle{plain}
\newtheorem{theorem}{Theorem}[section]

\newtheorem{corollary}[theorem]{Corollary}

\newtheorem{lemma}[theorem]{Lemma}

\newtheorem{keylemma}{Lemma}

\theoremstyle{definition}
\newtheorem{remark}[theorem]{Remark}
\newtheorem{definition}[theorem]{Definition}
\newtheorem{example}[theorem]{Example}

\newcommand{\Span}{\mathop\mathrm{span}\nolimits}

\def\RR{\mathbb{R}}
\def\ZZ{\mathbb{Z}}

\def\XX{{\mathbb{X}}}

\newcommand{\radius}{\mathop\mathrm{radius}\nolimits}
\newcommand{\diam}{\mathop\mathrm{diam}\nolimits}
\newcommand{\dist}{\mathop\mathrm{dist}\nolimits}

\newcommand{\Mod}{\operatorname{mod}}

\newcommand{\res}{\hbox{ {\vrule height .22cm}{\leaders\hrule\hskip.2cm} }}
\newcommand{\Haus}{\mathcal{H}}
\newcommand{\Start}{\mathop\mathsf{Start}\nolimits}
\newcommand{\End}{\mathop\mathsf{End}\nolimits}
\newcommand{\Domain}{\mathop\mathsf{Domain}\nolimits}
\newcommand{\Image}{\mathop\mathsf{Image}\nolimits}

\newcommand{\Diam}{\mathop\mathsf{Diam}\nolimits}

\newcommand{\Edge}{\mathop\mathsf{Edge}\nolimits}
\newcommand{\Child}{\mathop\mathsf{Child}\nolimits}

\newcommand{\flatepsilon}{\epsilon_2}
\newcommand{\flatarcsepsilon}{\epsilon_1}

\newcommand{\gap}{\mathop\mathrm{gap}\nolimits}

\newcommand{\interior}{\mathop\mathrm{int}\nolimits}

\numberwithin{figure}{section}

\begin{document}

\title[Subsets of rectifiable curves in Banach spaces II]{Subsets of rectifiable curves in Banach spaces II:\\ universal estimates for almost flat arcs}

\author{Matthew Badger \and Sean McCurdy}
\thanks{M.~Badger was partially supported by NSF DMS grants 1650546 and 2154047.}
\date{August 22, 2022}

\address{Department of Mathematics\\ University of Connecticut\\ Storrs, CT 06269-1009}
\email{matthew.badger@uconn.edu}
\address{Department of Mathematics\\ National Taiwan Normal University\\ Taipei City\\ Taiwan (R.\,O.\,C.)}
\email{smccurdy@ntnu.edu.tw}

\begin{abstract} We prove that in any Banach space the set of windows in which a rectifiable curve resembles two or more straight line segments is quantitatively small with constants that are independent of the curve, the dimension of the space, and the choice of norm. Together with Part I, this completes the proof of the necessary half of the Analyst's Traveling Salesman theorem with sharp exponent in uniformly convex spaces.
\end{abstract}

\maketitle

\vspace{-.32in}

\tableofcontents

\vspace{-.63in}

\listoffigures

\section{Introduction}\label{sec:intro}

\subsection{Background}

Given ``snapshots'' of a set $E$ in a metric space $\XX$ at all locations and scales, the \emph{Analyst's Traveling Salesman Problem} is to determine whether or not $E$ is contained in a rectifiable curve, and if so, to estimate the length of the shortest such curve. Full solutions to the Analyst's TSP (characterizations of subsets of rectifiable curves) have been found in $\RR^n$ \cite{Jones-TST,Ok-TST}, in arbitrary Carnot groups \cite{Li-TSP}, in Hilbert space \cite{Schul-Hilbert}, and in certain fractal-like metric spaces \cite{Guy-Schul}. For the related \emph{Measure-theorist's Traveling Salesman Problem} and its solution in $\RR^n$ and also in Carnot groups, see \cite{BS3,Badger-survey,BLZ}. Partial results on the Analyst's TSP in other metric spaces have been obtained by Hahlomaa \cite{Hah05,Hah08} and David and Schul \cite{DS-metric} and for higher-dimensional objects \cite{AS-TST,BNV,Hyde-TST,Villa-TST}. Refined estimates on the length of the shortest Jordan curve containing a set in $\RR^n$ or Hilbert space have been given in \cite{Bishop-TST,Krandel-Hilbert}. In Part I \cite{Badger-McCurdy-1} and in the present paper, we address the Analyst's TSP on a general Banach space.

Let $\XX$ be a (real) Banach space, let $E\subset \XX$ be a nonempty set, and let $Q\subset\XX$ be a set of finite, positive diameter. If $E\cap Q\neq \emptyset$, we define \begin{equation}\label{def:beta} \beta_E(Q) = \inf_{L} \sup_{x\in E\cap Q}\frac{\dist(x,L)}{\diam Q}\in[0,1],\end{equation} where the infimum ranges over all lines $L\subset \XX$. If $E\cap Q=\emptyset$, then we assign $\beta_E(Q)=0$. These are a geometric variant of least squares errors introduced in \cite{Jones-TST} and are now called \emph{Jones' beta numbers}. If $\beta_E(Q)=0$, then the portion of the set $E$ inside of the ``window'' $Q$ is contained in some line $L$; if $\beta_E(Q) \gtrsim 1$, then for each line $L$ passing through $Q$, at least some part of $E\cap Q$ is far away from $L$. An easy, but important consequence of the definition is \begin{equation} \label{beta-monotone} \beta_E(R) \leq \frac{\diam Q}{\diam R}\, \beta_F(Q)\quad\text{for all }E\subset F\text{ and }R\subset Q.\end{equation} Thus, an estimate of flatness at one scale yields (a worse) estimate of flatness at a smaller scale. Because any rectifiable curve $\Gamma\subset\XX$ admits tangents lines almost everywhere with respect to the 1-dimensional Hausdorff measure $\mathcal{H}^1$, it is perhaps reasonable to expect that $\lim_{r\rightarrow 0}\beta_{\Gamma}(B(x,r))=0$ at $\Haus^1$-a.e.~$x\in\Gamma$. Following \cite{Jones-TST}, which marks the start of quantitative geometric measure theory as its own subject, we are interested in making this qualitative statement more precise.

In Part I \cite{Badger-McCurdy-1}, we established universal sufficient conditions for a set in an arbitrary Banach space to be contained inside a rectifiable curve, as well as improved estimates on the length of the shortest curve containing a set in uniformly smooth spaces. The origin of this result is Jones's criterion \cite{Jones-TST} for the existence of a rectifiable curve passing through a given set in $\RR^n$, which is usually stated using systems of dyadic cubes. However, because we work in infinite-dimensional settings, we prefer to use Schul's formulation \cite{Schul-Hilbert} in terms of multiresolution families. Recall that an \emph{$\epsilon$-net} for $E\subset\XX$ is a maximal set $X\subset E$ such that $|x-y|\geq \epsilon$ for all distinct $x,y\in X$. A \emph{multiresolution family} $\mathscr{G}$ for $E$ with \emph{inflation factor} $A_\mathscr{G}>1$ is a family $\{B(x,A_\mathscr{G} 2^{-k}):x\in X_k,\,k\in\ZZ\}$ of closed balls with centers in some nested family $\cdots\subset X_{-1}\subset X_0\subset X_{1}\subset\cdots$ of $2^{-k}$-nets $X_k$ for $E$. Analogously, if each set $X_k$ is a $2^{-k}$-separated set, but possibly one or more of the sets $X_k$ are not $2^{-k}$-nets, then we call $\mathscr{G}$ a \emph{partial multiresolution family} for $E$.

\begin{theorem}[{\cite{Jones-TST,Schul-Hilbert,Badger-McCurdy-1}}] \label{t:suff} Let $\XX$ be Banach space and let $1\leq p\leq 2$. Suppose that \begin{itemize}
\item[(i)] $\XX$ is an arbitrary Banach space and $p=1$; or,
\item[(ii)] $\XX$ is a uniformly smooth Banach space of power type $1<p\leq 2$; or,
\item[(iii)] $\XX$ is a Hilbert space and $p=2$; or,
\item[(iv)] $\XX$ is a finite-dimensional Banach space and $p=2$.
\end{itemize} If $E\subset\XX$, $\mathscr{G}$ is a multiresolution family for $E$ with inflation factor $A_\mathscr{G}\geq 240$, and \begin{equation}S_{E,p}(\mathscr{G}):=\diam E+\sum_{Q\in\mathscr{G}} \beta_E(Q)^p\diam Q<\infty,\end{equation} then $E$ is contained in a rectifiable curve $\Gamma\subset\XX$ with $\Haus^1(\Gamma)\lesssim_{A_\mathscr{G},\XX} S_{E,p}(\mathscr{G})$. (When $p=1$, restrict the sum in the definition of $S_{E,1}(\mathscr{G})$ to balls $Q\in\mathscr{G}$ with $\diam Q\lesssim \diam E$.)
\end{theorem}

\begin{remark} In cases (i) and (iii), the implicit constant in Theorem \ref{t:suff} in the comparison $\Haus^1(\Gamma)\lesssim_{A_\mathscr{G},\XX} S_{E,p}(\mathscr{G})$ depends only on $A_\mathscr{G}$. In case (ii), the implicit constant depends only on $A_\mathscr{G}$ and the modulus of smoothness of $\XX$. In case (iv), the implicit constant depends on $A_\mathscr{G}$, the dimension of $\XX$, and the bi-Lipschitz constant of a chosen embedding $\XX\hookrightarrow \ell_2^{\,\dim\XX}$.\end{remark}

In the present paper, we complete the proof of the following theorem, which is dual to Theorem \ref{t:suff}. Where the modulus of smoothness is the relevant characteristic of a space for sufficient conditions, the modulus of convexity of the space is the relevant characteristic for necessary conditions. The special cases $X=\RR^2$ and $X=\RR^n$, $n\geq 3$ of Theorem \ref{t:nec} are originally due to Jones \cite{Jones-TST} and Okikiolu \cite{Ok-TST}, respectively. When $\XX$ is an infinite-dimensional Hilbert space, the theorem was identified in \cite{Schul-Hilbert}, but the proof in that paper has serious gaps (see \cite[Remark 3.8]{Badger-McCurdy-1} and Appendix \ref{appendix328}) and a complete proof seems to not have been written until now. (An alternative fix of some portions of the proof in Schul's paper is proposed by Krandel \cite{Krandel-Hilbert}.)

\begin{theorem} \label{t:nec} Let $\XX$ be a Banach space and let $2\leq p<\infty$. Suppose that \begin{itemize}
\item[(i)] $\XX$ is a uniformly convex Banach space of power type $2\leq p<\infty$; or,
\item[(ii)] $\XX$ is a Hilbert space and $p=2$; or,
\item[(iii)] $\XX$ is a finite-dimensional Banach space and $p=2$.
\end{itemize} If $E\subset\XX$ is contained in a rectifiable curve $\Gamma$ and $\mathscr{G}$ is any (partial) multiresolution family for $E$, then $S_{E,p}(\mathscr{G})\lesssim_{A_{\mathscr{G}},\XX} \Haus^1(\Gamma)$.
\end{theorem}

\begin{remark} Again, in case (ii), the implicit constant in the comparison $S_{E,p}(\mathscr{G})\lesssim_{A_\mathscr{G},\XX}$ depends only on the inflation factor $A_\mathscr{G}$. In case (i), the implicit constant depends only on $A_\mathscr{G}$ and the modulus of convexity of $\XX$. In case (iii), the implicit constant depends on $A_\mathscr{G}$, the dimension of $\XX$, and the bi-Lipschitz constant of an embedding $\XX\hookrightarrow \ell_2^{\,\dim\XX}$.\end{remark}

Combining Theorems \ref{t:suff} and \ref{t:nec}, we recover Schul's solution of the Analyst's TSP in Hilbert space \cite{Schul-Hilbert}. For derivation of Jones's and Okikiolu's dyadic cube formulation of Corollary \ref{c:hilbert} in any finite-dimensional Banach space, see \cite[\S4]{Badger-McCurdy-1}.

\begin{corollary} \label{c:hilbert} Let $\XX$ be any Hilbert space. A bounded set $E\subset\XX$ is a subset of a rectifiable curve in $\XX$ if and only if \begin{equation}\label{hilbert-condition}\sum_{Q\in \mathscr{G}}\beta_E(Q)^2\diam Q<\infty\end{equation} for some (for every) multiresolution family $\mathscr{G}$ for $E$ with inflation factor $A_\mathscr{G}\geq 240$. Furthermore, if \eqref{hilbert-condition} holds, then $E$ is contained in some rectifiable curve $\Gamma$ with extrinsic length $\Haus^1(\Gamma)\lesssim_{A_\mathscr{G}} S_{E,2}(\mathscr{G})$.
\end{corollary}

The solution of the Analyst's TSP in Hilbert space depends heavily on the Pythagorean theorem as well as invariance of distances under orthogonal transformation. These special features of Hilbert space are not available in a general Banach space. While Theorem \ref{t:suff} gives a sufficient test for a set to lie in a rectifiable curve and Theorem \ref{t:nec} provides us  necessary conditions, a complete characterization of subsets of rectifiable curves in an infinite-dimensional non-Hilbert Banach space is still unknown. The following example and remarks show that a new idea is needed. See \cite{Schul-survey} for further discussion of the underlying challenges and \cite{DS-metric} for recent partial progress.

\begin{example} If $1<p<\infty$, then the Banach space $\XX=\ell_p$ of real-valued sequences $x=(x_i)_1^\infty$ with $\|x\|_p=\left(\sum_1^\infty |x_i|^p\right)^{1/p}$ is uniformly smooth of power type $\min\{p,2\}$ and uniformly convex of power type $\max\{2,p\}$. Let $E\subset \ell_p$ be bounded. By Theorem \ref{t:suff}, \begin{equation}\label{p-suff}\sum_{Q\in\mathscr{G}} \beta_E(Q)^{\min\{p,2\}}\diam Q<\infty\Longrightarrow \text{$E$ lies inside some rectifiable curve $\Gamma$}.\end{equation} By Theorem \ref{t:nec}, \begin{equation}\label{p-nec} \sum_{Q\in\mathscr{G}} \beta_E(Q)^{\max\{2,p\}}\diam Q<\infty\Longleftarrow \text{$E$ lies inside some rectifiable curve $\Gamma$.}\end{equation} Because $\min\{p,2\}<\max\{2,p\}$ unless $p=2$, this means that there is a strict gap between Theorem \ref{t:suff} and \ref{t:nec} for infinite-dimensional non-Hilbert Banach spaces.\end{example}

\begin{remark} In \cite[\S5]{Badger-McCurdy-1}, we construct examples that show that the exponents in \eqref{p-suff} and \eqref{p-nec} are sharp. For instance, for any $2\leq p<\infty$, we build a curve $\Gamma$ in $\ell_p$ with $\Haus^1_{\ell_p}(\Gamma)<\infty$ and $S_{E,p-\epsilon}(\mathscr{G})=\infty$ for all $\epsilon>0$.\end{remark}

\begin{remark} Equivalence of norms on finite-dimensional spaces ensures that a curve is rectifiable independent of the choice of norm (although the length depends on the norm). By contrast, the infinite-dimensional $\ell_p$ spaces are distinguished by their rectifiable curves in the following sense. For each $1<p<\infty$, there exists a curve $\Gamma$ in $\ell_p$ such that $\Haus^1_{\ell_p}(\Gamma)=\infty$ and $\Haus^1_{\ell_{p+\epsilon}}(\Gamma)<\infty$ for all $\epsilon>0$. See \cite[Proposition 1.1]{Badger-McCurdy-1}.\end{remark}

The proof of Theorem \ref{t:nec} for uniformly convex Banach spaces started in \cite[\S3]{Badger-McCurdy-1} follows the outline of the argument in \cite{Schul-Hilbert}, but with the correction noted in \cite[Remark 3.24]{Badger-McCurdy-1}, which required weakening Schul's original definition of ``almost flat arcs''. More specifically, we proved Theorem \ref{t:nec} modulo verification of \cite[Theorem 3.30]{Badger-McCurdy-1}, which is the \hyperref[t:main]{Main Theorem} of this paper. Roughly speaking, the main theorem is a quantitative strengthening of the statement that at $\Haus^1$ almost every point, at sufficiently small scales, a rectifiable curve does not resemble a union of two or more line segments. By proving the main theorem, we shall complete the demonstration of Theorem \ref{t:nec}.

The estimates that we establish below are \emph{universal} in so far as they are valid in \emph{any} Banach space. Because of the general setting, we have very few tools at our disposal. Our primary tools are the triangle inequality, connectedness of arcs, and the existence of Lipschitz projections onto 1-dimensional subspaces (see Appendix \ref{sec:smooth}).

\subsection{Almost flat arcs and statement of the \texorpdfstring{\hyperref[t:main]{Main Theorem}}{Main Theorem}} \label{sec:main}

For the remainder of the paper fix a Banach space $(\XX,|\cdot|)$, a rectifiable curve $\Gamma$ in $\XX$, a (partial) multiresolution family $\mathscr{H}$ for $\Gamma$ with inflation factor $A_\mathscr{H}>1$ and centers in a family $(X_k)_{k\in\ZZ}$ of $2^{-k}$-separated sets for $\Gamma$, and a continuous parameterization $f:[0,1]\rightarrow\Gamma$. For the purpose of proving the \hyperref[t:main]{Main Theorem} below, we do not need to (and shall not) place any restrictions on the modulus of continuity or multiplicity of $f$, but if so desired, one may assume as in Part I that $f$ is Lipschitz continuous, $\#f^{-1}\{x\}\leq 2$ for $\Haus^1$-a.e.~$x\in\Gamma$, and $f(0)=f(1)$ (see \cite{AO-curves}).

\begin{definition}[classification of arcs {\cite{Badger-McCurdy-1}}] \label{def:arcs} An \emph{arc}, $\tau=f|_{[a,b]}$, of $\Gamma$ is the restriction of $f$ to some interval $[a,b]\subset[0,1]$. Given an arc $\tau:[a,b]\rightarrow\Gamma$, define $$\Domain(\tau)=[a,b],\quad \Start(\tau)=\tau(a)=f(a),\quad\End(\tau)=\tau(b)=f(b),$$ $$\Image(\tau)=\tau([a,b])=f([a,b])\quad\text{and}\quad \Diam(\tau)=\diam\Image(\tau).$$
For any ball $Q\in\mathscr{H}$ and scaling factor $\lambda\geq 1$, let \begin{equation}\Lambda(\lambda Q):=\left\{f|_{[a,b]}:\begin{array}{c}[a,b]\text{ is a connected component of }f^{-1}(\Gamma\cap 2\lambda Q)\\ \text{such that }\lambda Q\cap f([a,b])\neq\emptyset\end{array}\right\}.\end{equation} The elements in $\Lambda(\lambda Q)$ are arcs in $2\lambda Q$ that touch $\lambda Q$. Agree to write $\beta_{\Lambda(\lambda Q)}(2\lambda Q)$ as shorthand for $\beta_{\bigcup\{\Image(\tau):\tau\in\Lambda(\lambda Q)\}}(2\lambda Q)$.

An arc $\tau\in\Lambda(\lambda Q)$ is called \emph{$*$-almost flat} if \begin{equation}\label{e:star-flat} \beta(\tau):=\beta_{\Image(\tau)}(\Image(\tau))= \inf_{L}\sup_{z\in\Image(\tau)} \frac{\dist(z,L)}{\Diam(\tau)} \leq 50\flatepsilon \beta_{\Lambda(\lambda Q)}(2\lambda Q),\end{equation} where $L$ ranges over all lines in $\XX$ and $0<\flatepsilon\ll 1$ is a constant depending on at most the inflation factor $A_\mathscr{H}$ of $\mathscr{H}$ and $\epsilon_1$ (see Definition \ref{B-balls-def}). Denote the set of $*$-almost flat arcs in $\Lambda(\lambda Q)$ by $S^*(\lambda Q)$.

An arc $\tau\in\Lambda(\lambda Q)$ is called \emph{almost flat} if $\beta(\tau)\leq \flatepsilon \beta_{\Sigma}(Q).$ Denote the set of almost flat arcs in $\Lambda(\lambda Q)$ by $S(\lambda Q)$. An arc $\tau\in\Lambda(\lambda Q)\setminus S(\lambda Q)$ is called \emph{dominant}.\label{dom-arcs}%
\end{definition}

\begin{remark}We do not require that arcs be 1-to-1. By \eqref{beta-monotone}, every almost flat arc is $*$-almost flat provided that $\lambda \leq 25$. The peculiar definition of $*$-almost flat arc, i.e.~the constant 50 in \eqref{e:star-flat}, and the focus on scaling factors $\lambda\in\{1,5\}$ in arguments below are made in order to implement the proof of \cite[Lemma 3.29]{Badger-McCurdy-1}. However, these choices will play no direct role in the arguments in this paper.\end{remark}

Below, given an arc $\tau$ and window $Q$, we write $\beta_\tau(Q)$ as shorthand for $\beta_{\Image(\tau)}(Q)$. Similarly, given a set $S$ of arcs, we write $\beta_S(Q)=\beta_{\bigcup\{\Image(\tau):\tau\in S\}}(Q)$.

\begin{definition}[$\mathscr{B}$ balls] \label{B-balls-def} Let $0<\epsilon_1\ll 1$ be a constant depending on at most the inflation factor $A_\mathscr{H}$ of $\mathscr{H}$. Given $\lambda\geq 1$, let $\mathscr{B}^\lambda$ denote the collection of all balls $Q\in\mathscr{H}$ such that \begin{enumerate}
\item[(i)] $\beta_{\Gamma}(Q)\neq 0$ and $\Gamma\setminus 14 Q\neq\emptyset$;
\item[(ii)] if $\tau\in\Lambda(\lambda Q)$ and $\Image(\tau)$ intersects the \emph{net ball} $(1/3A_{\mathscr{H}})Q=B(x,(1/3)2^{-k})$ near the center of $Q=B(x,A_\mathscr{H}2^{-k})$, with $x\in X_k$, then $\tau\in S(\lambda Q)$, and \label{netball}
\item[(iii)] $\beta_{S^*(\lambda Q)}(2\lambda Q) > \epsilon_1 \beta_{\Lambda(\lambda Q)}(2\lambda Q)$.
\end{enumerate}Assign $\mathscr{B}=\mathscr{B}^1\cup\mathscr{B}^5$.\end{definition}

\begin{remark} In Part I, we take $\epsilon_1=1/126A_\mathscr{H}$ to prove and use \cite[Lemma 3.29]{Badger-McCurdy-1}. The importance of the net balls is that they are uniformly separated in each generation. That is, if $k\in\ZZ$, $x_1,x_2\in X_k$ are distinct points, and $Q_i=B(x_i,A_\mathscr{H}2^{-k})\in\mathscr{H}$, then \begin{equation}\label{net-gap} \gap((1/3A_\mathscr{H})Q_1,(1/3A_\mathscr{H})Q_2)\geq (1/3) 2^{-k}>0,\end{equation} where for any nonempty sets $S,T\subset\XX$, $\gap(S,T)=\inf_{s\in S, t\in T}|s-t|$ denotes the \emph{gap} between $S$ and $T$. (In harmonic analysis, the notation $\dist(S,T)$ may be more familiar.)\end{remark}

If $\epsilon_2$ is very small, then at the resolution of $2\lambda Q$, almost flat and $*$-almost flat arcs look like line segments.\footnote{At scales much smaller than $2\lambda Q$, almost flat and $*$-almost flat arcs can look like any rectifiable curve.} Roughly, the class $\mathscr{B}$ consists of all balls in the multiresolution family $\mathscr{H}$ such that $2\lambda Q$ contains at least two $*$-almost flat arcs (with distinct images) and the union of the images of arcs in $S^*(\lambda Q)$ is as non-flat as the union of the images of all arcs in $\Lambda(\lambda Q)$. Our main theorem says that for \emph{any} rectifiable curve in \emph{any} Banach space, the collection $\mathscr{B}$ of locations and scales with this behavior is rare relative to $\Haus^1(\Gamma)$.

\begin{theorem}[Main Theorem] \label{t:main} Assume that $\epsilon_2$ is sufficiently small depending only on $A_\mathscr{H}$ and $\epsilon_1$; $\epsilon_2=2^{-55}\epsilon_1/A_\mathscr{H}$ will suffice. For all $q>0$, \begin{equation}\label{B-sum} \sum_{Q\in\mathscr{B}} \beta_\Gamma(Q)^q\diam Q\lesssim_{q,A_\mathscr{H},\epsilon_1}\Haus^1(\Gamma),\end{equation} where the implicit constant blows up as $q\downarrow 0$.\end{theorem}

\begin{remark} A consequence of the \hyperref[t:main]{Main Theorem} is that in order to prove Theorem \ref{t:nec} for a particular curve $\Gamma$, in a particular Banach space $\XX$, and for a particular exponent $p$, it (essentially) suffices to prove $\sum_{Q\in\mathscr{A}} \beta_\Gamma(Q)^p\diam Q\lesssim_{p,A_\mathscr{H}} \Haus^1(\Gamma)$, where $Q\in\mathscr{A}\subset\mathscr{H}$ are balls whose \hyperref[netball]{net ball} $(1/3A_\mathscr{H})Q$ is intersected by a \hyperref[dom-arcs]{dominant arc}. (Besides $\mathscr{A}$ and $\mathscr{B}$, there is also a class of $\mathscr{C}$ balls; see \cite[\S3.3]{Badger-McCurdy-1} for details.) In Part I, we do this for curves in uniformly convex Banach spaces of power type $2\leq p<\infty$ and prove Theorem \ref{t:nec} assuming the \hyperref[t:main]{Main Theorem} (\cite[Theorem 3.30]{Badger-McCurdy-1}). \end{remark}

\begin{remark}\label{r:Schul-1} In \cite{Schul-Hilbert}, Schul gives a version of the \hyperref[t:main]{Main Theorem} in Hilbert space, but with some differences.
In particular, almost flat arcs $\tau=f|_{[a,b]}$ in \cite{Schul-Hilbert} satisfy the more stringent requirement \begin{equation}\label{e:arc-beta}\tilde\beta(\tau) = \sup_{c\in[a,b]}\frac{\dist(f(c),[f(a),f(b)])}{\diam \Image(\tau)}\leq \epsilon_2 \beta_{\Gamma}(Q).\end{equation} A geometric consequence is that $\tilde\beta$ almost flat arcs that pass near the center of $Q$ are ``diametrical'' in the sense that $\diam \Image(\tau)\cap Q\geq (1-O(\epsilon_2))\diam Q$. By contrast an almost flat arc with our definition that passes near the center of $Q$ may be ``radial'' in the sense that $\diam \Image(\tau)\cap Q \leq (1/2+O(\epsilon_2))\diam Q$. The existence of radial arcs causes substantial difficulties in the proof of the main theorem; see Remark \ref{remark:challenges}. For additional comments on the proof of the theorem in \cite{Schul-Hilbert}, see Appendix \ref{appendix328}.\end{remark}

\begin{remark}R.~Schul (personal communication) suggested an alternate approach to handling radial arcs. If one assumes $f$ is Lipschitz, then \eqref{B-sum} for the subset of all $Q\in\mathscr{B}$ that contain one or more radial arcs is subsumed by the Carelson-type estimates in Azzam and Schul's quantitative metric differentation theorem \cite{AS-metric}. Such an approach entails passing between multiresolution families in the domain and image of $f$, which is not needed in the direct argument below. The techniques in this paper may be better suited to proving a converse to the H\"older Traveling Salesman theorem \cite{BNV}.\end{remark}

We devote the remainder of the paper to the proof of the \hyperref[t:main]{Main Theorem}. The journey is somewhat long, but we try to make the first few sections as easy to read as possible. We hope that the reader who reaches the end may say that they have gained at least an incrementally better insight into the mysteries of Banach space geometry. Sections \ref{martingale}--\ref{proof-II} are best read in the order presented. In \S\ref{martingale}, we describe Schul's clever idea to prove \eqref{B-sum} by constructing geometric martingales out of curve fragments \cite{Schul-Hilbert}. We show how to modify the original argument to account for the possibility of ``radial'' arcs. An important quantity introduced is $\diam H_Q$, the diameter of a ``maximal arc fragment'' in the ``core'' $U_Q$ of a ball. In \S\ref{proof-main}, we outline the proof of the main theorem, including a discussion of the underlying challenges and a plan to overcome them. Ultimately, we reduce the proof of the main theorem to two key estimates, Lemma \ref{improve} and Lemma \ref{l:B_2 subset}. Section \ref{sec:prelim} sketches the geometry of possible configurations of almost flat arcs that we encounter in later proofs. Section \ref{sec: N'} takes a crucial step towards better estimates by identifying an auxiliary family of cores nearby a given arc fragment that possesses a sufficient amount of ``extra length.'' A vital technical tool is Lemma \ref{topological lemma}, which is proved using a topological argument.

We prove the main estimates in several stages. First, in Lemma \ref{43-estimate}, we show that \begin{equation}\label{intro-est} \diam H_Q\leq 2\Haus^1(\Gamma\cap R_Q)+1.37\sum \diam H_{Q'},\end{equation} where $R_Q$ is a ``remainder'' set and the sum ranges over all ``children'' $U_{Q'}$ of the core $U_Q$. While this is a substantial improvement of the coarse estimate \eqref{e:coarse estimate}, which holds with 2.01 instead of 1.37, to prove the main theorem we need the estimate to hold with the coefficient of the sum strictly less than 1! In the end, by a case analysis and iterating the proof of \eqref{intro-est}, using \eqref{intro-est} instead of \eqref{e:coarse estimate}, we obtain the key estimate with a coefficient less than 0.96. See \S\ref{proof-I} (proof of Lemma \ref{improve}) and \S\ref{proof-II} (proof of Lemma \ref{l:B_2 subset}) for details.

\subsection{Acknowledgements} 
This paper would not exist without the solid foundations of its predecessors, especially the body of work by P.~Jones, K.~Okikiolu, and R.~Schul.

Part I of this project began when M.~Badger hosted S.~McCurdy at UConn in Fall 2018. The extended visit was made possible by NSF DMS 1650546.

\section{Schul's martingale argument in a Banach space}\label{martingale}

We describe Schul's martingale argument (see \cite[\S3.3]{Schul-Hilbert}) for upper bounding sums of $\beta_\Gamma(Q)^q\diam Q$ over subfamilies of $\mathscr{B}$, where $q$ is any positive exponent! In this context, martingale refers to a recursively defined sequence of geometric weights associated to a tree of ``cores'' of overlapping balls. We formalize this terminology below. Schul's method is robust and can be implemented in any Banach space.

\subsection{Start of the proof: reduction to existence of weights}\label{ss:B-decomp}
Recall that if $Q\in\mathscr{B}$, say $Q=B(x,A_\mathscr{H}2^{-k})$ for some $k\in\ZZ$ and $x\in X_k$, then there exists $\lambda=\lambda(Q)\in\{1,5\}$ such that every arc $\tau\in\Lambda(\lambda Q)$ that intersects the net ball $B(x,(1/3)2^{-k})$ is almost flat, $\beta(\tau)\leq \flatepsilon\beta_\Gamma(Q)$, and hence is $*$-almost flat, $\beta(\tau)\leq 50\flatepsilon \beta_{\Lambda(\lambda Q)}(2\lambda Q)$. Moreover, the set $S^*(\lambda Q)$ of $*$-almost flat arcs in $\Lambda(\lambda Q)$ satisfies $\beta_{S^*(\lambda Q)}(2\lambda Q) > \flatarcsepsilon \beta_{\Lambda(\lambda Q)}(2\lambda Q)$.

Suppose that we have broken up $\mathscr{B}$ into a finite number of (possibly overlapping) families $\mathscr{B}(1),\dots,\mathscr{B}(N)$, where $N$ is independent of $\XX$ and $\lambda(Q)\equiv \lambda \in\{1,5\}$ is uniform across all $Q$ in any fixed family $\mathscr{B}(n)$. (The partition that we eventually use is described in \S\ref{proof-main}.) To prove the \hyperref[t:main]{Main Theorem}, in particular \eqref{B-sum}, it suffices to prove that for each $\mathscr{B}'=\mathscr{B}(n)$ and $q>0$, we have \begin{equation}\label{H2-goal} \sum_{Q\in\mathscr{B}'} \beta_{S^*(\lambda Q)}(2\lambda Q)^q\diam Q \lesssim_{q,A_\mathscr{H}} \mathscr{H}^1(\Gamma),\end{equation} because $\beta_\Gamma(Q)^q =\beta_{\Lambda(\lambda Q)}(Q)^q\lesssim_q \beta_{\Lambda(\lambda Q)}(2\lambda Q)^q\lesssim_{q,\epsilon_1} \beta_{S^*(\lambda Q)}(2\lambda Q)^q$ for all $Q\in\mathscr{B}$.

We now fix a family $\mathscr{B}'=\mathscr{B}(n)$ and describe a strategy to prove \eqref{H2-goal} for $\mathscr{B}'$. For the remainder of the paper, we set \begin{align}\label{K def}
    K:=100+\lceil \log_2 A_\mathscr{H}\rceil\geq 100.
\end{align} The value of $K$ is chosen according to certain geometric requirements below, but for now the reader may think of $K$ as being some large positive integer that is independent of the family $\mathscr{B}'$.
For the duration of the paper, for all integers $M\geq 1$ and $0\leq j\leq KM-1$, we let $\mathscr{G}^{M,j}$ denote the set of all $Q\in\mathscr{B}'$ such that \begin{itemize}
\item $Q=B(x,A_\mathscr{H}2^{-k})$ for some $k\equiv j\ (\Mod KM)$ and $x\in X_k$, and
\item $2^{-M}<\beta_{S^*(\lambda Q)}(2\lambda Q)\leq 2^{-(M-1)}$.
\end{itemize} Each $Q\in\mathscr{B}'$ belongs to precisely one of the families $\mathscr{G}^{M,j}$ for some integers $M\geq 1$ and $0\leq j<KM-1$. We will prove that when $\flatepsilon$ is sufficiently small compared to $\epsilon_1/A_\mathscr{H}$,   \begin{equation}\label{B1-goal} \sum_{Q\in\mathscr{G}^{M,j}} \diam Q \lesssim_{A_\mathscr{H}} \Haus^1(\Gamma)\end{equation} for all $M$ and $j$. (We only refer to $\epsilon_1$ two more times, once in \eqref{restrictions} and once in the derivation of  \eqref{star-almost-flat-line-estimate}.) This suffices, because for any $q>0$, \begin{align*} \sum_{Q\in\mathscr{B}'} \beta_{S^*(\lambda Q)}(2\lambda Q)^q\diam Q &\leq \sum_{M=1}^\infty 2^{-(M-1)q}\sum_{j=0}^{KM-1} \sum_{Q\in\mathscr{G}^{M,j}}\diam Q\lesssim_{q,A_\mathscr{H}} \Haus^1(\Gamma),\end{align*} where in the last inequality, we used $\sum_{M=1}^\infty M2^{-(M-1)q}\lesssim_q 1$ and $K\lesssim_{A_\mathscr{H}} 1$. That is, \eqref{B1-goal} for all $M$ and $j$ implies \eqref{H2-goal} holds for the family $\mathscr{B}'$.

We now fix integers $M\geq 1$ and $0\leq j_1\leq KM-1$ and write $\mathscr{G}=\mathscr{G}^{M,j_1}$. We make a further reduction. Suppose that for each ball $Q\in\mathscr{G}$ we possess a Borel measurable function $w_Q:\XX\rightarrow[0,\infty]$ which satisfies two properties: \begin{equation} \label{B-mass}
\int_\Gamma w_Q\,d\Haus^1\gtrsim_{A_\mathscr{H}} \diam Q\quad\text{for all }Q\in\mathscr{G};\end{equation}
\begin{equation}\label{B-bounded} \sum_{Q\in\mathscr{G}} w_Q(x)\lesssim 1\quad\text{at $\Haus^1$-a.e. }x\in\Gamma.\end{equation} Then $$\sum_{Q\in\mathscr{G}} \diam Q \lesssim_{A_\mathscr{H}} \sum_{Q\in\mathscr{G}} \int_\Gamma w_Q\,d\Haus^1
= \int_\Gamma \sum_{Q\in\mathscr{G}} w_Q\,d\Haus^1 \lesssim_{A_\mathscr{H}} \Haus^1(\Gamma).$$ That is, the existence of weights $w_Q$ satisfying \eqref{B-mass} and \eqref{B-bounded} imply \eqref{B1-goal}. Our task is to construct the weights, assuming that $\flatepsilon$ is sufficiently small.

\subsection{Cores and maximal almost flat arcs}\label{ss:cores-def}

Following \cite{Schul-Hilbert}, it will be convenient to introduce a \emph{nested} family of ``cores'' $U_{Q}^{J,c}$, lying near the center of balls $Q\in \mathscr{H}$. Cores are formed by joining overlapping dilations of balls in $\mathscr{H}$ from future generations, skipping by $J$ generations at a time.

\begin{definition}[{\cite{Schul-Hilbert}}] \label{def-general-cores} Let $Q\in\mathscr{H}$, say that $Q=B(x,A_\mathscr{H}2^{-k})$ for some $k\in\ZZ$ and $x\in X_k$. For any integer $J\geq 4$ and $0<c\leq 1/5$, we define the $(J,c)$-\emph{core} $U^{J,c}_Q$ of $Q$ inductively by setting $U^{J,c}_{Q,0}:= B(x,c2^{-k})=(c/A_\mathscr{H})Q,$
$$U^{J,c}_{Q,i} := U^{J,c}_{Q,i-1}\cup \bigcup_{\mathclap{\substack{y\in X_{k+Jj}\text{for some }j\geq 1\\ B(y,c 2^{-(k+Jj)})\cap U^{J,c}_{Q,i-1}\neq\emptyset}}} B(y,c 2^{-(k+Jj)})\quad\text{for all $i\geq 1$, and\quad} U^{J,c}_Q := \bigcup_{i=0}^\infty U^{J,c}_{Q,i}.$$
\end{definition}

Cores are a variation on the Christ-David ``dyadic cubes'' in a doubling metric space. Although an infinite-dimensional Banach space is not a doubling metric space, we note that the nets $X_k$ are finite because $\Gamma$ is compact. For a streamlined construction of metric cubes that starts with any nested family of locally finite nets, see \cite{KRS-cubes}.

\begin{lemma}[properties of cores, cf.~{\cite[Lemma 3.19]{Schul-Hilbert}}] \label{core-family} Given $J\geq 4$, $0<c\leq 1/5$, and $0\leq j\leq J-1$, let $\mathscr{U}$ be the family of cores defined by $$\mathscr{U}=\{U^{J,c}_Q:Q=B(x,A_\mathscr{H}2^{-k})\text{ for some } x\in X_k\text{ and }k\equiv j\ (\Mod J)\}.$$ If $Q,R\in\mathscr{H}$ with $Q=B(x,A_\mathscr{H}2^{-k})$ and $R=B(y,A_\mathscr{H}2^{-m})$ for some $k,m\equiv j\ (\Mod J)$, $x\in X_k$, and $y\in X_m$, then the cores $U^{J,c}_Q$ and $U^{J,c}_R$ belong to the family $\mathscr{U}$ and satisfy: \begin{enumerate}
\item[(i)] Shape: $B(x,c2^{-k})\subset U^{J,c}_Q \subset B(x, (1+3/2^J)c2^{-k})\subset B(x,(1/4)2^{-k})$.
\item[(ii)] Separation within levels: If $k=m$ and $x\neq y$, then $\gap(U^{J,c}_Q,U^{J,c}_R) \geq (1/2)2^{-k}.$
\item[(iii)] Tree structure: If $m\geq k$ and $U^{J,c}_Q\cap U^{J,c}_R\neq\emptyset$, then $U^{J,c}_R\subset U^{J,c}_Q$.
\end{enumerate}\end{lemma}

\begin{proof} For (i), given $Q=B(x,A_\mathscr{H}2^{-k})$, the first containment is immediate, because $B(x,c2^{-k})=U^{J,c}_{Q,0}\subset U^{J,c}_Q$. For the second containment, $U^{J,c}_Q\subset B(x,(1+3/2^J)c2^{-k})$, apply Lemma \ref{union-ball-chains} with parameters $\xi=2^J$ and $r_0=c2^{-k}$ and balls $B(y,c2^{-(k+Jj)})$ appearing in the definition of $U^{J,c}_Q$ assigned to level $j$. The reader should check that the hypotheses of Lemma \ref{union-ball-chains} are satisfied, but here are the essential points: With $J\geq 4$, the parameter $\xi\geq 16>6$. The chain hypothesis is satisfied by the construction of the cores. The separation hypothesis is satisfied, because the centers of balls in level $j$ are $2^{-(k+Jj)}$ separated and $(1-2c)\geq 3c$. The final containment in (i) holds, since $(1+3/2^J)c\leq 19/80<1/4$ when $J\geq 4$ and $c\leq 1/5$. Property (ii) holds by property (i) and fact that $|x-y|\geq 2^{-k}$ when $x,y \in X_k$ are distinct. When $m=k$, property (iii) is immediate from property (ii). Finally, when $m>k$, property (iii) follows from the construction. Indeed, $U_Q^{J,c}\cap U_R^{J,c}\neq\emptyset$ only if $U_{Q,i}^{J,c}\cap U_{R,j}^{J,c}\neq\emptyset$ for some $i,j$, so $U_{R,j+l}^{J,c}\subset U_{Q,i+j+1+l}^{J,c}$ for all $l\geq 0$, since $m>k$.\end{proof}

\begin{definition} For all $Q\in\mathscr{H}$ (in particular, for $Q\in\mathscr{G}$), let $U_Q$ denote the $(J,c)$-core $U^{J,c}_Q$ with parameters $J=KM$ and $c=2^{-12}$, with $K$ as in \eqref{K def}, and let $Q_*$ denote $U^{J,c}_{Q,0}$.\end{definition}

\begin{remark}A core $U_Q$ looks like the ball $Q_*$, except that it may have ``tiny bubbles'' pushing outward near the boundary $\partial Q_*$ of the ball. Cores are not necessarily convex.\end{remark}

\begin{remark}If $Q=B(x,A_\mathscr{H}2^{-k})\in\mathscr{H}$ for some $k\in\ZZ$ and $x\in X_k$, then \begin{equation} \label{shape-of-B1-cores} Q_*=B(x,2^{-12}\,2^{-k})\subset U_Q\subset 1.00001 Q_*\end{equation} by Lemma \ref{core-family}, since $1+3/2^{KM}\leq 1+3/2^{100}<1.00001$. (Fifth decimal place precision is chosen to facilitate select estimates in \S\S \ref{proof-main}--\ref{proof-II}.) If $Q\in\mathscr{G}$ and $Q'=B(y,A_\mathscr{H}2^{-m})\in\mathscr{G}$ for some $m\equiv k\ (\Mod KM)$ with $m>k$, then \begin{equation}\label{Q'-diam} \diam 2\lambda Q' \leq 20A_\mathscr{H}2^{-m} \leq 32A_\mathscr{H}2^{-KM}2^{-k}\leq 2^{-84}\diam Q_*,\end{equation} since $2^{-KM}\leq 2^{-100} A_\mathscr{H}^{-1}$, $32A_\mathscr{H}2^{-KM}\leq 2^{-95}$, and $\diam Q_*=2^{-11}2^{-k}$. In particular,  \begin{equation}\label{Q'-containment} 2\lambda Q'\cap 0.99999 Q_*\neq\emptyset \Longrightarrow 2\lambda Q'\subset Q_*\subset U_Q.\end{equation}\end{remark}

\begin{remark}\label{core-separation} The core $U_Q$ of a ball $Q\in\mathscr{H}$ is much smaller than the \hyperref[netball]{net ball} $(1/3A_\mathscr{H})Q$: $2^{10}U_Q\subset (1/3A_\mathscr{H}) Q$, where dilations are relative to the center of $Q$. When $Q'\in\mathscr{H}\setminus\{Q\}$ and $\diam Q'=\diam Q$, Lemma \ref{core-family}(ii) implies $\gap(U_Q,U_{Q'})\geq 2^{10}\diam Q_*\geq 2^9\diam U_Q$. 
\end{remark}

\begin{remark}[tree structure] \label{r:G is a tree} By Lemma \ref{core-family}, we may view $\mathscr{G}$ as a tree ordered by inclusion of the cores $\{U_Q:Q\in\mathscr{G}\}$. That is, we declare $P\in\mathscr{G}$ to be the \emph{parent} of $Q\in\mathscr{G}$ if and only if $P$ is the unique element such that $U_Q\subsetneq U_P$ and $U_Q\subset U_{R}\subset U_P$ for some $R\in\mathscr{G}$ implies $R\in\{P,Q\}$. Note that \begin{equation}\label{max-scale}\sup_{Q\in\mathscr{B}} \diam Q\leq (1/14)\diam \Gamma<\infty,\end{equation} since $\Gamma\setminus 14Q\neq\emptyset$ for all $Q\in\mathscr{B}$ and $\Gamma$ is a rectifiable curve. Hence every element of $\mathscr{G}$ sits below a maximal element in $\mathscr{G}$, i.e.~a ball without a parent. Extending the metaphor, we say that $Q$ is a \emph{child} of $P$ if $P$ is the parent of $Q$. We let $\Child(P)$ denote the set of all $Q\in\mathscr{G}$ such that $Q$ is a child of $P$. For each ball $P\in\mathscr{G}$, the set $\Child(P)$ may be empty, nonempty and finite, or countably infinite.
We also view $\{U_Q:Q\in\mathscr{G}\}$ as a tree ordered by inclusion and call $U_{Q'}$ a \emph{child} of $U_{Q}$ if and only if $Q'\in\Child(Q)$. A child is a \emph{1st generation descendent}, a child of a child is a \emph{2nd generation} descendent, etc.
\end{remark}

We now diverge slightly from \cite{Schul-Hilbert} and introduce (possibly disconnected) fragments of $*$-almost flat arcs on the image side of $f$. We also define a class of closed, connected subsets of fragments called subarcs.

\begin{definition}[fragments of $*$-almost flat arcs] \label{def:arc-fragments} For each $Q\in\mathscr{G}$ and nonempty set $W\subset 2\lambda Q$, with $\lambda\in\{1,5\}$ determined by $\mathscr{G}$, let $\Gamma^*_W=\{\Image(\tau)\cap W:\tau\in S^*(\lambda Q)\}\setminus\{\emptyset\}$.\end{definition}

\begin{definition}[subarcs] \label{def:subarc} Let $T'\in\Gamma^*_W$, say $T'=\Image(\tau)\cap W$ for some arc $\tau\in S^*(\lambda Q)$. We say that $T\subset T'$ is a \emph{subarc} of $T'$ if $T=\tau(I)=f(I)$ for some non-degenerate interval $I=[a,b]\subset\Domain(\tau)$. We say that $T$ is \emph{efficient} if, in addition, $\diam T=|f(a)-f(b)|$.\end{definition}

\begin{remark}[choosing maximal arc fragments] \label{r:maximal-fragments}
 For any $Q\in\mathscr{G}$, say $Q=B(x,A_\mathscr{H}2^{-k})$ for some $k\in\ZZ$ and $x\in X_k$, the set $\Gamma^*_{U_Q}$ of arc fragments is nonempty, since the core $U_Q$ is contained in the net ball $B(x,(1/3)2^{-k})$ and $x\in U_Q$. In fact, for every set $T'\in\Gamma^*_{U_Q}$, there exists an almost flat arc $\tau\in S(\lambda Q)$ such that $T'=\Image(\tau)\cap U_Q$, since $Q\in\mathscr{G}$ and $\mathscr{G}\subset\mathscr{B}^\lambda$ (see Definition \ref{B-balls-def}).

Among all sets in $\Gamma_{U_Q}^*$, choose $H_Q\in\Gamma^*_{U_Q}$ such that \begin{equation}\begin{split}\label{e:HQ}
&H_Q \cap (1/4)Q_* \neq \emptyset\quad\,\text{and}  \\
&\diam H_Q \geq \diam T' \quad\text{for all }T'\in\Gamma^*_{U_Q} \text{ such that } T' \cap (1/4)Q_* \neq \emptyset.\end{split}\end{equation} That is, let $H_Q$ have maximal diameter among all fragments in $U_Q$ of almost flat arcs that intersect $(1/4)Q_*$. Let $\eta_Q\in S(\lambda Q)$ denote any arc such that $H_Q=\Image(\eta_Q)\cap U_Q$. Existence of $H_Q$ is immediate, as $\Gamma^*_{U_Q}$ is a nonempty finite set and at least one fragment in $\Gamma^*_{U_Q}$ passes through the center of $(1/4)Q_*$. If there are several candidates, pick one in an arbitrary fashion. In principle, $H_Q$ may have several connected components; e.g.~even if $\eta_Q$ traces a line segment, the core $U_Q$ need not be a convex set. Nevertheless, $H_Q$ always contains an efficient subarc $G_Q$ with diameter nearly equal to that of $H_Q$; see \eqref{H-G diam est} below. By comparison with an arc $\tau\in S(\lambda Q)$ with $x\in \Image(\tau)$ and \eqref{shape-of-B1-cores}, \begin{equation}\label{HQ-bounds} 0.5\diam Q_*\leq \diam H_Q \leq 1.00001\diam Q_*<3\Haus^1(\Gamma\cap U_Q),\end{equation} where the diameter of $H_Q$ is closer to the lower bound if $H_Q$ is ``radial" and closer to the upper bound if $H_Q$ is ``diametrical". (The constant 3 is overkill.) Alternatively, \begin{equation} \label{HQ-bounds-2} 0.49999 \diam U_Q \leq \diam H_Q \leq \diam U_Q\leq 2.00002\diam H_Q.\end{equation}

Below, we use $\diam H_Q$ to the play the role that $\diam U_Q$ had in \cite{Schul-Hilbert}.
\end{remark}

\subsection{Martingale construction}\label{ss:martingale}
In probability theory \cite[Chapter 4]{Durrett}, a \emph{martingale} defined with respect to an increasing sequence $(\mathcal{F}_k)_{k\geq 0}$ of $\sigma$-algebras is any sequence of real-valued random variables $(Y_k)_{k\geq 0}$ such that each $Y_k$ is $\mathcal{F}_k$ measurable and has finite expectation, and moreover, the conditional expectations $\mathbb{E}(Y_{k+1}|\mathcal{F}_k)=Y_k$ for all $k$. The \emph{martingale convergence theorem} asserts that if $(Y_k)_{k\geq 0}$ is a martingale and $Y_k\geq 0$ for all $k$, then $Y_k$ converges to some random variable $Y$ almost surely. We will use martingales to construct weights satisfying \eqref{B-mass} and \eqref{B-bounded}, where the background ``probability'' is the finite measure $\ell=\Haus^1\res\Gamma$.

Let $P\in\mathscr{G}$ be a fixed ball. For each $k\geq 0$, let $\mathcal{F}_k$ denote the $\sigma$-algebra generated by the cores $U_{Q}$, where $Q$ is a descendent of $P$ in $\mathscr{G}$ of generation at most $k$ (including $P$). Thus, $\mathcal{F}_0=\{\emptyset,U_P,\XX\setminus U_P,\XX\}$ is the $\sigma$-algebra generated by $\{U_P\}$, $\mathcal{F}_1$ is the $\sigma$-algebra generated by $\{U_P\}\cup\{U_Q:Q\in\Child(P)\}$, etc. We remark that $\mathcal{F}_0\subset\mathcal{F}_1\subset\mathcal{F}_2\subset \cdots\subset\mathcal{B}_\XX$, the Borel $\sigma$-algebra.
We build $(Y_k)_{k\geq 0}$ inductively. First, assign $Y_0$ to be the $\mathcal{F}_0$ simple function \begin{equation}\label{X-0} Y_0=\frac{\diam H_P}{\ell(U_P)}\chi_{U_P},\end{equation} where $H_P$ denotes the maximal arc fragment chosen in Remark \ref{r:maximal-fragments}. Note that $Y_0$ is $\mathcal{F}_0$ measurable and $\int Y_0\,d\ell = \diam H_P$. To continue, suppose that $Q\in\mathscr{G}$ with $U_Q\subset U_P$. Let $k\geq 0$ denote the unique integer such that $Q$ is a descendent of $P$ of generation $k$, i.e.~$k=0$ if $Q=P$, $k=1$ if $Q\in\Child(P)$, etc. We will define $Y_{k+1}|_{U_Q}$ to take constant values on elements of $\mathcal{F}_{k+1}$ contained in $U_Q$. If $\Child(Q)=\emptyset$, then $Q$ is terminal in $\mathscr{G}$ and we simply set $Y_{k+i}|_{U_Q}=Y_k|_{U_Q}$ for all $i\geq 1$. Otherwise, $Q$ has at least one and possibly $\aleph_0$ many children in $\mathscr{G}$; let  $Q^1,Q^2,\dots$ be an enumeration of $\Child(Q)$. We remark that the cores $U_{Q^i}$ of children of $Q$ are pairwise disjoint. Now, define the \emph{remainder} $R_Q$, \begin{equation}\label{remainder} R_Q:=U_Q \setminus \textstyle\bigcup_i U_{Q^i},\end{equation} and define the auxiliary quantity $s_Q$, \begin{equation}\label{s-Q} s_Q:=101\,\ell(R_Q)+\textstyle\sum_i \diam H_{Q^i}.\end{equation} Observe that $s_Q\leq 101\,\ell(U_Q)<\infty$ by \eqref{HQ-bounds} and countable additivity of measures.  Assign $Y_{k+1}|_{U_Q}$ to be the function \begin{equation}\label{X-next} Y_{k+1}|_{U_Q}= \Big(\frac{101}{s_Q}\chi_{R_Q} + \sum_{i} \frac{\diam H_{Q^i}}{\ell(U_{Q^i})s_Q}\chi_{Q^i}\Big)\int_{U_Q} Y_k\,d\ell;\end{equation} also assign $Y_{k+i}|_{R_Q}=Y_{k+1}|_{R_Q}$ for all $i\geq 2$. Then $Y_{k+1}|_{U_Q}$ is $\mathcal{F}_{k+1}$ measurable and $\int_{U_Q} Y_{k+1}\,d\ell=\int_{U_Q} Y_k\,d\ell$. As $U_Q$ is an atom in the $\sigma$-algebra $\mathcal{F}_k$, the equality of the integrals ensures that $\mathbb{E}(Y_{k+1}|\mathcal{F}_k)= Y_k$ on $U_Q$. Repeating this construction on each $Q$ that sits $k$ levels below $P$ in $\mathscr{G}$ concludes the description of $Y_{k+1}$ given $Y_k$. We have verified that $Y_{k+1}$ is $\mathcal{F}_{k+1}$ measurable and $\mathbb{E}(Y_{k+1}|\mathcal{F}_k)=Y_k$. Furthermore, the function $Y_{k+1}$ has finite expectation, since $\int Y_{k+1}\,d\ell=\int Y_k\,d\ell=\cdots=\int Y_0\,d\ell=\diam H_P<\infty$. Therefore, $(Y_k)_{k\geq 0}$ is a martingale relative to $(\mathcal{F}_k)_{k\geq 0}$. By the martingale convergence theorem, $(Y_k)_{k\geq 0}$ converges almost surely. Thus, we may define the weight $w_P$ to be any non-negative Borel measurable function such that $w_P=\lim_{k\rightarrow\infty} Y_k$ $\ell$-a.e.

The following observation is the key to unlocking \eqref{B-mass} and \eqref{B-bounded}.

\begin{lemma}[cf.~{\cite[Lemma 3.25, Steps 2--3]{Schul-Hilbert}}] \label{q-lemma} Suppose there is a universal constant $0<q<1$ such that $\diam H_Q \leq q\, s_Q$ for all $Q\in\mathscr{G}$. Then \eqref{B-mass} and \eqref{B-bounded} hold for $\mathscr{G}$. \end{lemma}

\begin{proof}
Suppose that $Q_0=P$, $Q_1\in\Child(Q_0)$, \dots, $Q_k\in\Child(Q_{k-1})$ is a finite branch of $\mathscr{G}$ below $P$. Then, for all $x\in U_{Q_k}$,
\begin{align*}Y_k(x) &= \frac{\diam H_{Q_k}}{\ell(U_{Q_k})s_{Q_{k-1}}}\int_{U_{Q_k}}Y_{k-1}\,d\ell=\frac{\diam H_{Q_k}}{\ell(U_{Q_k})s_{Q_{k-1}}}\frac{\diam H_{Q_{k-1}}}{s_{Q_{k-2}}}\int_{U_{Q_{k-1}}} Y_{k-2}\,d\ell\\
&=\cdots =\frac{\diam H_{Q_k}}{\ell(U_{Q_k})s_{Q_{k-1}}}\frac{\diam H_{Q_{k-1}}}{s_{Q_{k-2}}}\cdots \frac{\diam H_{Q_1}}{s_{P}}\int_{U_P} \frac{\diam H_P}{\ell(U_P)}\,d\ell
\\ &\leq q^k \frac{\diam H_{Q_k}}{\ell(U_{Q_k})} < 3q^k
\end{align*} by the hypothesis of the lemma and \eqref{HQ-bounds}. Similarly, for all $x\in R_{Q_k}$, $$Y_{k+i}(x)=Y_{k+1}(x) = \frac{101}{s_{Q_k}}\int_{U_{Q_k}} Y_k\,d\ell\leq 101 q^{k+1}\quad\text{for all }i\geq 2.$$ Now, every point $x\in U_P$ either belongs to some $R_Q$ and $Y_k(x)$ is eventually constant, or $x$ is contained in an infinite branch of $\mathscr{G}$ and $Y_k(x)\rightarrow 0$. Hence \begin{equation}\begin{split} \label{wP-bound} Y_k(x) &\leq 101\quad\quad \text{for all $x\in\XX$ and $k\geq 0$, and}\\ w_P(x)&\leq 101 q^k\quad \text{whenever $x$ belongs to a branch }U_{Q_k}\subset U_{Q_{k-1}}\subset\cdots\subset U_P.\end{split}\end{equation} Because $Y_k\rightarrow w_P$ $\ell$-a.e.~and $Y_k$ is uniformly bounded, $Y_k\rightarrow w_P$ in $L^1(\ell)$ by Lebesgue's dominated convergence theorem. Thus, $$\int_\Gamma w_P\,d\Haus^1=\int w_P\,d\ell = \lim_{k\rightarrow\infty} \int Y_k\,d\ell = \diam H_P \gtrsim_{A_\mathscr{H}} \diam P$$ by \eqref{HQ-bounds}. That is, \eqref{B-mass} holds.

Finally, if some ball $Q_0\in\mathscr{G}$ is maximal in $\mathscr{G}$ (i.e.~$Q_0$ has no parent in $\mathscr{G}$) and for some branch $Q_1\in\Child(Q_0)$, \dots, $Q_k\in \Child(Q_{k-1})$ of $\mathscr{G}$ below $Q_0$, a point $x\in U_{Q_k}$, then $$w_{Q_0}(x)+w_{Q_1}(x)+\dots + w_{Q_k}(x) \leq 101q^k+ 101q^{k-1}+\cdots+101 \leq \frac{101}{1-q}$$ by \eqref{wP-bound}. Since the upper bound is independent of the length of the branch and $q$ is a universal constant, this yields \eqref{B-bounded}. \end{proof}

\subsection{Summary}\label{ss:summary} All things considered, we have shown that in order to prove \eqref{H2-goal} for a given family $\mathscr{B}' \subset\mathscr{B}$, it suffices to verify the hypothesis of Lemma \ref{q-lemma} for each subfamily $\mathscr{G}=\mathscr{G}^{M,j_1}$ associated to $\mathscr{B}'$. (Look between \eqref{K def} and \eqref{B1-goal} for the definition of $\mathscr{G}$.)

\section{Outline of the proof of the \texorpdfstring{\hyperref[t:main]{Main Theorem}}{Main Theorem}}\label{proof-main}
Recall that $\mathscr{B}=\mathscr{B}^1\cup\mathscr{B}^5$. Some balls in $\mathscr{B}$ may belong to both families, but this will not concern us. For the remainder of the paper, we let $\lambda\in\{1,5\}$ be fixed and focus on establishing \eqref{H2-goal} for $\mathscr{B}'=\mathscr{B}^\lambda$. Throughout the sequel, we demand that \begin{equation}\label{restrictions} \flatepsilon\leq 2^{-55}\epsilon_1/A_\mathscr{H},\end{equation} which ensures that at appropriate resolutions, every point in the image of an almost flat arc lies close to some line segment. Furthermore, this choice guarantees that any individual $*$-almost flat arc $\tau\in S^*(\lambda Q)$ is much flatter than the union of the images of all $*$-almost flat arcs in $S^*(\lambda Q)$. See \S\ref{ss:basic} for details. We do not optimize $\epsilon_2$.

\begin{remark}If desired, one can replace the scaling factor $\lambda\in\{1,5\}$ in the arguments below with an arbitrary scaling factor $\lambda\geq 1$. However, if $\lambda$ is very large, then one must adjust the values of several parameters, including $\epsilon_2$ in the definition of almost flat arcs, and $J$ and $c$ in the definition of the cores $U_Q=U_Q^{J,c}$. We restrict to $\lambda\in\{1,5\}$, because these are the values needed for the proof of Theorem \ref{t:nec} presented in \cite{Badger-McCurdy-1}.\end{remark}

Later on, we would like to assume that every almost flat arc $\tau\in S(\lambda Q)$ that passes through the \hyperref[netball]{net ball} for $Q$ has endpoints on the boundary of $2\lambda Q$. Exceptions may occur if an endpoint of the full parameterization lies on the arc, but for each endpoint this happens at most a finite number of times per scale. Checking \eqref{H2-goal} for such balls is easy.

\begin{lemma}\label{l:B0} Let $\mathscr{B}^\lambda_0$ denote the set of all $Q\in\mathscr{B}^\lambda$ for which there exists an arc $\tau\in S(\lambda Q)$ such that $\Image(\tau)$ contains $f(0)$ or $f(1)$ and $\Image(\tau)\cap (1/3A_\mathscr{H})Q\neq\emptyset$. For all $q>0$, \begin{equation}\label{B0-sum}
\sum_{Q\in \mathscr{B}^\lambda_0} \beta_{S^*(\lambda Q)}(2\lambda Q)^q \diam Q \leq \sum_{Q\in\mathscr{B}^\lambda_0} \diam Q \lesssim_{A_\mathscr{H}} \Haus^1(\Gamma).\end{equation}
\end{lemma}

\begin{proof} Fix any $z\in\XX$ (e.g.~$z=f(0),f(1)$). For the duration of the proof, let $\mathscr{B}^\lambda_z$ denote the set of all $Q\in\mathscr{B}^\lambda$ for which there exists an arc $\tau\in S(\lambda Q)$ such that $\Image(\tau)$ contains $z$ and intersects the net ball $(1/3A_\mathscr{H})Q$. Choose $k_0\in\ZZ$ so that $A_\mathscr{H} 2^{-k_0}$ is the largest radius of a ball in $\mathscr{B}^\lambda_z$.  For each $k\geq k_0$, let $\mathscr{E}_k$ denote all balls $Q\in\mathscr{B}^\lambda_z$ of radius $A_\mathscr{H} 2^{-k}$. Choose $v_Q\in \Image(\tau)\cap (1/3A_\mathscr{H})Q$ for each $Q\in\mathscr{E}_k$. By \eqref{net-gap},  \eqref{almost-flat-line-estimate}, Lemma \ref{j-proj facts}, and Lemma \ref{l:count}, the set $\{v_Q: Q\in\mathscr{E}_k\} \cap B(z,4\lambda A_\mathscr{H} 2^{-k})$ has cardinality at most $1+36\lambda A_\mathscr{H}$. Thus, by \eqref{max-scale}, $\sum_{k=k_0}^\infty \sum_{Q\in\mathscr{E}_k} \diam Q \leq 2(1+36\lambda A_\mathscr{H})(1/14)\diam \Gamma\lesssim_{A_\mathscr{H}} \Haus^1(\Gamma)$.
\end{proof}

Our strategy to prove \eqref{H2-goal} for $\mathscr{B}'=\mathscr{B}^\lambda\setminus \mathscr{B}^\lambda_0$ is to run Schul's martingale argument. That is to say, we must verify that the hypothesis of Lemma \ref{q-lemma} holds for all $Q\in\mathscr{G}$, for  each possible subfamily $\mathscr{G}=\mathscr{G}^{M,j_1}\subset\mathscr{B}^\lambda\setminus \mathscr{B}^\lambda_0$: \begin{equation}\label{q-goal} \exists_{0<q<1}\,\forall_M\,\forall_{j_1}\,\forall_{Q\in\mathscr{G}}\quad \diam H_Q \leq q\, s_Q,\end{equation} where the maximal arc fragment $H_Q$ associated to $Q$ was chosen in Remark \ref{r:maximal-fragments} and \begin{equation}\label{s-Q-again} s_Q=101\,\ell(R_Q)+\sum_{Q'\in\Child(Q)} \diam H_{Q'}.\end{equation} There will be a number of cases, depending on the geometry of arc fragments in $U_Q$ as well as on the geometry of arcs associated to $Q'\in\Child(Q)$, the children of $Q$ in the tree $\mathscr{G}$ (see Remark \ref{r:G is a tree}), and the size of the remainder $R_Q$ \eqref{remainder}. Let us quickly dispense with an easy case, which is connected to the choice of the constant 101 in \eqref{s-Q-again}.

\begin{definition}\label{remainder-sizes} Let $Q\in \mathscr{G}$. \begin{itemize}
\item  If $\ell(R_Q)>(1/100)\diam H_Q$, then we say that the remainder of $Q$ is \emph{large}.
\item If $\ell(R_Q)\leq (1/100) \diam H_Q$, then we say the remainder of $Q$ is \emph{small}.
\end{itemize}\end{definition}

\begin{lemma}[Case 1: large remainder]\label{case1} If $Q\in\mathscr{G}$ has large remainder, then $\diam H_Q < 0.9901 s_Q.$
\end{lemma}

\begin{proof} By \eqref{s-Q-again} and  definition of large remainder, $\diam H_Q <100 \ell(R_Q)\leq (100/101)s_Q$ and $100/101=0.\overline{9900}<0.9901$.\end{proof}

Case 1 occurs if, for example, $Q$ has no children in $\mathscr{G}$. Having dealt with Case 1, we may now make a standing assumption that any $Q\in\mathscr{G}$ that we examine has small remainder. At a minimum, this assumption ensures that $\Child(Q)\neq\emptyset$. In fact, the picture that the reader should keep in mind is that $H_Q$ (imagine a line segment through the center of $U_Q$) is intersected by many disjoint cores $U_{Q'}$ with $Q'\in\Child(Q)$. We emphasize that $\Child(Q)$ may be finite or infinite and $\diam U_{Q'}$ can be arbitrarily small relative to $\diam U_Q$.

\begin{remark}[challenges] \label{remark:challenges}Broadly speaking, there are two challenges to verifying \eqref{q-goal} for $Q\in\mathscr{G}$ with small remainder. First, as we previously noted in Remark \ref{r:maximal-fragments}, each fragment $H_Q$ may be disconnected. In principle, it is possible that
\begin{equation}\label{e:bad-if-disconnected}
\diam H_Q > \ell(R_Q \cap H_Q) + \sum_{U_{Q'} \cap H_Q \neq \emptyset} \diam U_{Q'}.
\end{equation} Thus, to verify $\diam H_Q \leq q\, s_Q$, we must locate additional cores $U_{Q'}$ with $Q'\in\Child(Q)$ that do not intersect $H_Q$. In \eqref{e:bad-if-disconnected} and throughout the sequel, when we write $Q'$ inside the subscript position of a summation or union, we implicitly mean that, in addition to any other restrictions, $Q'$ ranges over all $Q'\in\Child(Q)$, with $Q$ fixed nearby.

Secondly and more seriously, $\diam U_{Q'} \geq \diam H_{Q'}$ for all children, but $\diam H_{Q'}$ could be significantly smaller than $\diam U_{Q'}$ if $H_{Q'}$ is ``radial".  For any closed, connected set $T \subset T' \in \Gamma^*_{U_Q}$, the diameter bound \eqref{HQ-bounds-2}
only leads to the \emph{coarse estimate}
    \begin{align}\label{e:coarse estimate}
        \diam T & \leq \ell(R_Q \cap T) + 2.00002\sum_{U_{Q'} \cap T \neq \emptyset}\diam H_{Q'}.
    \end{align}
This implies $\diam T \le 2.00002 s_Q$, which is insufficient to verify \eqref{q-goal} as the coefficient $2.00002\geq 1$. See Lemma \ref{proof-of-coarse} for a proof of \eqref{e:coarse estimate}.\end{remark}

To sidestep the first challenge in Remark \ref{remark:challenges} and avoid complications near the boundary, we narrow our focus to a smaller region inside of $U_Q$ and to an \hyperref[def:subarc]{efficient subarc} $G_Q\subset H_Q$.

\begin{remark}[choosing $G_Q$] \label{proj-lemma} For each $Q\in\mathscr{G}$, we may invoke Lemma \ref{proof:proj-lemma} with $T'=H_Q$ to choose $I_Q=[a_Q,b_Q]\subset\Domain(\eta_Q)$ such that $G_Q:=f(I_Q)\subset H_Q\cap 0.99999Q_*$ and
\begin{equation}\label{H-G diam est}|f(a_Q)-f(b_Q)|=\diam G_Q>0.99993\diam H_Q.\end{equation}
\end{remark}

Overcoming the second challenge is more complicated. We need to account for length in $R_Q$ and cores $U_{Q'}$ appearing in a neighborhood of $T=G_{Q}$ that do not necessarily intersect $G_Q$. Ultimately, the reason that we can improve upon \eqref{e:coarse estimate} is because we can find a sufficient amount of ``extra length" nearby $G_Q$.  Roughly speaking, for each $U_{Q'}$ intersecting $G_Q$, there exist at least two $*$-almost flat arcs in $2\lambda Q'$ that intersect $\lambda Q'$. To describe improved estimates for balls with small remainder, we need to introduce a classification of cores $U_{Q'}$ of $Q'\in\Child(Q)$ involving projections onto lines.

\begin{figure}\begin{center}\includegraphics[width=.57\textwidth]{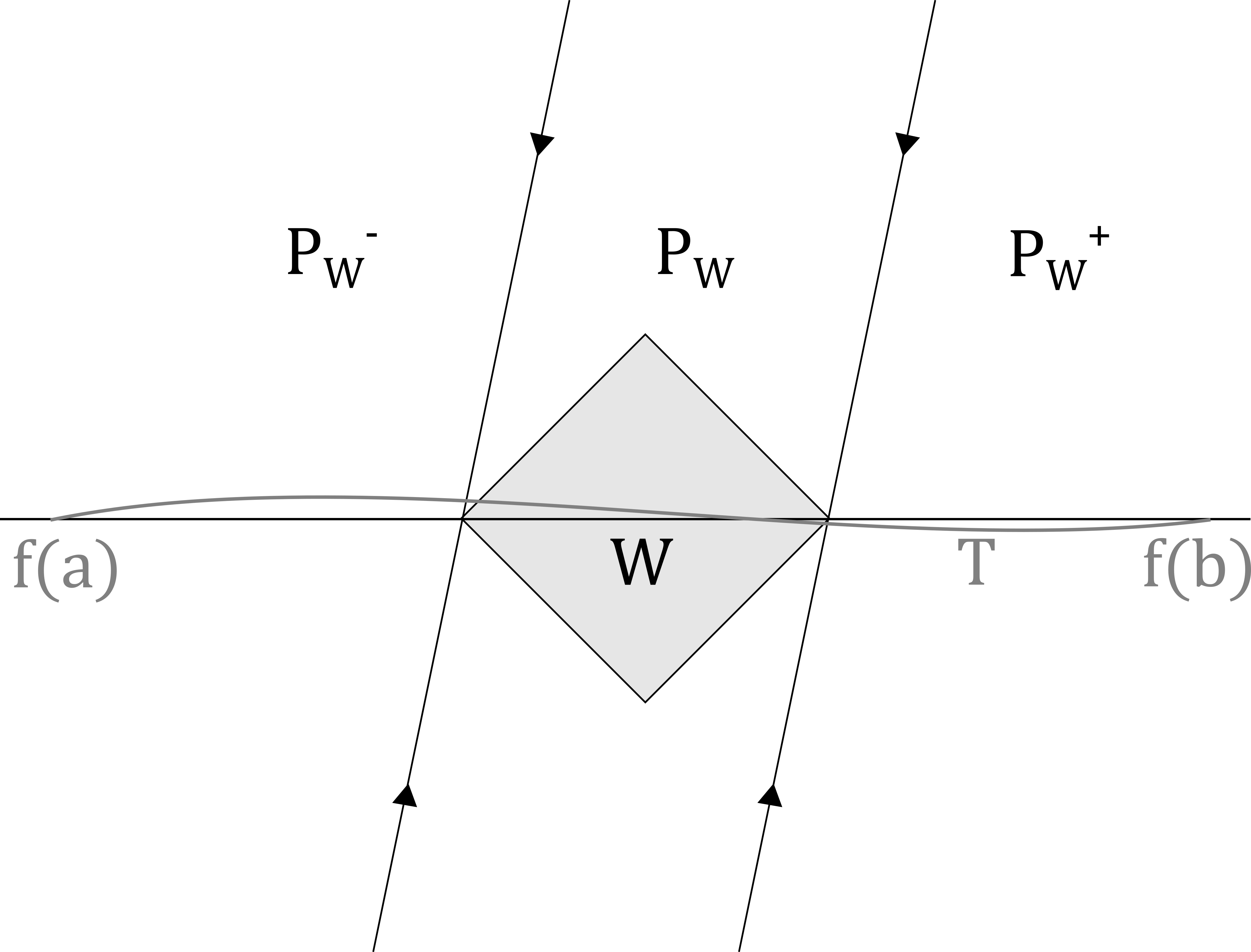}\end{center}\caption[$J$-projection onto a line in $\ell_1^2$]{The cylinder $P_W$ over a ball $W$ with respect to a $J$-projection $\Pi_T=\Pi_{L_T}$ in $\ell_1^2$ (see Appendix \ref{sec:smooth}).}\label{fig:cylinder}\end{figure}

\begin{remark}[projections, cylinders, and transverse arcs] \label{general projections} Given $Q\in\mathscr{G}$ and a \hyperref[def:subarc]{subarc} $T=f([a,b])\subset T'\in\Gamma^*_{U_Q}$, we define the line $L_T := f(a) +\Span\{f(a)-f(b)\}$ and choose a $J$-projection $\Pi_T:\XX\rightarrow L_T$ onto $L_T$ (see Appendix \ref{sec:smooth}). We will often identify $L_T$ with $\mathbb{R}$. By default, we choose this identification so that $f(a)$ lies ``to the left'' of $f(b)$. For every nonempty, bounded set $W\subset \XX$, we define the \emph{cylinder} $P_W:=\Pi_{T}^{-1}(\Pi_{T}(W))$ of $W$ over $L_{T}$. If $W$ is connected, then $P_W$ is connected (because $\Pi_T$ is continuous) and its complement $\XX\setminus P_W$ has two connected components, which we label $P_W^+$ and $P_W^-$ consistent with the orientation of $L_{T}$. If $W$ is convex, then $P_W$ is convex, as well.
See Figure \ref{fig:cylinder}.

We say that an arc $\tau=f|_{[c,d]} \in S^*(\lambda Q)$ is \emph{$W$-transverse} if its two endpoints lie on opposite sides of $P_W$: $\Start(\tau)=f(c)\in P_{W}^\pm$ and $\End(\tau)=f(d)\in P_{W}^\mp$.
\end{remark}

\begin{definition}[``necessary'' cores]\label{nec-cores} Let $Q\in\mathscr{G}$ and let $T=f([a,b]) \subset T' \in \Gamma^*_{U_Q}$ be an efficient subarc. Let $\Pi_T$ be given by Remark \ref{general projections}. Relative to $T$, we declare that a core $U_{Q'}$ with $Q'\in\Child(Q)$ such that $1.00002Q'_*\cap T\neq\emptyset$ has:
\begin{itemize}
\item \emph{Property (N1)} if there exists an arc $\tau\in S(\lambda Q')$ such that $\Image(\tau)$ intersects both $1.00002Q'_*$ and the closed region $P_{1.01Q'_*}\setminus \interior(4Q'_*)$; we say that $\tau$ is \label{tall}\emph{tall}.
\item \emph{Property (N2)} if there exists an arc $\tau\in S(\lambda Q')$ such that $\Image(\tau)\cap 1.00002Q'_*\neq\emptyset$ and $\tau$ is $U_{Q'}$-transverse; we say that $\tau$ is \label{wide}\emph{wide}. See Figure \ref{fig:core-types}.
\end{itemize}
(These properties do not classify all cores $U_{Q'}$ with $Q' \in \Child(Q)$.) Let $\mathcal{N}_1(T)$ and $\mathcal{N}_2(T)$ denote the set of all (N1) cores, and all (N2) cores that are not (N1), respectively. Assign $\mathcal{N}(T):=\mathcal{N}_1(T)\cup\mathcal{N}_2(T)$.
\end{definition}

\begin{figure}\begin{center}\includegraphics[width=.8\textwidth]{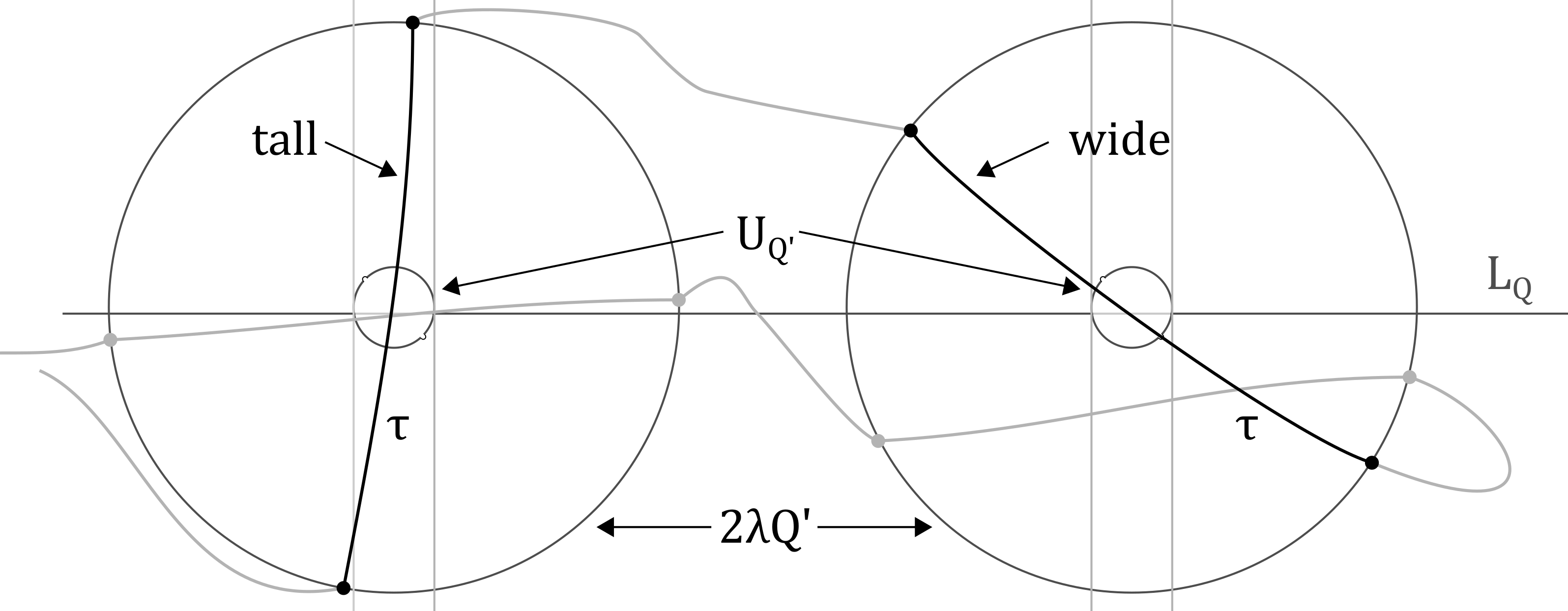}\end{center}\caption[Tall arcs and wide arcs]{On the left, we show a core $U_{Q'}$ of $Q'\in\Child(Q)$ with $2\lambda Q'$ containing a tall arc $\tau$. On the right, we show $U_{Q'}$ with $2\lambda Q'$ containing a wide arc $\tau$. The full set $T=G_Q$ associated to the larger core $U_Q$ is not displayed; since $\diam U_Q\gg \diam U_{Q'}$, the set $G_Q$ may include the union of all arcs in the figure. Cores are much smaller than illustrated.}\label{fig:core-types}\end{figure}

\begin{remark}\label{r:N_1 and N_2 coverage} The cores in $\mathcal{N}(T)$ are ``necessary," because we need them to improve the coarse estimate \eqref{e:coarse estimate}. While necessary cores $U_{Q'}$ lie close to $T$ in the sense that $1.00002Q'_*\cap T\neq\emptyset$, we do not require them to intersect $T$.  The shadows $\Pi_T(U_{Q'})$ of necessary cores cover $\Pi_T(T)\setminus \Pi_T(R_Q)$ up to a small error; see \S\ref{sec: N'} for the details, especially Definition \ref{def:sufficient} and Lemma \ref{topological lemma}.\end{remark}

We now record the main estimates of the paper.

\begin{keylemma}[improving coarse estimate \eqref{e:coarse estimate}] \label{improve} Let $Q\in\mathscr{G}$ and let $T=f([a,b])\subset T'\in\Gamma^*_{U_Q}$ be an \hyperref[def:subarc]{efficient subarc}. Define scales \begin{equation}\label{e:r_T def} r_{T}  := \max\{\diam Q'_*: Q'\in\Child(Q), 1.00002Q'_*\cap T\neq\emptyset\}\ \ \text{and}\ \  \rho_T  := 2\lambda A_\mathscr{H}\cdot 2^{12}r_T.\end{equation}
Suppose $\mathcal{F}$ is a (possibly empty) finite family of cores $U_{Q''}$ with $Q''\in\Child(Q)$ such that $\{2\lambda Q'':U_{Q''}\in\mathcal{F}\}$ is pairwise disjoint and $\mathcal{F}$ satisfies:  \begin{itemize}
\item[(F)] \label{property-F} For all $U_{Q''}\in\mathcal{F}$, we have $2\lambda Q'' \cap 16Q'_*=\emptyset$ for every core $U_{Q'}\in\Child(Q)$ with $\diam Q'>\diam Q''$.
\end{itemize} Let $\mathcal{N}_2=\mathcal{N}_2(T)$ and let $\mathcal{N}_\mathcal{F}$ denote the set of all cores $U_{Q'}$ with $Q'\in\Child(Q)$ such that $U_{Q'}\subset 1.99\lambda Q''$ for some $U_{Q''}\in\mathcal{F}$. Then
\begin{equation}\begin{split} \label{consolidated-improvement}
   \diam T-2\rho_T\leq 2.2\,\ell(R_Q&\cap B_{9r_T}(T)) + \sum_{\mathclap{U_{Q''}\in \mathcal{F}}} \diam 2\lambda Q'' \\ &+ 1.00016 \sum_{\mathclap{U_{Q'}\in \mathcal{N}_2\setminus\mathcal{N}_\mathcal{F}}}\diam H_{Q'} + 0.95 \sum_{\mathclap{U_{Q'}\not\in\mathcal{N}_2\cup\mathcal{N}_\mathcal{F}}} \diam H_{Q'}.
\end{split}\end{equation} where the sums in the second line may be further restricted to $U_{Q'}$ contained in $B_{9r_T}(T)$.
\end{keylemma}

The proof of Lemma \ref{improve} is given in \S\ref{proof-I}, using the setup of \S\ref{sec:prelim} and \S\ref{sec: N'}. We invite the reader to compare and contrast \eqref{consolidated-improvement} with \eqref{e:coarse estimate}.
While the coefficient 1.00016 is substantially smaller than 2.00002, it is unfortunately still not less than 1. As a consequence, we must split verification of \eqref{q-goal} for balls with small remainder into two cases.

\begin{lemma}[Case 2: many non-$\mathcal{N}_2$ cores]\label{case2} If $Q\in\mathscr{G}$ (with or without small remainder) and $U_Q$ has many non-$\mathcal{N}_2(G_Q)$ cores in the sense that \begin{equation}\label{few-N2-cores}\sum_{\mathclap{U_{Q'}\not\in\mathcal{N}_2(G_Q)}}\diam U_{Q'} > 0.05\diam H_Q,\end{equation} then $\diam H_Q < 0.999 s_Q$.
\end{lemma}

\begin{proof} By Lemma \ref{improve}, with $T=G_Q$ and $\mathcal{F}=\emptyset$, together with \eqref{H-G diam est}, the observation $2\rho_{G_Q}\ll \diam H_Q$ (see \eqref{Q'-diam}, \eqref{HQ-bounds}), and \eqref{HQ-bounds-2} and \eqref{few-N2-cores}, we have
\begin{align*}
1.00016s_Q&= 101.01616\,\ell(R_Q)+ 1.00016\sum_{\mathclap{Q'\in\Child(Q)}} \diam H_{Q'}\\
&\geq \diam G_Q - 2\rho_{G_Q} + (1.00016-0.95)\sum_{\mathclap{U_{Q'}\not \in \mathcal{N}_2(G_Q)}} \diam H_{Q'}\\
&\geq 0.99993\diam H_Q - 0.00001\diam H_Q + (0.05016\times 0.49999\times
0.05)\diam H_Q.\end{align*}
Rearranging, we obtain $\diam H_Q \leq 0.99898...s_Q$.
\end{proof}

The final case is the most difficult, requiring us to combine estimates inside and outside of $\{2\lambda Q'':U_{Q''}\in\mathcal{A}\}$ for a family of cores $\mathcal{A}\subset\mathcal{N}_2(G_Q)$. The family $\mathcal{A}$ is chosen according to the following lemma, which we prove in \S\ref{proof-II}.

\begin{keylemma}\label{l:B_2 subset} If $Q\in\mathscr{G}$ has small remainder and $U_Q$ has few non-$\mathcal{N}_2(G_Q)$ cores in the sense that \begin{equation}\label{many-N2-cores}\sum_{\mathclap{U_{Q'}\not\in\mathcal{N}_2(G_Q)}}\diam U_{Q'} \leq 0.05\diam H_Q,\end{equation}
then there exists a finite collection $\mathcal{A} \subset \mathcal{N}_2(G_Q)$ such that $\{2\lambda Q'':U_{Q''}\in\mathcal{A}\}$ is pairwise disjoint, $\mathcal{A}$ satisfies property \hyperref[property-F]{(F)} with $T=G_Q$, \begin{equation}\label{A-estimates 2}
\sum_{U_{Q''}\in\mathcal{A}} \diam 2\lambda Q''\geq 0.04\diam H_Q,\quad\text{and}\end{equation} \begin{equation} \label{A-estimates} \sum_{U_{Q''}\in\mathcal{A}} \diam 2\lambda Q'' \leq 2\,\ell(R_Q)+0.91\sum_{U_{Q'}\in\mathcal{N}_\mathcal{A}} \diam H_{Q'},\end{equation} where $\mathcal{N}_\mathcal{A}:=\{U_{Q'}:Q'\in\Child(Q), U_{Q'}\subset 1.99\lambda Q''\text{ for some }U_{Q''}\in\mathcal{A}\}$.
\end{keylemma}

\begin{lemma}[Case 3: few non-$\mathcal{N}_2$ cores]\label{case3} If $Q\in\mathscr{G}$ has small remainder and \eqref{many-N2-cores} holds, then $\diam H_Q < 0.9963s_Q$.
\end{lemma}

\begin{proof} Let $\mathcal{A}$ be given by Lemma \ref{l:B_2 subset}. By Lemma \ref{improve}, with $T=G_Q$ and $\mathcal{F}=\mathcal{A}$, and \eqref{A-estimates}, $$\diam G_Q-2\rho_{G_Q}\leq 4.2\ell(R_Q) + 0.91\sum_{\mathclap{U_{Q'}\in\mathcal{N}_\mathcal{A}}} \diam H_{Q'} + 1.00016 \sum_{\mathclap{U_{Q'}\not\in \mathcal{N}_\mathcal{A}}}\diam H_{Q'}.$$ Together with \eqref{H-G diam est} and the observation $2\rho_{G_Q}\ll \diam H_Q$ (see \eqref{Q'-diam}, \eqref{HQ-bounds}), followed by \eqref{A-estimates 2} and \eqref{A-estimates} (again), we obtain \begin{align*}
 1.00016s_Q &= 101.01616\,\ell(R_Q)+ 1.00016\sum_{\mathclap{Q'\in\Child(Q)}} \diam H_{Q'}\\
 &\geq \diam G_Q - 2\rho_{G_Q} + (101-4.2)\,\ell(R_Q)+(1.00016-0.91)\sum_{\mathclap{U_{Q'}\in\mathcal{N}_\mathcal{A}}} \diam H_{Q'}\\
&\geq 0.99993\diam H_Q-0.00001\diam H_Q+(0.04\times 0.09016\div 0.91)\diam H_Q.\end{align*} Rearranging, we obtain $\diam H_Q \leq 0.99629... s_Q$.\end{proof}

In review, the hypothesis of Lemma \ref{q-lemma} is satisfied with $q=0.999<1$. This completes the proof of the \hyperref[t:main]{Main Theorem}, up to verification of Lemma \ref{improve} and Lemma \ref{l:B_2 subset}.

\section{Geometric preliminaries and coarse estimates}\label{sec:prelim}

\subsection{Basic geometry with beta numbers}\label{ss:basic}

Let's record consequences of \eqref{restrictions} on the flatness of \hyperref[def:arcs]{almost flat} and \hyperref[def:arcs]{$*$-almost flat arcs} at some common scales. We use the fact that all beta numbers are bounded by 1, $\epsilon_2=2^{-55}\epsilon_1/A_\mathscr{H}\leq 2^{-55}/A_\mathscr{H}$, and $\lambda\leq 5<8$. Let $Q\in\mathscr{G}$ and $\tau\in\Lambda(\lambda Q)$. If $\tau$ is almost flat, i.e.~$\tau\in S(\lambda Q)$, then there is a line $L$ such that
   \begin{equation}\begin{split}\label{almost-flat-line-estimate} \dist(x,L) \leq 2\epsilon_2\beta_{\Sigma}(Q) \Diam \tau &\leq 2^{-54}A_\mathscr{H}^{-1}\diam 2\lambda Q \\
   & \leq 2^{-50}A_\mathscr{H}^{-1} \diam Q \leq 2^{-38}\diam Q_*\quad \forall x\in\Image(\tau).\end{split}\end{equation}
If $\tau$ is $*$-almost flat, i.e.~$\tau\in S^*(\lambda Q)$, then there is a line $L$ such that
    \begin{equation}\begin{split}\label{star-almost-flat-line-estimate} \dist(x,L) \leq 64\epsilon_2\beta_{\Lambda(\lambda Q)}(2\lambda Q&) \Diam \tau \leq 2^{-49}A_\mathscr{H}^{-1} \beta_{S^*(\lambda Q)}(2\lambda Q)\diam 2\lambda Q\\ &\leq 2^{-45} A_\mathscr{H}^{-1}\diam Q \leq 2^{-33}\diam Q_*\quad \forall x\in\Image(\tau),\end{split}\end{equation} where in the second inequality we used $\epsilon_1 \beta_{\Lambda(\lambda Q)}(2\lambda Q)<\beta_{S^*(\lambda Q)}(2\lambda Q)$ by Definition \ref{B-balls-def}. (We shall never refer to $\epsilon_1$ again.) Recall that $2^{-M}<\beta_{S^*(\lambda Q)}(2\lambda Q)\leq 2^{-(M-1)}$ whenever $Q\in\mathscr{G}$.
In particular, for any $Q\in\mathscr{G}$ and $\tau\in S^*(\lambda Q)$, the line $L$ from \eqref{star-almost-flat-line-estimate} also satisfies \begin{equation}\label{line-estimate-with-M}
\dist(x,L)\leq 2^{-49}A_\mathscr{H}^{-1}\beta_{S^*(\lambda Q)}(2\lambda Q) \diam 2\lambda Q < 2^{-M-48} \diam 2\lambda Q\quad \forall  x\in\Image(\tau).\end{equation}

\begin{lemma}[bilateral-$\beta$ estimate for arcs] \label{l:bilateral} Let $\tau=f|_{[a,b]}$ be an arc, let $L$ be a line in $\XX$, and let $\Pi_L$ be a \hyperref[def:J-proj]{$J$-projection} onto $L$. If $\dist(x,L)\leq \beta$ for all $x\in\Image(\tau)$, then \begin{equation}\label{Pi-x} |\Pi_L(x)-x|\leq 2\dist(x,L)\leq 2\beta\quad\text{for all }x\in\Image(\tau),\text{ and}\end{equation}
\begin{equation}\label{bilateral} \dist(y,\Image(\tau))\leq \dist(y,\Pi_L(\Image(\tau)))+2\beta\quad\text{for all }y\in L.\end{equation}
\end{lemma}

\begin{proof} Let $y\in L$. Choose $z\in \Pi_L(\Image(\tau))$ such that $|y-z|=\dist(y,\Pi_L(\Image(\tau)))=:\delta$. Next, choose $x\in\Image(\tau)$ such that $\Pi_L(x)=z$. By Lemma \ref{j-proj facts}, $|z-x|=|\Pi_L(x)-x|\leq 2\dist(x,L)\leq 2\beta$. Thus, $\dist(y,\Image(\tau))\leq |y-x|\leq |y-z|+|z-x|\leq \delta+2\beta$.
\end{proof}

We emphasize that the following inequality (used to prove Lemma \ref{proof:proj-lemma}) is valid in any Banach space; in particular, it does not require uniform nor strict convexity of the norm. It is instructive to think about the inequality in the case when $\XX=\ell^2_\infty=(\RR^2,|\cdot|_\infty)$ and the line segment $(c,d)$ is horizontal.

\begin{lemma}\label{convexity-lemma} Let $c,d\in\XX$, $r>0$, and $0<s<1$. If $c,d\in B(x,r)$ and the segment $(c,d)$ intersects $B(x,sr)$, then $|(1-\mu)c+\mu d-x|\leq r-r(1-s)\min\{\mu,1-\mu\}$ for all $0\leq \mu\leq 1$.
\end{lemma}

\begin{proof} Without loss of generality, we may assume that $x=0$. By assumption, there exists $0<\rho<1$ such that $z=(1-\rho)c+\rho d$ satisfies $|z|\leq sr$. Suppose that $y=(1-\mu)c+\mu d$ for some $0\leq \mu\leq \rho$. Then $y=(1-\nu)c+\nu z=(1-\nu\rho)c+\nu\rho d$ for some $0\leq \nu\leq 1$. This shows $\mu=\nu\rho$; in particular, $\mu\leq \nu$. Hence $|y|\leq (1-\nu)|c|+\nu|z|\leq (1-\nu)r+\nu sr \leq r-r(1-s)\mu$. The case $\rho\leq \mu\leq 1$ is similar, except that $\mu$ should be replaced by $1-\mu$. \end{proof}
\begin{figure}\begin{center}\includegraphics[width=3.2in]{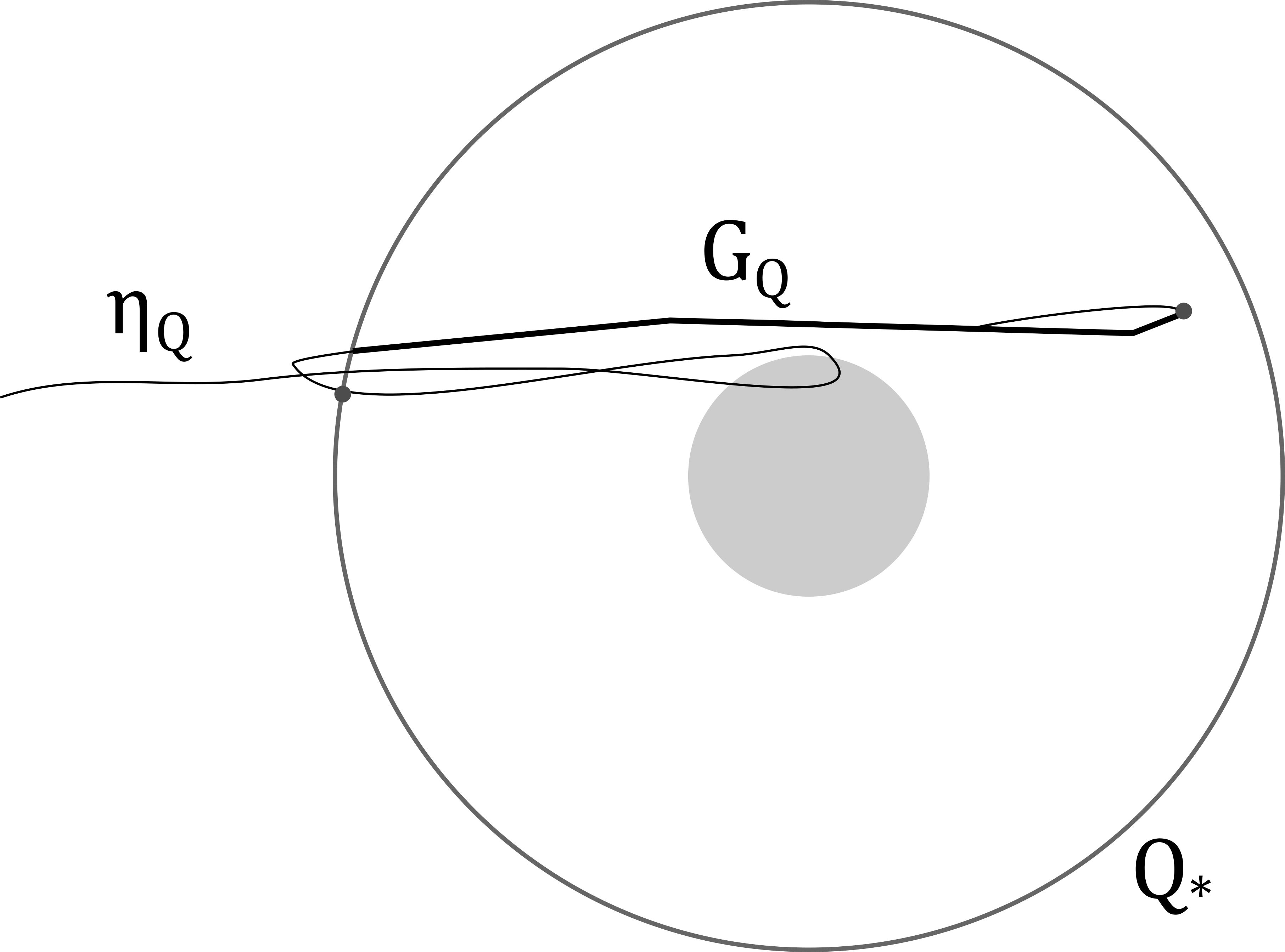}\end{center}\caption[Example of a radial, disconnected arc fragment]{Exaggerated picture (curve should be flatter) of $\eta_Q$ such that $H_Q$ has two connected components. The dots indicate points in $H_Q$ with distance equal to $\diam H_Q$. (Arc through center of $Q_*$ is not displayed.)}\label{fig:GQ}\end{figure}
\begin{lemma}[existence of $G_Q$] \label{proof:proj-lemma} Let $Q\in\mathscr{G}$ and let $T'\in\Gamma^*_{U_Q}$, say $T'=\Image(\tau)\cap U_Q$ for some $\tau=f|_{[a,b]}\in S(\lambda Q)$. If $T'\cap (1/4)Q_*\neq\emptyset$, then there exists $[a_T,b_T]\subset[a,b]$ such that $T:=f([a_T,b_T])$ lies in $T'\cap 0.99999Q_*$, and $|f(a_T)-f(b_T)|=\diam T>0.99993\diam T'$. Moreover, the subarc $T$ intersects $0.25007Q'_*$.\end{lemma}

\begin{proof} Because $\tau$ is almost flat, we can find a line $L$ such that \eqref{almost-flat-line-estimate} holds. Further, since $T'$ intersects $(1/4)Q_*$, it follows that $\diam T'\geq (3/8)\diam Q_*>(1/4)\diam Q_*$ and \begin{equation}
\label{central-almost-flat-line-estimate} \dist(x,L)\leq 2^{-38}\diam Q'_*\leq 2^{-36}\diam T'\quad\forall x\in \Image(\tau);\quad 2^{-38}< 10^{-10}.\end{equation} Let $\Pi_L$ be a $J$-projection onto $L$. Then, by Lemma \ref{l:bilateral}, \begin{equation}\label{Pi-short-move} |\Pi_L(x)-x|< 0.00000\,00002\diam Q'_*\leq 0.00000\,00008\diam T'\quad\forall x\in\Image(\tau).\end{equation} Using \eqref{Pi-short-move} and the triangle inequality, we obtain \begin{equation}\label{flat-projection-bilipschitz} |\Pi_L(x)-\Pi_L(y)|\leq |x-y|< 1.00000\,002\,|\Pi_L(x)-\Pi_L(y)|\end{equation} whenever $x,y\in \Image(\tau)$ and $|x-y|\geq 0.1\diam T'$. Identifying $L$ with $\RR$, we can define $$c:=\min\{\Pi_L(x):x\in \overline{T'}\}\quad\text{and}\quad d:=\max\{\Pi_L(x):x\in \overline{T'}\}.$$ Choosing any $u,v\in \overline{T'}$ such that $|u-v|=\diam\overline{T'}=\diam T'$ and using \eqref{flat-projection-bilipschitz}, we see that \begin{equation}\label{d-c-below} \diam T'\geq d-c\geq |\Pi_L(u)-\Pi_L(v)| > (1.00000\,002)^{-1}\diam T'.\end{equation}

Suppose $c+0.00003\diam T'\leq p\leq d-0.00003\diam T'$ and let $x\in \Pi_L^{-1}(p)\cap\Image(\tau)$. By the first inequality in \eqref{d-c-below}, $p=(1-\mu)c+\mu d$ for some $0<\mu<1$ with $\min\{\mu,1-\mu\}\geq 0.00003$. We would like to use Lemma \ref{convexity-lemma} to show that $x\in 0.99999Q_*$. Let's check the hypothesis of the lemma. Certainly, $c,d\in 1.000011Q_*$ and the segment $(c,d)$ intersects $0.25000\,00002 Q_*\subset 0.25\cdot 1.000011Q_*$, since $T'\subset U_Q\subset 1.00001Q_*$, $T'\cap (1/4)Q_*\neq\emptyset$, and \eqref{Pi-short-move} is in effect. By Lemma \ref{convexity-lemma}, applied with $s=0.25$ and $\min\{\mu,1-\mu\}\geq 0.00003$, we discover $p\in 0.99997\,75\cdot 1.000011Q_*\subset 0.99998\,85Q_*$. Thus, by \eqref{Pi-short-move}, $x\in 0.99999Q_*$.

To continue, because $\Pi_L$ is continuous and $\Image(\tau)$ is connected, there must exist $[\tilde a_T,\tilde b_T]\subset[a,b]$ such that $\Pi_L(f(\tilde a_T))=c+.00003\diam T'$, $\Pi_L(f(\tilde b_T))=d-.00003\diam T'$ (or vice-versa), and $\Pi_L(f(t))$ lies in between for all $t\in[\tilde a_T,\tilde b_T]$. Define $\tilde T:=f([\tilde a_T,\tilde b_T])$. On the one hand, by the previous paragraph, we have $\tilde T\subset \Image(\tau)\cap 0.99999Q_*=T'\cap 0.99999Q_*$, since $0.99999Q_*\subset U_Q$ and $T'\in\Gamma^*_{U_Q}$. Hence, by \eqref{Pi-short-move}, \begin{align*}\diam T'\geq \diam \tilde T\geq |f(\tilde a_T)-f(\tilde b_T)| &\geq d-c-0.00006\diam T'-0.00000\,00016\diam T'.\end{align*} Using the last inequality in \eqref{d-c-below}, it follows that $\diam \tilde T>0.99993\,997\diam T'$. On the other hand, if $s,t\in[\tilde a,\tilde b]$, $\Pi_L(f(s))<\Pi_L(f(t))$, and $\Pi_L(f(s))\geq c+0.00003\,01\diam T'$ or $\Pi_L(f(t))\leq d-0.00003\,01\diam T'$, then \begin{align*}|f(s)-f(t)|\leq d-c-0.00006\,01\diam T' + 0.00000\,00016\diam T'&< 0.99993\,991\diam T',\end{align*} whence $|f(s)-f(t)|<\diam \tilde T$. Choose any $a_T,b_T\in[\tilde a_T,\tilde b_T]$ such that $\Pi_L(a_T)<\Pi_L(b_T)$ and $|f(a_T)-f(b_T)|=\diam \tilde T$. By the previous computation, we necessarily have \begin{equation}\label{aTbTposition} \Pi_L(a_T)<c+0.00003\,01\diam T'\quad\text{and}\quad \Pi_L(b_T)>d-0.00003\,01\diam T'.\end{equation} Define $T:=f([a_T,b_T])$. Then $T$ is an efficient subarc of $T'\cap 0.99999Q_*$ with $\diam T=\diam \tilde T>0.99993\diam T'$.

Lastly, let $y$ be any point such that $y \in (c,d)\cap 0.25000\,00002Q_*$. Shift from $y$ to a point $y'\in\Pi_L(T)$ as needed. By \eqref{aTbTposition}, we can do this in such a way that $|y-y'|<0.00003\,01\diam T'$. Then we can find at least one point $x\in T$ such that $\Pi_L(x)=y'$ and $|x-y|<0.00003\,011\diam T'$ by \eqref{Pi-short-move}. Since $\diam T'$ is at most $2.00002$ times the radius of the ball $Q_*$, we conclude that $x\in T$ lies in $0.25006\,023 Q_*$.
\end{proof}

\subsection{Geometry of $\mathcal{N}_1$ cores}

For each necessary core $U_{Q'}$, we define neighborhoods $D_{Q'}$, $E_{Q'}$, and $F_{Q'}$. Their relationship is that $E_{Q'}$ is slightly smaller than $D_{Q'}$, $F_{Q'}$ is slightly smaller than $E_{Q'}$, and $U_{Q'}$ is smaller than $F_{Q'}$. In \S\ref{proof-I}, we use ``extra length'' from $\diam H_{Q''}$ associated to cores $U_{Q''}$ that intersect $F_{Q'}$ and lie inside of $E_{Q'}$ to ``pay for'' the length of the interval $\Pi_T(D_{Q'})$. The definition of the neighborhoods depends on the type of core.

\begin{definition}\label{d:E, E-star N1 def}
Let $Q\in\mathscr{G}$ and let $T=f([a,b])\subset T'\in\Gamma^*_{U_Q}$ be an \hyperref[def:subarc]{efficient subarc}. For all $U_{Q'}\in\mathcal{N}_1(T)$, we define neighborhoods $D_{Q'}\supset E_{Q'}\supset F_{Q'}$ of $U_{Q'}$ by
\begin{align*}
    D_{Q'} := P_{1.04Q'_*}\cap 4Q'_*, \quad E_{Q'} := P_{1.03Q'_*}\cap 3.99Q'_*, \quad F_{Q'} := P_{1.02Q'_*}\cap 3.98Q'_*.
\end{align*}
\end{definition}

\begin{lemma}[tall subarcs] \label{tall arc criterion} Let $Q\in\mathscr{G}$ and $T=f([a,b])\subset T'\in\Gamma^*_{U_Q}$ be an efficient subarc. If $U_{Q'}\in\mathcal{N}_1(T)$ and $\tau\in S(\lambda Q')$ is a \hyperref[tall]{tall} arc, then there exists a subarc $T_{\tau}$ of $\Image(\tau)\cap F_{Q'} \setminus U_{Q'}$ such that $\diam T_{\tau} \geq 1.48\diam Q'_*$.
\end{lemma}

\begin{proof}Pick any $t_0,t_3\in\Domain(\tau)$ such that $\tau(t_0)\in P_{1.01Q'_*}\setminus \interior(4Q'_*)$ and  $\tau(t_3)\in 1.00002Q'_*$. Without loss of generality, suppose that $t_0<t_3$. We let $t_2 > t_0$ be the first time after $t_0$ with $\tau(t_2) \in \partial(1.00003Q'_*)$. Then we define $t_1 := \max\{t \in [t_0, t_2] : \tau(t) \in \partial(3.97999Q_*') \}.$

We claim that the subarc $T_{\tau}:=\tau([t_1,t_2])$ satisfies the required conditions. Foremost, $\diam T_\tau\geq |\tau(t_1)-\tau(t_2)|\geq 2.97996 \radius Q'_*=1.48998\diam Q'_*$. Also, $T_{\tau}\subset 3.98Q_*\setminus U_{Q'}$ by the way we defined $t_1$ and $t_2$. It remains to verify that $\tau([t_1,t_2]) \subset P_{1.02Q'_*}$. First note that we arranged for $\tau(t_0)$ and $\tau(t_2)$ to lie in $P_{1.01 Q'_*}$. Second note that $\tau$ is almost flat. Consulting \eqref{almost-flat-line-estimate} and \eqref{Pi-x}, we can find a line $L$ and $J$-projection $\Pi_L$ onto $L$ such that \begin{equation}\label{tall-move} |\Pi_L(x)-x|\leq 2\dist(x,L)\leq 2^{-37}\diam Q'_*\quad\text{for every $x\in \Image(\tau)$}.\end{equation} Hence we can locate $y,z\in L$ nearby $\tau(t_0)$ and $\tau(t_3)$ such that $y\not\in 3.999Q'_*$, $z\in 1.001Q'_*$, and $y,z\in P_{1.011}$. By convexity, the whole segment $[y,z]\subset P_{1.011Q'_*}$ too. From \eqref{tall-move}, the fact that $2^{-37}\ll 0.001$, and the triangle inequality it follows that $\tau([t_1, t_2]) \subset B_{2^{-37}\diam Q'_*}([x, y]) \subset P_{1.012Q'_*}$, as well. This shows---with plenty of room to spare---that $T_\tau=\tau([t_1,t_2])$ is a subarc of $\Image(\tau)\cap F_{Q'}\setminus U_{Q'}$.
\end{proof}

\begin{lemma}\label{proof-of-coarse} If $Q\in\mathscr{G}$ and $T\subset \Gamma\cap U_Q$ is a closed, connected set, then the coarse estimate \eqref{e:coarse estimate} holds for $T$.\end{lemma}

\begin{proof} Choose $x,y\in T$ such that $|x-y|=\diam T$ and let $\Pi_T$ be a $J$-projection onto the line through $x$ and $y$; see Appendix \ref{sec:smooth}. Since $\Pi_T$ is 1-Lipschitz, $\Pi_T$ fixes $x$ and $y$, and $T$ is connected, $\Pi_T(T)=[x,y]$. Since $T\subset \Gamma\cap U_Q$, we can cover $T$ by $R_Q\cap T$ and the set of cores $U_{Q'}$ of $Q'\in\Child(Q)$ such that $U_{Q'}\cap T\neq\emptyset$. By countable subadditivity of $\Haus^1$, the isodiametric inequality $\Haus^1(A)\leq \diam A$ for all sets $A\subset\RR$, and $\Pi_T$ being 1-Lipschitz,  \begin{equation}\label{pre-coarse-estimate} \diam T \leq \Haus^1(\Pi_T(R_Q\cap T))+\sum_{\mathclap{U_{Q'}\cap T\neq\emptyset}}\Haus^1(\Pi_T(U_{Q'})) \leq \ell(R_Q\cap T)+\sum_{\mathclap{U_{Q'}\cap T\neq\emptyset}} \diam U_{Q'}.\end{equation} Hence \eqref{e:coarse estimate} follows from \eqref{pre-coarse-estimate} and \eqref{HQ-bounds-2}. \end{proof}

\begin{lemma}\label{N1-auxiliary} Let $Q\in\mathscr{G}$ and let $T\subset T'\in\Gamma^*_{U_Q}$ be an efficient subarc. If $U_{Q'}\in\mathcal{N}_1(T)$, then there is a set $\mathcal{M}_{Q'}$ of cores $U_{Q''}$ with $Q''\in\Child(Q)$ and $U_{Q''}\cap F_{Q'}\neq\emptyset$ such that \begin{equation} \label{N1-auxliary-cores-total}
\diam \Pi_T(D_{Q'})< 0.5\,\ell(R_Q\cap F_{Q'}) + 0.84\sum_{U_{Q''}\in\mathcal{M}_{Q'}} \diam H_{Q''}.\end{equation}
\end{lemma}

\begin{proof} Choose a tall arc $\tau\in S(\lambda Q')$ and let $T_{\tau}$ be the subarc of $\Image(\tau)\cap F_{Q'}\setminus U_{Q'}$ given by Lemma \ref{tall arc criterion}. Define $\mathcal{M}_{Q'}=\{U_{Q'}\}\cup \{U_{Q''}: U_{Q''}\cap T_\tau\neq\emptyset\}$. Applying the coarse estimate \eqref{e:coarse estimate}, we find that \begin{equation*}
1.48\diam Q'_* \leq \diam T_{\tau} \leq \ell(R_Q\cap T_\tau) + 2.00002\sum_{U_{Q''}\in\mathcal{M}_{Q'}\setminus\{U_{Q'}\}}\diam H_{Q''}.\end{equation*} We also know that $\diam Q'_*\leq \diam U_{Q'}\leq 2.00002 \diam H_{Q'}$ by \eqref{HQ-bounds-2}. Hence \begin{equation*} 2.38461 \diam 1.04Q'_*\leq 2.48\diam Q'_* \leq \ell(R_Q\cap T_{\tau}) + 2.00002\sum_{U_{Q''}\in\mathcal{M}_{Q'}} \diam H_{Q''}.\end{equation*} Since $\diam \Pi_T(D_{Q'})\leq \diam 1.04Q'_*$ and $T_\tau\subset F_{Q'}$, this yields \eqref{N1-auxliary-cores-total}.
\end{proof}

\subsection{Geometry of $\mathcal{N}_2$ cores}\label{ss:N2} By definition, every core $U_{Q'}\in \mathcal{N}_2(T)$ admits a \hyperref[wide]{wide} arc. To prove Lemma \ref{improve}, we will need to distinguish between the case that some wide arc $\tau$ lies near the center of $Q'_*$ and the case that every wide arc is far from the center of $Q'_*$.

\begin{definition} Let $Q\in\mathscr{G}$ and let $T\subset T'\in\Gamma^*_{U_Q}$ be an \hyperref[def:subarc]{efficient subarc}. Suppose that $U_{Q'}\in\mathcal{N}_2(T)$. We say that $U_{Q'}\in\mathcal{N}_{2.1}(T)$ if there exists a wide arc $\tau$ such that $\Image(\tau)\cap 2^{-14}Q'_*\neq\emptyset$. Otherwise, we say that $U_{Q'}\in\mathcal{N}_{2.2}(T)$.\end{definition}

\begin{definition}\label{d:E, E-star N2 def}
Let $Q\in\mathscr{G}$ and $T=f([a,b])\subset T'\in\Gamma^*_{U_{Q'}}$ be an efficient subarc.  For all $U_{Q'}\in\mathcal{N}_{2.1}(T)$, we define neighborhoods $D_{Q'}\supset E_{Q'}\supset F_{Q'}$ of $U_{Q'}$ by
\begin{align*}
    D_{Q'} := 1.00002Q'_*, \quad E_{Q'} := U_{Q'}, \quad F_{Q'} := U_{Q'}.
\end{align*}
For all $U_{Q'}\in\mathcal{N}_{2.2}(T)$, we define neighborhoods $D_{Q'}\supset E_{Q'}\supset F_{Q'}$ of $U_{Q'}$ by
\begin{align*}
    D_{Q'} := 16Q'_*, \quad     E_{Q'} := 15.99Q'_*, \quad F_{Q'} := 15.98Q'_*.
\end{align*}
\end{definition}

\begin{lemma}\label{N21-auxiliary}
Let $Q\in\mathscr{G}$ and let $T\subset T'\in\Gamma^*_{U_Q}$ be an efficient subarc. If $U_{Q'}\in\mathcal{N}_{2.1}(T)$, then $\diam D_{Q'} \le 1.00016 \diam H_{Q'}$.
\end{lemma}

\begin{proof}
Let $\tau$ be a wide arc such that $\Image(\tau)\cap 2^{-14}Q'_*\neq\emptyset$.  By \eqref{almost-flat-line-estimate}, there exists a line $L$ such that $\dist(p,L)\leq 2^{-38}\diam Q'_*$ for all $p\in\Image(\tau)$. Since $\tau$ is wide and $\Image(\tau)$ intersects $2^{-14}Q'_*$, the set $\Image(\tau)$ meets both connected components of $\partial Q'_*\cap B_{2^{-38}\diam Q'_*}(L)$; choose points $y,z\in\Image(\tau)\cap \partial Q'_*$, one from each of the components. Let $x$ denote the center of $Q'_*$; then $\dist(x,\Image(\tau))\leq 2^{-14}\radius Q'_*=2^{-15}\diam Q'_*$. By our assertions above, we can find points $x',y',z'\in L$, with $x'$ lying between $y'$ and $z'$, such that $|x-x'| \leq (2^{-15}+2^{-38})\diam Q'_*$, $|y-y'| \leq 2^{-38}\diam Q'_*$, and $|z-z'|\leq 2^{-38}\diam Q'_*$. Define $y''=y'+x-x'$ and $z''=z'+x-x'$, so that $y''$ and $z''$ lie on a line through $x$, with $x$ in between $y''$ and $z''$. Now, $$|y''-x|\geq |y-x|-|y''-y'|-|y'-y|\geq (1/2-2^{-15}-2^{-37})\diam Q'_*.$$ Similarly, $|z''-x|\geq (1/2-2^{-15}-2^{-37})\diam Q'_*$. Hence $|y''-z''|=|y''-x|+|x-z''|\geq (1-2^{-14}-2^{-36})\diam Q'_*$. It follows that $$|y-z|\geq |y''-z''|-|y''-y'|-|y'-y|-|z''-z'|-|z'-z|\geq (1-2^{-13}-2^{-35})\diam Q'_*.$$ Thus, $\diam H_{Q'}\geq |y-z| \geq 0.99987\diam Q'_*\geq 0.99985\diam D_{Q'}$. The lemma follows.\end{proof}

\begin{lemma}\label{N22-geometry} Let $Q\in\mathscr{G}$ and let $T\subset T'\in\Gamma^*_{U_Q}$ be an efficient subarc. If $U_{Q'}\in\mathcal{N}_{2.2}(T)$, then there exists a finite set $\mathcal{Y}$ of efficient subarcs of arc fragments in $\Gamma^*_{F_{Q'}}$ such that the sets $\{1.00002Q'_*\}\cup\{B_{2^{-40}\diam Q'_*}(Y):Y\in\mathcal{Y}\}$ are pairwise disjoint, $\diam Y\geq 0.00021 \diam Q'_*$ for all $Y\in\mathcal{Y}$, and $\sum_{Y\in\mathcal{Y}} \diam Y \geq 22.46\diam Q'_*$. (The cardinality of $\mathcal{Y}$ is 3 or 4.)
\end{lemma}

\begin{proof} Since $U_{Q'}\in\mathcal{N}_{2.2}(T)$, we can find a wide arc $\tau \in S(\lambda Q')$ such that $\Image(\tau)$ intersects $1.00002Q'_*$ and is disjoint from $2^{-14}Q'_*$. Let $\xi\in S(\lambda Q')$ be any arc whose image contains the center of $Q'$. Since the image of $\tau$ does not contain the center of $Q'$, the arcs $\tau$ and $\xi$ are distinct. The family $\mathcal{Y}$ will be built from subarcs of $\Image(\tau)\cap F_{Q'}$ and $\Image(\xi)\cap F_{Q'}$.

Let $A$ denote the annulus $15.98Q'_*\setminus \interior(1.00004Q'_*)$, which is contained in $F_{Q'}$. Choose a subarc $T_1$ of $\Image(\xi)\cap A$ with one endpoint on $\partial(15.98Q'_*)$ and one endpoint on $\partial (1.00004Q'_*)$ so that $\diam T_1\geq 14.97996\radius Q'_*$; cf.~proof of Lemma \ref{tall arc criterion}. Similarly, we may find two subarcs $T_2$ and $T_3$ of $\Image(\tau) \cap A$ with endpoints in $\partial 15.98Q'_* \cap P_{U_{Q'}}^+$ and $\partial(1.00004Q'_*)$ and endpoints in $\partial(15.98Q'_*) \cap P_{U_{Q'}}^-$ and $\partial (1.00004Q'_*)$, respectively. Observe that $\min\{\diam T_2,\diam T_3\}\geq 14.97996\radius Q'_*$, and the total diameter of the three subarcs is at least $44.93988\radius Q'_*=22.46994\diam Q'_*$.

Now, we show that $B_{2^{-40}\diam Q'_*}(T_2)$ and $B_{2^{-40}\diam Q'_*}(T_3)$ are disjoint.  Let $L_{\tau}$ be a line such that (\ref{almost-flat-line-estimate}) holds for $L_{\tau}$ and all $x \in \Image(\tau)$.  In particular,
\begin{align*}
    B_{2^{-40}\diam Q'_*}(T_2 \cup T_3) \subset B_{(2^{-38}+ 2^{-40})\diam Q'_*}(L_{\tau}) \subset B_{2^{-37}\diam Q'_*}(L_{\tau}).
\end{align*}
By assumption, $\Image(\tau) \cap 1.00002Q'_* \not = \emptyset$. Hence there exists $\tilde{w} \in \Image(\tau) \cap 1.00002Q'_*$ such that $B(\tilde{w}, 0.00002\radius Q'_*) \subset 1.00004 Q'_*$. Let $w \in L_{\tau} \cap B(\tilde{w}, 2^{-38}\diam Q'_*)$. Note that $B(w, 0.00001 \radius Q'_*) \subset 1.00004 Q'_*$. Labeling the two connected components of $L_{\tau} \setminus B(w, 0.00001\radius Q'_*)$ by $L_{\tau}^+, L_{\tau}^-$ we conclude that
\begin{equation}\begin{split}\label{2.2 gap estimate}
&\gap(B_{2^{-40}\diam Q'_*}(T_2), B_{2^{-40}\diam Q'_*}(T_3))  \ge \gap(B_{2^{-37}\diam Q'_*}(L_{\tau}^+), B_{2^{-37}\diam Q'_*}(L_{\tau}^-))\\
&\qquad\ge (0.00001 - 2^{-35})\diam Q'_* > 0.000009 \diam Q'_*.
\end{split}\end{equation}
Observe that for any arc we may shrink its domain as needed to produce an efficient arc of the same diameter. Thus, it remains to obtain a subarc or subarcs of $T_1$ which satisfy the disjointness and diameter estimates in the conclusion of the lemma.

Let $L_{\xi}$ be a line such that \eqref{almost-flat-line-estimate} holds with $L_{\xi}$ and all $x \in \Image(\xi)$. As with $\tau,$ we have
\begin{align*}
    B_{2^{-40}\diam Q'_*}(T_1) \subset B_{(2^{-38}+ 2^{-40})\diam Q'_*}(L_{\tau}) \subset B_{2^{-37}\diam Q'_*}(L_{\xi}).
\end{align*}
If $B_{2^{-40}\diam Q'_*}(T_1)$ and $B_{2^{-40}\diam Q'_*}(T_2 \cup T_3)$ do not intersect, by the previous paragraph we are done.  If, on the other hand, $B_{2^{-40}\diam Q'_*}(T_1)$ and $B_{2^{-40}\diam Q'_*}(L_{\tau})$ intersect, then $B_{2^{-37}\diam Q'_*}(L_\xi)$ and $B_{2^{-37}\diam Q'_*}(L_{\tau})$ intersect.  In this case, we must either shrink $T_1$ or split $T_1$ into two subarcs to obtain the desired disjointness.  See Figure \ref{fig:arc-separation}.

\begin{figure}\begin{center}\includegraphics[width=.9\textwidth]{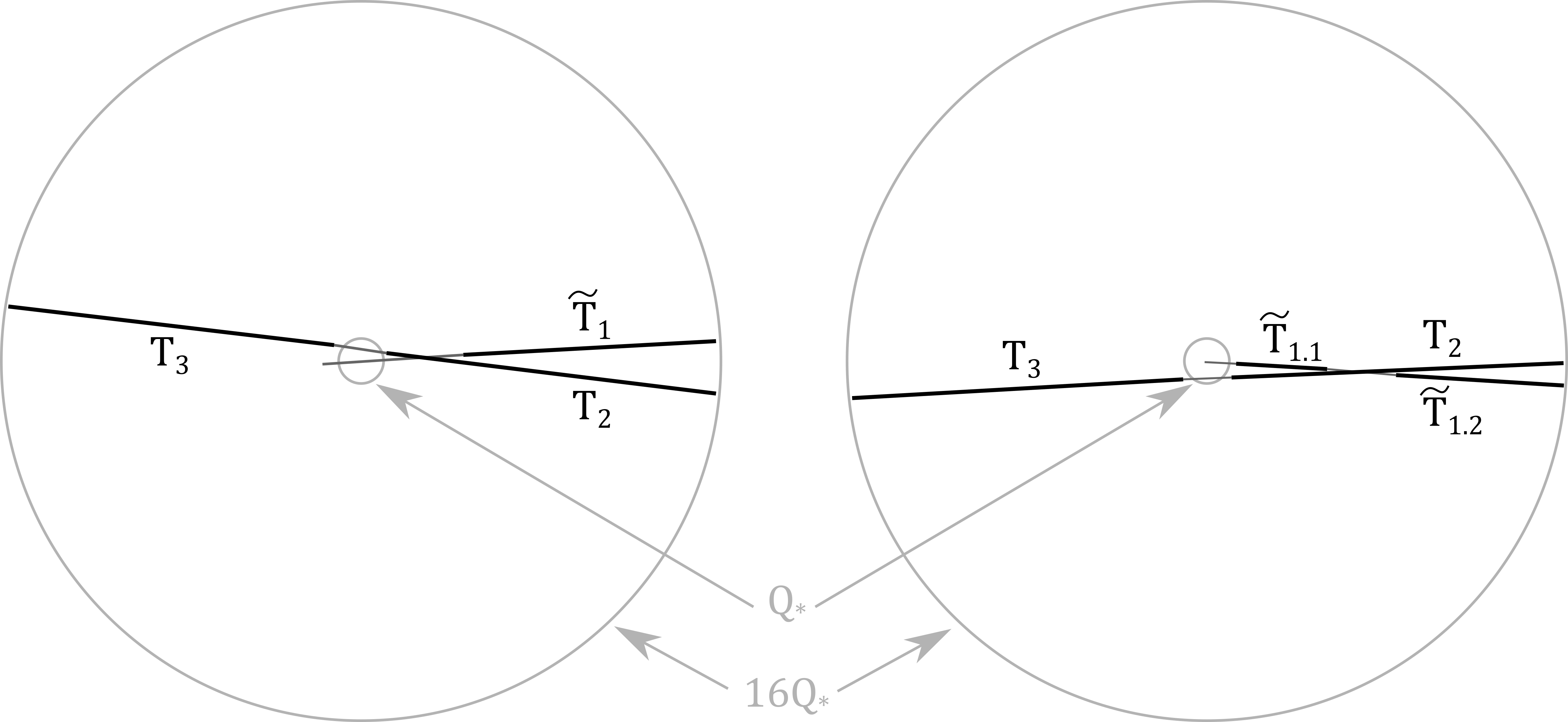}\end{center}\caption[Separated subarcs (small scale geometry)]{Separated subarcs $\mathcal{Y}$ associated to $\mathcal{N}_{2.2}(T)$-type cores $U_{Q'}$. Either $\#\mathcal{Y}=3$ (left) or $\#\mathcal{Y}=4$ (right). The arc $T$ is not displayed.}\label{fig:arc-separation}\end{figure}

Suppose then that $B_{2^{-37}\diam Q'_*}(L_{\tau})$ intersects $B_{2^{-37}\diam Q'_*}(L_{\xi})$. Then, $L_\xi$ intersects $B_2:=B_{2^{-35}\diam Q'_*}(L_\tau)$ by the triangle inequality. Define $$r_1:=\min\{|z-x|: z\in L_\xi \cap B_2\}\quad\text{and}\quad r_2:=\max\{|z-x|:z\in L_\xi \cap B_2\},$$ where as before $x$ denotes the center of $Q'$. Our goal is to show that $r_2-r_1$ is relatively small. There are two cases.

For the easier case, suppose that $r_2\leq 1.00054\radius Q'_*$ or $r_1\geq 15.9795\radius Q'_*$. Replace $T_1$ with a subarc $\tilde T_1$ using the annulus $15.97949Q'_*\setminus\interior(1.00055 Q'_*)$ instead of $A$. Then $\diam \tilde T_1\geq 14.97894\radius Q'_*$ and $\diam \tilde T_1+\diam T_2+\diam T_3\geq 22.46943\diam Q'_*$. Furthermore, because $\tilde T_1\subset T_1$ and $\tilde T_1$ avoids $\{w:|w-x|\in [r_1,r_2]\}$,
\begin{align*}
    &\gap(B_{2^{-40}\diam Q'_*}(\tilde T_1), B_{2^{-40}\diam Q'_*}(T_2 \cup T_3))\\
    &\qquad \ge \gap(B_{2^{-37}\diam Q'_*}\left(L_{\xi}\cap (15.97949Q'_*\setminus\interior(1.00055 Q'_*)\right), B_{2^{-37}\diam Q'_*}(L_{\tau}))\\
    &\qquad \geq (0.00001- 2^{-36})\diam Q'_*>0.
\end{align*}
Thus, the neighborhoods $B_{2^{-40}\diam Q'_*}(\tilde T_1)$ and $B_{2^{-40}\diam Q'_*}(T_2 \cup T_3)$ are disjoint.

For the harder case, suppose that \begin{equation}\label{hard-r} r_2> 1.00054\radius Q'_*\quad\text{and}\quad r_1<15.9795\radius Q'_*.\end{equation} Let $y\in L_\xi \cap B_2\cap \partial B(x,r_1).$ Let $z \in L_\xi \cap B_2 \cap\partial B(x,r_2).$  By translation, we may replace $L_\tau$ with a line (which we relabel as $L_\tau$) such that $y\in L_\xi\cap L_\tau$.  Since we translate by at most $2^{-35}\diam Q'_*$, the triangle inequality implies that $B_{2^{-40}}(\Image(\tau)) \subset B_{2^{-34}\diam Q'_*}(L_{\tau})$ and $$\dist(x,L_\tau)\geq \dist(x,\Image(\tau))-\sup_{w\in \Image(\tau)}\dist(w,L_\tau)\geq (2^{-15}-2^{-34})\diam Q'_*.$$

Now, choose $J$-projections $\Pi_\xi$ and $\Pi_\tau$ onto $L_\xi$ and (the relabeled line) $L_\tau$, respectively. Then the points $x_\xi:=\Pi_\xi(x)$,  $x_{\xi\tau}:=\Pi_\tau(x_\xi)$, and $z_{\tau}:=\Pi_\tau(z)$ satisfy: \begin{align}
|x_\xi-x_{\xi\tau}| &\geq |x-x_{\xi\tau}|-|x-x_\xi|\geq (2^{-15} - 2^{-34})\diam Q'_*,\\
|x_\xi-y| &\leq |x-y|+|x-x_\xi|\leq 15.9795\radius Q'_*+2^{-38}\diam Q'_*< 2^3\diam Q'_*,\\
|z-y|&\geq |z-x|-|x-y|\geq r_2-r_1,\text{ and}\\
|z-z_{\tau}| &\leq 2\dist(z,L_\tau)\leq 2^{-33}\diam Q'_*.\end{align} By ``similar triangles'', it follows that  \begin{equation} r_2-r_1\leq |z-y|=|x_\xi-y|\frac{|z-z_{\tau}|}{|x_\xi-x_{\xi\tau}|}\leq (2^{3}\diam Q'_*)\frac{2^{-33}}{2^{-15} - 2^{-34}}.\end{equation} Hence $r_2-r_1 < 2^{-14}\diam Q'_*$.

Since $\xi$ contains $x$ and \eqref{almost-flat-line-estimate} is in effect we may translate $L_\xi$ (by at most $2^{-38}\diam Q'_*$) to obtain a line $\tilde L_{\xi}$ which contains $x$. Thus, each component of $\tilde L_{\xi} \cap B(x, r_2) \setminus B(x, r_1)$ has diameter $r_2-r_1$. By \eqref{almost-flat-line-estimate}, and \eqref{hard-r}, and the triangle inequality, we estimate that each component $L_{\xi}^{\pm}$ of $L_{\xi} \cap (B_{r_2}(x)\setminus B_{r_1}(x))$ satisfies
\begin{align*}
    \diam L^{\pm}_{\xi} & \le r_2 - r_1 + 2^{-37}\diam Q'_* \le (2^{-14} + 2^{-37}) \diam Q'_*.
\end{align*}
In particular, $\diam B_{2^{-35}\diam Q'_*}(L_{\xi}^{\pm} \cap (B_{r_2}(x)\setminus B_{r_1}(x))) \le (2^{-14} +2^{-34} + 2^{-37}) \diam Q'_* \le 0.000062\diam Q'_*$.
This estimate, and the assumption \eqref{hard-r} imply that we may choose radii $\tilde r_1$ and $\tilde r_2$ such that $$1.00047\radius Q'_*<\tilde r_1<r_1<r_2<\tilde r_2<15.97957\radius Q'_*$$ and $\tilde r_2-\tilde r_1=0.00007\radius Q'_*$. Let $\tilde T_{1.1}$ be a subarc of $T_1\cap B(x,\tilde r_1)\setminus \interior(1.00005Q'_*)$ with one endpoint in $\partial(1.00005Q'_*)$ and one endpoint in $\partial B(x,\tilde r_1)$. Define $\tilde T_{1.2}$ similarly using the annulus $15.97999Q'_*\setminus\interior(B(x,\tilde r_2))$.

We now demonstrate that the neighborhoods $B_{2^{-40}\diam Q'_*}(\tilde T_{1.1})$, $B_{2^{-40}\diam Q'_*}(\tilde T_{1.2})$, and $B_{2^{-40}\diam Q'_*}(T_2 \cup T_3)$ are pairwise disjoint.  First, note that
\begin{align*}
    &\gap(B_{2^{-40}\diam Q'_*}(\tilde T_{1.1}), B_{2^{-40}\diam Q'_*}(\tilde T_{1.2}))\\
    &\qquad \geq \gap(B_{3^{-37}}(L_{\xi} \cap B(x,\tilde r_1)), B_{3^{-37}}(L_{\xi} \cap B(x,\tilde r_2)^c)) \geq (0.00007 - 2^{-36})\diam Q'_*>0.
\end{align*}
Thus, the neighborhoods $B_{2^{-40}\diam Q'_*}(\tilde T_{1.1})$ and $B_{2^{-40}\diam Q'_*}(\tilde T_{1.2})$ are pairwise disjoint. They are also pairwise disjoint from $B_{2^{-40}\diam Q'_*}(T_2 \cup T_3)$, because 
\begin{align*}
    &\gap(B_{2^{-40}\diam Q'_*}(\tilde T_{1.1} \cup \tilde T_{1.2}), B_{2^{-40}\diam Q'_*}(T_2 \cup T_3))\\
    &\qquad \geq \gap\left(B_{2^{-37}}\left(L_{\xi} \cap \left(B(x,\tilde r_1) \cup B(x,\tilde r_2)^c \right)\right), B_{2^{-37}\diam Q'_*} (L_{\tau})\right)\\
    &\qquad \geq (2^{-35} - 2^{-36})\diam Q'_*>0.
\end{align*}
By definition of $\tilde r_1, \tilde r_2$, $\min\{\diam \tilde T_{1.1},\diam \tilde T_{1.2}\}\geq 0.00042\radius Q'_*=0.00021\diam Q'_*$ and $\diam \tilde T_{1.1}+\diam \tilde T_{1.2} \geq (15.97999-1.00005-0.00007)\radius Q'_*=14.97987\radius Q'_*.$ Thus, $\diam \tilde T_{1.1}+\diam\tilde T_{1.2}+\diam T_2+\diam T_3\geq 22.46989\diam Q'_*$. This concludes the proof of the lemma.
\end{proof}

\begin{lemma}\label{N22-auxiliary} Let $Q\in\mathscr{G}$ and let $T\subset T'\in\Gamma^*_{U_Q}$ be an efficient subarc. If $U_{Q'}\in\mathcal{N}_{2.2}(T)$, then there is a family $\mathcal{M}_{Q'}$ of cores $U_{Q''}$ with $Q''\in\Child(Q)$ and $U_{Q''} \cap F_{Q'} \not= \emptyset$ such that \begin{equation} \label{N22-auxliary-cores-total}
\diam D_{Q'} < 0.7\,\ell(R_Q\cap F_{Q'}) + 1.37 \sum_{U_{Q''}\in\mathcal{M}_{Q'}} \diam H_{Q''}.\end{equation}
\end{lemma}

\begin{proof} Given $\mathcal{Y}$ from Lemma \ref{N22-geometry}, let $\mathcal{M}_{Q'}=\{U_{Q''}:Q''\in\Child(Q), U_{Q''}\cap F_{Q'}\neq\emptyset\}.$ If there happens to exist $U_{Q''} \in \mathcal{M}_{Q'}$ with $\diam Q'' > \diam Q'$, then $$\diam H_{Q''} \ge 0.49999\diam Q''_*\geq 2^{98}\diam Q'_*\gg \diam D_{Q'}$$ and \eqref{N22-auxliary-cores-total} holds trivially. Assume otherwise that $\diam Q'' \le \diam Q'$ for every core $U_{Q''} \in \mathcal{M}_{Q'}$, so that $\diam U_{Q''}\leq 2^{-98}\diam Q'_*$ for all $U_{Q''} \in \mathcal{M}_{Q'}\setminus\{U_{Q'}\}$. Because $\diam Q'_*\leq 2.00002\diam H_{Q'}$ and \eqref{e:coarse estimate} holds for  each $Y\in\mathcal{Y}$, Lemma \ref{N22-geometry} implies that
\begin{align}\label{e:N22 ineq}
    (1+22.46)\diam Q'_* & \le \ell(R_Q\cap F_{Q'}) + 2.00002 \sum_{U_{Q''}\in\mathcal{M}_{Q'}} \diam H_{Q''}.
\end{align}
Since $\diam D_{Q'}=16\diam Q'_*$, this estimate yields \eqref{N22-auxliary-cores-total}.\end{proof}

\subsection{Geometry of unnecessary cores}

\begin{definition} Let $\Pi_T$ be a $J$-projection onto some line $L_T$ in $\XX$. For any line $L$ in $\XX$, the \emph{antislope} $\mathsf{as}(L,\Pi_T)$ of $L$ \emph{relative to} $\Pi_T$ is the unique number in $[0,1]$ given by \begin{equation}\label{as-def}\mathsf{as}(L,\Pi_T)=\frac{|\Pi_T(u)-\Pi_T(v)|}{|u-v|}\quad\text{for any }u,v\in L\text{ with }u\neq v.\end{equation}\end{definition}

\begin{remark}The antislope $\mathsf{as}(L,\Pi_T)$ is well-defined (i.e.~the quantity in \eqref{as-def} does not depend on the choice of points $u,v$) by linearity of $J$-projections onto linear subspaces. At one extreme, $\mathsf{as}(L,\Pi_T)=0$ if and only if $L$ is \emph{vertical} in the sense that $\Pi_T(u)=\Pi_T(v)$ for every $u,v\in L$. At the other extreme, $\mathsf{as}(L,\Pi_T)=1$ if and only if $L$ is parallel to $L_T$.\end{remark}

\begin{lemma}[location of endpoints]\label{unnecessary-endpoints}
Let $Q \in \mathscr{G}$ and let $T=f([a,b]) \subset T' \in \Gamma^*_{U_Q}$ be an \hyperref[def:subarc]{efficient subarc}. If $Q'\in \Child(Q)$ and $1.00002Q'_*\cap T\not = \emptyset$, but the core $U_{Q'}$ is ``unnecessary'' in the sense that $U_{Q'} \not \in \mathcal{N}(T)=\mathcal{N}_1(T)\cup\mathcal{N}_2(T)$, then for all arcs $\tau=f|_{[c,d]}\in S(\lambda Q')$ such that $\Image(\tau)\cap 1.00002 Q'_*\neq \emptyset$,
\begin{align*}
\text{either}\quad\{f(c), f(d)\} \subset P_{15Q_*'}^+ \cap \partial(2\lambda Q')\quad\text{or}\quad\{f(c), f(d)\} \subset P_{15Q_*'}^- \cap \partial (2\lambda Q');
\end{align*} moreover, if $L$ is any line such that \eqref{almost-flat-line-estimate} holds for $\tau$, then $\mathsf{as}(L,\Pi_T)>0.001$.
\end{lemma}

\begin{proof} Let $Q'$ be given as in the statement. Fix any $\tau=f|_{[c,d]}\in S(\lambda Q')$. Because $Q'\not\in \mathscr{B}^\lambda_0$ (see Lemma \ref{l:B0}), the endpoints $f(c)$ and $f(d)$ of $\tau$ lie on $\partial(2\lambda Q')$. Since $U_{Q'} \not \in \mathcal{N}_1(T),$ we know that $f(c),f(d)\not\in P_{1.01Q'_*}$. Suppose without loss of generality that $f(c)\in P_{1.01Q'_*}^+$ (see Remark \ref{general projections}). Since $U_{Q'}\not\in\mathcal{N}_2(T)$, we have $f(d)\in P_{1.01Q'_*}^+$, as well. To complete the proof, it suffices to show that $f(c), f(d) \not \in P_{15Q'_*}$.

Let $L$ be a line such that \eqref{almost-flat-line-estimate} holds for $\tau$. Since $\Image(\tau)\cap 1.00002Q'_*\neq\emptyset$, it follows that $L\cap 1.000021Q'_* \neq \emptyset$. Choose any $u\in L\cap 1.000021Q'_*$. Similarly, let $x\in \Image(\tau)\cap \partial(4Q'_*)$. Since $\tau$ is not tall, $x\not\in P_{1.01Q'_*}$. Thus, by \eqref{almost-flat-line-estimate}, there exists $v\in L\cap 4.00001Q'_*\cap P_{1.00999Q'_*}^+$. Finally, choose $w\in L$ such that $|w-f(c)|\leq 2^{-38}\diam Q'_*$. This more than guarantees $w\in \XX\setminus (2^{13}\lambda A_\mathscr{H} - 1)Q'_*\subset \XX\setminus 8191Q'_*$. Now, \begin{align*} |\Pi_T(w)-\Pi_T(u)|= |w-u|\frac{|\Pi_T(v)-\Pi_T(u)|}{|v-u|}&\geq (8189\radius Q'_*)\frac{1.00999-1.000021}{4.00001+1.000021}\\ &\geq 16.01697\radius Q'_*.\end{align*} Hence $w$ lies outside of $P_{15.01676Q'_*}$, and therefore, $f(c)$ certainly lies outside of $P_{15Q'_*}$. An identical argument shows that $f(d)$ lies outside of $P_{15Q'_*}$, as well. Finally, from the display, we read off $\mathsf{as}(L;\Pi_T)\geq (1.00999-1.000021)/(4.00001+1.000021)=0.00199...$.
\end{proof}

\begin{lemma}[overlapping arcs]\label{unnecessary-overlaps} Let $Q\in\mathscr{G}$ and let $T=f([a,b])\subset T'\in\Gamma^*_{U_Q}$ be an efficient subarc. Let $Q^\sigma,Q^\tau\in\Child(Q)$ with $\diam Q^\sigma\leq \diam Q^\tau$ and suppose that there is a point $x\in\Pi_T(1.00002Q^\sigma_*\cap T)\cap\Pi_T(1.00002Q^\tau_*\cap T)$, but $U_{Q^\sigma},U_{Q^\tau}\not\in \mathcal{N}(T)$. For any arcs $\sigma\in S(\lambda Q^\sigma)$ and $\tau\in S(\lambda Q^\tau)$ such that $\Domain(\sigma),\;\Domain(\tau)\subset[a,b]$, $x\in \Pi_T(\Image(\sigma))\cap \Pi_T(\Image(\tau))$, and $\Domain(\sigma)\cap\Domain(\tau)\neq\emptyset$, either \begin{enumerate}
\item[(i)] $\diam Q^\sigma<\diam Q^\tau$ and $\Domain(\sigma)\subset\Domain(\tau)$, or
\item[(ii)] $\diam Q^\sigma=\diam Q^\tau$ and $[c,d]:=\Domain(\sigma)\cup\Domain(\tau)$ satisfies \begin{align*}
\text{either}\quad\{f(c), f(d)\} \subset P_{12Q^\sigma_*}^+\cap P_{12Q^\tau_*}^+\quad\text{or}\quad\{f(c), f(d)\} \subset P_{12Q^\sigma_*}^-\cap P_{12Q^\tau_*}^-.
\end{align*}
\end{enumerate}
\end{lemma}

\begin{proof} Firstly, suppose that $\diam Q^\sigma<\diam Q^\tau$. Let $x_\sigma$ denote the center of $Q^\sigma$ and pick $y_\sigma\in 1.00002Q^\sigma_*\cap \Image(\sigma)\cap P_x$. Here $P_x$ is shorthand for $P_{\{x\}}$ (see Remark \ref{general projections}). Then, for any $z_\sigma\in \Image(\sigma)\subset 2\lambda Q^\sigma$, \begin{equation*}\begin{split}
|\Pi_T(z_\sigma)-x| &\leq |z_\sigma-y_\sigma| \leq |z_\sigma-x_\sigma|+|x_\sigma-y_\sigma| \\ &\leq (1+1.00002\cdot 2^{-13})\radius 2\lambda Q^\sigma
\leq 2^{-99}A_\mathscr{H}^{-1}\radius 2\lambda Q^\tau\leq 2^{-83}\radius Q^\tau_*.\end{split}\end{equation*} Thus, $f(t)\in P_{2Q^\tau_*}$ for all $t\in\Domain(\sigma)$. However, the endpoints $\Start(\tau),\End(\tau)\not\in P_{15Q^\tau_*}$ by Lemma \ref{unnecessary-endpoints}. Therefore, $\Domain(\sigma)\cap\Domain(\tau)\neq\emptyset$ implies $\Domain(\sigma)\subset \Domain(\tau)$.

Secondly, suppose that $\diam Q^\sigma=\diam Q^\tau$ and $Q^\sigma=Q^\tau$. Since the arcs in $\Lambda(\lambda Q^\tau)$ have pairwise disjoint domains (see Definition \ref{def:arcs}), $\Domain(\sigma)\cap\Domain(\tau)$ implies $\sigma=\tau$. Hence the conclusion in this case follows from Lemma \ref{unnecessary-endpoints}.

Finally, suppose that $\diam Q^\sigma=\diam Q^\tau$, but $Q^\sigma\neq Q^\tau$. Let $x_\sigma,y_\sigma$ be given as above; similarly, let $x_\tau$ denote the center of $Q^\tau$ and choose $y_\tau\in 1.00002Q^\tau_*\cap \Image(\tau)\cap P_x$. Using the triangle inequality to form nested balls centered at $x_\sigma$ and $y_\sigma$ and nested balls centered at $x_\tau$ and $y_\tau$, plus the fact that $\radius \Pi_T(B)=\radius B$ for any ball $B$, one can show that \begin{equation}\label{compare-cylinders} P_{15Q^\tau_*}^\pm\subset P_{(15-2.00004)Q^\sigma_*}^\pm\subset P_{12Q^\sigma_*}^\pm\quad\text{and}\quad P_{15Q^\sigma_*}^\pm\subset P_{(15-2.00004)Q^\tau_*}^\pm\subset P_{12 Q^\tau_*}^\pm.\end{equation} Let $[c_\sigma,d_\sigma]$ and $[c_\tau,d_\tau]$ denote the domains of $\sigma$ and $\tau$, respectively. If it happens that $[c_\sigma,d_\sigma]\subset [c_\tau,d_\tau]$ or $[c_\tau,d_\tau]\subset [c_\sigma,d_\sigma]$ or $d_\sigma=c_\tau$ or $d_\tau=c_\sigma$, then the conclusion follows immediately from Lemma \ref{unnecessary-endpoints} and \eqref{compare-cylinders}. Thus, without loss of generality, we may focus on the case that $c=c_\sigma<c_\tau<d_\sigma<d_\tau=d$ and $\Start(\sigma),\End(\sigma)\in P_{15Q^\sigma_*}^-\subset P_{12Q^\sigma_*}^-$. Suppose to reach a contradiction that $\Start(\tau),\End(\tau)\in P_{15Q^\tau_*}^+\subset P_{12Q^\sigma_*}^+$. We will show that this violates the antislope estimate in Lemma \ref{unnecessary-endpoints}. Since $\diam Q^\sigma=\diam Q^\tau$, but $Q^\sigma\neq Q^\tau$, the centers of the balls are far apart: $|x_\sigma-x_\tau|\geq 2^{-k}$, where $k\in\ZZ$ is the unique integer determined by $Q^\sigma_*=B(x_\sigma, 2^{-12-k})$. Since $|x_\sigma-y_\sigma|\leq 1.00002\cdot 2^{-12-k}$ and $|x_\tau-y_\tau|\leq 1.00002\cdot 2^{-12-k}$, the triangle inequality gives $|y_\sigma-y_\tau| \geq (1-1.00002\cdot 2^{-11})2^{-k}$.

To continue, write $$[c,d]=\underbrace{[c_\sigma,c_\tau]}_{I_1}\cup\underbrace{[c_\tau,d_\sigma]}_{I_2}\cup\underbrace{[d_\sigma,d_\tau]}_{I_3}.$$ Choose $t_\sigma\in\Domain(\sigma)$ and $t_\tau\in\Domain(\tau)$ such that $f(t_\sigma)=y_\sigma$ and $f(t_\tau)=y_\tau$. There are three (sub) cases, depending on which of the intervals $I_1,I_2,I_3$ contain $t_\sigma$ and $t_\tau$.

\emph{Case 1. Assume that $t_\sigma,t_\tau\in I_1\cup I_2=[c_\sigma,d_\sigma]=\Domain(\sigma)$.} Choose a line $L$ such that \eqref{almost-flat-line-estimate} holds for $\sigma$ and let $\Pi_L$ be any $J$-projection onto $L$. Since $y_\sigma,y_\tau\in\Image(\sigma)$, their projections $w_\sigma:=\Pi_L(y_\sigma)$ and $w_\tau:=\Pi_L(y_\tau)$ satisfy $\max\{|w_\sigma-y_\sigma|,|w_\tau-y_\tau|\}\leq 2^{-49-k}$ by \eqref{almost-flat-line-estimate} and \eqref{Pi-x}. Hence the estimate on $|y_\sigma-y_\tau|$ from above and the triangle inequality yields $|w_\sigma-w_\tau|\geq (1-2^{-10})2^{-k}$. Since $\Pi_T(y_\sigma)=x=\Pi_T(y_\tau)$ and $\Pi_T$ is 1-Lipschitz, we also have $|\Pi_T(w_\sigma)-\Pi_T(w_\tau)|\leq |w_\sigma-y_\sigma|+|w_\tau-y_\tau|\leq 2^{-48-k}$. It follows that $$\mathsf{as}(L,\Pi_T) = \frac{|\Pi_T(w_\sigma)-\Pi_T(w_\tau)|}{|w_\sigma-w_\tau|} \leq \frac{2^{-48-k}}{(1-2^{-10})2^{-k}} < 0.00000\,00000\,00004.$$ This (radically!) contradicts the antislope estimate for $L$ from Lemma \ref{unnecessary-endpoints}.

\emph{Case 2. Assume that $t_\sigma,t_\tau\in I_2\cup I_3=[c_\tau,d_\tau]=\Domain(\tau)$.} Repeat the argument from Case 1 using an approximating line $L$ for $\tau$ instead of an approximating line $L$ for $\sigma$.

\emph{Case 3. Assume that $t_\sigma\in I_1$ and $t_\tau\in I_3$.} By our supposition above, $I_2=[c_\tau,d_\sigma]$ satisfies $f(c_\tau)\in P_{12Q^\sigma_*}^+$ and $f(d_\sigma)\in P_{12Q^\sigma_*}^-$. Because $x\in P_{1.00002Q^\sigma_*}$ and $\Pi_T\circ f$ is continuous, the intermediate value theorem produces $t'\in (c_\tau,d_\sigma)\subset \Domain(\sigma)\cap\Domain(\tau)$ such that $\Pi_T(f(t'))=x$. Write $y':=f(t')\in\Image(\sigma)\cap\Image(\tau)$. Because  $|y_\sigma-y_\tau|>0.98\cdot 2^{-k}$, the metric pigeon hole principle implies that $|y_\sigma-y'|>0.49\cdot 2^{-k}$ or $|y_\tau-y'|>0.49\cdot 2^{-k}$, say without loss of generality that $|y_\sigma-y'|>0.49\cdot 2^{-k}$. As in Case 1, choose any line $L$ such that \eqref{almost-flat-line-estimate} holds for $\sigma$ and let $\Pi_L$ be any $J$-projection onto $L$. Since $y'\in\Image(\sigma)$, its projection $w':=\Pi_L(y')$ satisfies $|w'-y'|\leq 2^{-49-k}$. Hence $$|w_\sigma-w'|\geq |y_\sigma-y'|-|w_\sigma-y_\sigma|-|w'-y'|>0.48\cdot 2^{-k}.$$ Since $\Pi_T(y_\sigma)=x=\Pi_T(y')$, we again find that $|\Pi_T(w_\sigma)-\Pi_T(w')|\leq |w_\sigma-y_\sigma|+|w'-y'|\leq 2^{-48-k}$. This time it follows that $$\mathsf{as}(L,\Pi_T) = \frac{|\Pi_T(w_\sigma)-\Pi_T(w')|}{|w_\sigma-w'|} < \frac{2^{-48-k}}{0.48\cdot 2^{-k}} < 0.00000\,00000\,00008.$$ This (again!) contradicts the antislope estimate for $L$ from Lemma \ref{unnecessary-endpoints}.
\end{proof}

\begin{remark} In Lemma \ref{unnecessary-overlaps}, intersection of $\Domain(\sigma)$ and $\Domain(\tau)$ in the case $\diam Q^\sigma=\diam Q^\tau$, but $Q^\sigma\neq Q^\tau$ is possible. For example, consider $\XX=\ell^2_\infty=(\RR^2,|\cdot|_\infty)$, $L_T$ horizontal, $\Pi_T$ the vertical projection onto $L_T$, and stack two squares $2\lambda Q^\sigma$ and $2\lambda Q^\tau$ whose centers lie on a common vertical line $P_x$ with $x\in L_T$. Then one can easily draw a picture where $U_{Q^\sigma},U_{Q^\tau}\not\in\mathcal{N}(T)$ and $\End(\sigma)=\Start(\tau)\in\partial(2\lambda Q^\sigma)\cap\partial (2\lambda Q^\tau)$.\end{remark}

\section{Necessary and sufficient cores}\label{sec: N'}

Imagine (or see \S\ref{proof-I}) that you want to ``pay for'' $\diam T=\Haus^1(\Pi_{T}(T))$ for some \hyperref[def:subarc]{efficient subarc} $T\subset T'\in\Gamma^*_{U_Q}$ using $\ell(R_Q)$ and $\{\diam H_{Q''}:Q''\in\Child(Q)\}$. The length $\ell(R_Q)$ pays for $\Haus^1(\Pi_T(R_Q))$, because $\Pi_T$ is 1-Lipschitz. We will pay for the remaining balance $\Haus^1(\Pi_T(T)\setminus\Pi_T(R_Q))$ in installments. Loosely speaking, given a point $x\in \Pi_{T}(T)\setminus \Pi_T(R_Q)$, if we can locate a core $U_{Q'}\in\mathcal{N}_1(T)\cup\mathcal{N}_2(T)$ whose shadow $\Pi_{T}(U_{Q'})$ contains $x$, then we can use Lemma \ref{N1-auxiliary}, \ref{N21-auxiliary}, or \ref{N22-auxiliary} to pay for $\Haus^1(\Pi_T(D_{Q'}))$ using $\{\diam H_{Q''}:U_{Q''}\cap F_{Q'}\neq\emptyset\}$. A worry that we might have is that there exists an exceptional point $x\in \Pi_T(T)\setminus \Pi_T(R_Q)$, which is not contained in the shadow of a core in $\mathcal{N}_1(T)\cup\mathcal{N}_2(T)$. Another concern is that some core $U_{Q''}$ intersecting $F_{Q'}$ could have $\diam Q''>\diam Q'$, in which case $U_{Q''}\not\subset E_{Q'}$. This section ensures that we can effectively ignore these situations.

\begin{definition}\label{def:locally-maximal} Let $Q\in\mathscr{G}$ and let $T\subset T'\in\Gamma^*_{U_Q}$ be an efficient subarc. We say that a core $U_{Q'}\in\mathcal{N}(T)=\mathcal{N}_1(T)\cup\mathcal{N}_2(T)$ is \emph{locally maximal} if $Q''\in\Child(Q)\setminus\{Q'\}$ and $U_{Q''}\cap 16Q'_*\neq\emptyset$ implies $\diam Q''<\diam Q'$.\end{definition}

\begin{remark}\label{top-level-is-locally-maximal} Every core $U_{Q'}\in\mathcal{N}(T)$ with $\diam Q'=2^{-KM}\diam Q$ is locally maximal by Remark \ref{core-separation} and the fact that there do not exist $Q''\in\Child(Q)$ with $\diam Q''>\diam Q'$.\end{remark}

\begin{lemma}\label{locally-maximal-M's} Let $Q\in\mathscr{G}$ and let $T\subset T'\in\Gamma^*_{U_Q}$ be an efficient subarc. If $U_{Q'}\in\mathcal{N}(T)$ is locally maximal, then $Q''\in\Child(Q)$ and $U_{Q''}\cap F_{Q'}\neq\emptyset$ implies $U_{Q''}\subset E_{Q'}$. In particular, if $U_{Q'}\in\mathcal{N}_1(T)\cup\mathcal{N}_{2.2}(T)$ is locally maximal, then $\bigcup\mathcal{M}_{Q'}\subset E_{Q'}$, where $\mathcal{M}_{Q'}$ is the set of auxiliary cores defined in Lemma \ref{N1-auxiliary} / \ref{N22-auxiliary}.\end{lemma}

\begin{proof} If $U_{Q'}\in\mathcal{N}_{2.1}(T)$, then $F_{Q'}=E_{Q'}=U_{Q'}$ and the conclusion follows since the cores $\{U_{Q''}:Q''\in \Child(Q)\}$ are pairwise disjoint. Thus, suppose that $U_{Q'}\in\mathcal{N}_1(T)\cup\mathcal{N}_{2.2}(T)$ is locally maximal, $Q''\in\Child(Q)$, and $U_{Q''}\cap F_{Q'}\neq\emptyset$. Since $F_{Q'}\subset 16Q'_*$ and $U_{Q'}$ is locally maximal, either $U_{Q''}=U_{Q'}$ or $\diam Q''<\diam Q'$. In the former case, we have $U_{Q''}=U_{Q'}\subset E_{Q'}$ trivially by definition of $E_{Q'}$. In the latter case, $$\diam U_{Q''}\leq 1.00001\diam Q''_* \leq 2^{1-KM}\diam Q'_*\leq 2^{-99}\diam Q'_*.$$ When $U_{Q'}\in\mathcal{N}_1(T)$, it easily follows that $U_{Q''}$ intersecting $F_{Q'}=P_{1.02Q'_*}\cap 3.98Q'_*$ implies $U_{Q''}\subset P_{1.03Q'_*}\cap 3.99Q'_*=E_{Q'}$. Similarly, when $U_{Q'}\in\mathcal{N}_{2.2}(T)$, it follows that $U_{Q''}$ intersecting $F_{Q'}=15.98Q'_*$ implies $U_{Q''}\subset 15.99Q'_*=E_{Q'}$.\end{proof}

\begin{lemma}\label{location-of-E's} Let $Q\in\mathscr{G}$, let $T\subset T'\in\Gamma^*_{U_Q}$ be an efficient subarc, and let $r_T$ be given by \eqref{e:r_T def}. For all $U_{Q'}\in\mathcal{N}(T)$, the neighborhood $E_{Q'}\subset B_{9r_T}(T)$. Moreover, $E_{Q'}\cap 1.99\lambda Q''=\emptyset$ for all $Q''\in\Child(Q)$ such that $\diam Q'\leq \diam Q''$ and $\Pi_T(16Q'_*)\cap \big(L_T\setminus \Pi_T(2\lambda Q'')\big)\neq\emptyset$.\end{lemma}

\begin{proof} Let $U_{Q'}\in\mathcal{N}(T)$. Then $T\cap 1.00002Q'_*\neq\emptyset$; choose any point $y$ in the intersection. Letting $x'$ denote the center of $Q'$, we have $|x'-y|\leq 1.00002\radius Q'_*$. Let $x\in E_{Q'}\subset 15.99Q'_*$. Then $|x-y|\leq |x-x'|+|x'-y| \leq 16.99002\radius Q'_*\leq 8.49501\diam Q'_*<9r_T$. Hence $E_{Q'}\subset B_{9r_T}(T)$ with room to spare.

Let $Q''\in\Child(Q)$ and suppose that $\diam Q'\leq \diam Q''$ and $\Pi_T(16Q'_*)$ intersects the complement of $\Pi_T(2\lambda Q'')$. Then $16Q'_*\cap(\XX\setminus 2\lambda Q'')\neq\emptyset$, as well. Note that $$\gap(\XX\setminus 2\lambda Q'',1.99\lambda Q'') \geq 0.01\radius Q''\geq 20.48\diam Q''_*\geq 20.48\diam Q'_*.$$ Thus, $\gap(E_{Q'},1.99\lambda Q'')\geq \gap(16Q'_*,1.99\lambda Q'')\geq \gap(\XX\setminus 2\lambda Q'',1.99\lambda Q'')-\diam 16Q'_*\geq 4.48\diam Q'_*>0.$ Therefore, $E_{Q'}$ does not intersect $1.99\lambda Q''$.
\end{proof}

\begin{lemma}\label{disjoint-E's} Let $Q\in\mathscr{G}$ and let $T\subset T'\in\Gamma^*_{U_Q}$ be an efficient subarc. If $U_{Q'},U_{Q''}\in\mathcal{N}(T)$, $\diam Q'<\diam Q''$, and $\Pi_T(D_{Q'})$ intersects $L_T\setminus\Pi_T(D_{Q''})$, then $D_{Q'}\cap E_{Q''}=\emptyset$. Also, $D_{Q'}\cap D_{Q'''}=\emptyset$ for all $U_{Q'''}\in\mathcal{N}(T)\setminus\{U_{Q'}\}$ such that $\diam Q'''=\diam Q'$.\end{lemma}

\begin{proof} Under the hypotheses of the lemma, $\diam D_{Q'}\leq 16 Q'_*\leq 2^{-96}\diam Q''_*$ and $D_{Q'}$ intersects $\XX\setminus D_{Q''}$. Reviewing Definitions \ref{d:E, E-star N1 def} / \ref{d:E, E-star N2 def}, we further know $\gap(\XX\setminus D_{Q''},E_{Q''})\geq 0.00001\diam Q''_*$. Therefore, \begin{equation*}\gap(D_{Q'},E_{Q''})\geq \gap(\XX\setminus D_{Q''},E_{Q''})-\diam D_{Q'}\geq (0.00001-2^{-96})\diam Q''_*>0.\end{equation*} If $D_{Q'''}\in\mathcal{N}(T)\setminus\{U_{Q'}\}$ and $\diam Q'''=\diam Q'$, then $D_{Q'}\cap D_{Q'''}=\emptyset$ by Remark \ref{core-separation}.
\end{proof}

\begin{lemma}\label{not-locally-maximal-are-contained} Let $Q\in\mathscr{G}$ and let $T\subset T'\in\Gamma^*_{U_Q}$ be an efficient subarc. If  $U_{Q'}\in\mathcal{N}(T)$ is not locally maximal, then $16Q'_*\subset 1.00002Q''_*$ for some $U_{Q''}\in\mathcal{N}(T)$ that is locally maximal or for some $U_{Q''}\not\in\mathcal{N}(T)$.\end{lemma}

\begin{proof} Assume that $U_{Q^1}\in\mathcal{N}(T)$ is not locally maximal. Then there exists $Q^2\in\Child(Q)$ with $\diam Q^2>\diam Q^1$ such that $U_{Q^2}\cap 16Q^1_*\neq\emptyset$. Let $x_1$ and $x_2$ denote the centers of $Q^1$ and $Q^2$, respectively, and choose $w_1\in U_{Q^2}\cap 16Q^1_*\subset 1.00001Q^2_*\cap 16Q^1_*$. We have \begin{equation}\label{nlm1}
|x_1-x_2| \leq |x_1-w_1|+|w_1-x_2| \leq 16\radius Q^1_*+1.00001\radius Q^2_*.\end{equation} Since $\radius Q^1_*\leq 2^{-100}\radius Q^2_*$, it follows that for all $z\in 16Q^1_*$, \begin{align*} |z-x_2| & \leq 32\radius Q^1_*+1.00001\radius Q^2_*\leq (2^{-95}+1.00001)\radius Q^2_*.\end{align*} Hence $16Q^1_*\subset 1.00002Q^2_*$. If perchance either $U_{Q^2}\in\mathcal{N}(T)$ and $U_{Q^2}$ is locally maximal or $U_{Q^2}\not\in\mathcal{N}(T)$, then we are done. The other possibility is that $U_{Q^2}\in\mathcal{N}(T)$ and $U_{Q^2}$ is not locally maximal and we repeat the argument.

Suppose that for some $j\geq 3$ we have found cores $U_{Q^1},\cdots U_{Q^{j-1}}\in\mathcal{N}(T)$, each of which is not locally maximal, such that \begin{equation}\label{nlm-intersections}\diam Q^i>\diam Q^{i-1}\text{ and }U_{Q^i}\cap 16Q^{i-1}_*\neq\emptyset\quad\text{for all }2 \leq i\leq j-1,\end{equation} and such that the centers $x_1,\dots,x_{j-1}$ of the balls $Q^1,\dots,Q^{j-1}$ satisfy \begin{equation}\label{nlm-centers} |x_{i-1}-x_i| \leq 16\radius Q^{i-1}_* + 1.00001\radius Q^i_*\quad\text{for all }2\leq i\leq j-1.\end{equation} Since $Q^{j-1}$ is not locally maximal, $\diam Q^{j-1}\leq 2^{-2KM}\diam Q$ by Remark \ref{top-level-is-locally-maximal} and there exists $Q^j\in\Child(Q)$ such that $\diam Q^j>\diam Q^{j-1}$ and $U_{Q^j}\cap 16Q^{j-1}_*\neq\emptyset$. Let $x_j$ denote the center of $Q^j$ and choose $w_{j-1}\in U_{Q^j}\cap 16Q^{j-1}_*\subset 1.00001Q^j_*\cap 16Q^{j-1}_*$. Then \begin{equation}\label{nlmj} |x_{j-1}-x_j| \leq |x_{j-1}-w_{j-1}|+|w_{j-1}-x_j| \leq 16\radius Q^{j-1}_*+1.00001\radius Q^j_*.\end{equation} Thus, \eqref{nlm-intersections} and \eqref{nlm-centers} also hold when $i=j$. Let $z\in 16Q^1_*$ and write $|z-x_j|\leq |z-x_1|+|x_1-x_2|+\cdots+|x_{j-1}-x_j|$. Since $\radius Q^{i-1}_*\leq 2^{-100}\radius Q^i_*$ for all $2\leq i\leq j$, we get \begin{equation}\begin{split}\label{nlm-telescope} |z-x_j| &\leq 16\radius Q^1_*+17.00001\left(\sum_{i=1}^{j-1} \radius Q^i_*\right)+1.00001\radius Q^j_*\\ &<(2^{-93}+1.00001)\radius Q^j_*.\end{split}\end{equation} Hence $16Q^1_*\subset 1.00002Q^j_*$. Once again, if either $U_{Q^j}\in\mathcal{N}(T)$ and $U_{Q^j}$ is locally maximal, or $U_{Q^j}\not\in\mathcal{N}(T)$, then we are done. Otherwise, $U_{Q^j}\in\mathcal{N}(T)$ and $U_{Q^j}$ is not locally maximal and we go to the next step of the induction. The iterative scheme eventually terminates after finitely many steps by Remark \ref{top-level-is-locally-maximal}.
\end{proof}

\begin{definition}\label{def:sufficient} Let $Q\in\mathscr{G}$ and let $T\subset T'\in\Gamma^*_{U_Q}$ be an efficient subarc. We say that  $U_{Q'}\in\mathcal{N}(T)$ is \emph{sufficient} if $U_{Q'}$ is \hyperref[def:locally-maximal]{locally maximal} or if $16Q'_*\subset 1.00002Q''_*$ for some locally maximal $U_{Q''}\in\mathcal{N}(T)$. Let $\mathcal{S}(T)\subset\mathcal{N}(T)$ denote the set of all sufficient cores. \end{definition}

The proof of the following lemma is ultimately a topological argument, which follows from our assumption that the parameterization $f:[0,1]\rightarrow\Gamma$ is continuous. (Furthermore, the proof invokes Lemma \ref{unnecessary-overlaps}, which also exploited the continuity of $f$.)

\begin{figure}\begin{center}\includegraphics[width=.8\textwidth]{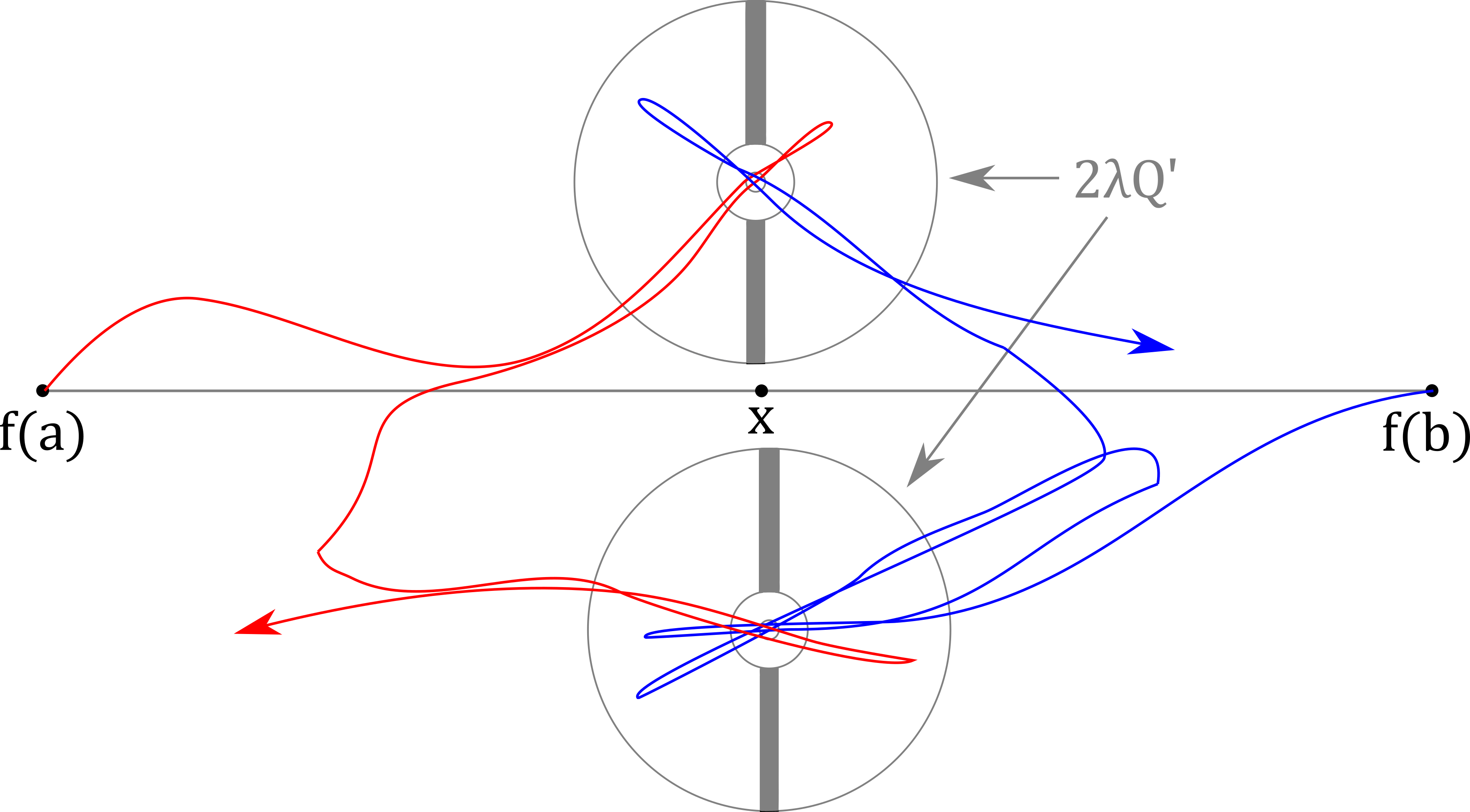}\end{center}\caption[Topological argument]{Proof of Lemma \ref{topological lemma} (simplified): If no core in $\mathcal{N}_1(T)$ or $\mathcal{N}_2(T)$ intersects $T$ above $x\in\Pi_T(T)\setminus\Pi_T(R_Q)$, with $x$ far away from the endpoints of $T$, then it is impossible to reach $f(b)$ from $f(a)$.}\label{fig:topo}\end{figure}

\begin{lemma}[topological lemma] \label{topological lemma}
Let $Q\in\mathscr{G}$ and let $T=f([a,b]) \subset T'\in\Gamma^*_{U_Q}$ be an \hyperref[def:subarc]{efficient subarc}. Define $\rho_T$ by \eqref{e:r_T def}; that is, let $\rho_T$ be largest diameter of a ball $2\lambda Q''$ among all $Q''\in\Child(Q)$ such that $1.00002Q''_*\cap T\neq\emptyset$. For all points $x$ such that
\begin{equation}\label{x-far-away} x \in \Pi_T(T)\setminus\Big( \Pi_T(R_Q\cap T)\cup B_{0.51\rho_T}(\{f(a),f(b)\})\Big),\end{equation} there exists $U_{Q'}\in\mathcal{S}(T)$ such that $x \in \Pi_{T}(U_{Q'} \cap T)$.
\end{lemma}

\begin{proof}Let $x$ satisfying \eqref{x-far-away} be given. Following the convention in Remark \ref{general projections}, $f(a)\in P_{\{x\}}^-$ and $f(b)\in P_{\{x\}}^+$. For simplicity, we shall write $P_x$ and $P^\pm_x$ instead of $P_{\{x\}}$ and $P^\pm_{\{x\}}$. Consider the set $\mathcal{U}:=\{U_{Q''}:Q''\in\Child(Q), U_{Q''}\cap T\cap P_{x}\neq\emptyset\}$ of cores that intersect $T$ and whose shadows contain $x$. Our assumption that $x\in\Pi_T(T)\setminus \Pi_T(R_Q\cap T)$ guarantees that $\mathcal{U}$ is nonempty and $\emptyset\neq T\cap P_x\subset \bigcup_{U_{Q''}\in\mathcal{U}} U_{Q''}$. Suppose for the sake of contradiction that no core $U_{Q''}\in\mathcal{U}$ belongs to $\mathcal{S}(T)$. Then, by Lemma \ref{not-locally-maximal-are-contained}, for all $U_{Q''}\in\mathcal{U}$, there exists at least one core $U_{Q'}$ in $$\mathcal{O}:=\{U_{Q'}\not\in\mathcal{N}(T):Q'\in\Child(Q),1.00002Q'_*\cap T\cap P_x\neq\emptyset\}$$ such that $U_{Q''}\subset 16Q''_*\subset 1.00002Q'_*$. Hence $\mathcal{O}\neq\emptyset$ and \begin{equation}\label{O-cover} T\cap P_x\subset \bigcup_{U_{Q'}\in\mathcal{O}}1.00002Q'_*.\end{equation} Further,
our assumption that $x\in \XX\setminus B_{0.51\rho_T}(\{f(a),f(b)\})$ ensures that if $U_{Q'}\in\mathcal{O}$, then $2\lambda Q'\cap\{f(a),f(b)\}=\emptyset$. Indeed, given $U_{Q'}\in\mathcal{O}$, let $x'$ denote the center of $Q'$ and pick $y'\in 1.00002Q'_*\cap T\cap P_x$. Since $\Pi_T$ is 1-Lipschitz, $$|x-\Pi_T(x')|\leq |y'-x'|\leq \radius 1.00002Q'_*<2^{-12}\radius 2\lambda Q'\leq 2^{-13}\rho_T.$$ Using the fact that $\Pi_T$ is 1-Lipschitz once more and the fact that $\Pi_T$ fixes $f(a)$ and $f(b)$, we find that \begin{align*}\dist(x',\{f(a),f(b)\})&\geq \dist(\Pi_T(x'),\{f(a),f(b)\}) \geq \dist(x,\{f(a),f(b)\})-|x-\Pi_T(x')|\\ &>(0.51-2^{-13})\rho_T>0.509\rho_T\geq 0.009\rho_T+\radius 2\lambda Q'.\end{align*} That is, $\{f(a),f(b)\}$ does not intersect an open tubular neighborhood of $2\lambda Q'$ of width $0.009\rho_T$. As a corollary, since $f$ is uniformly continuous, there exists $\delta>0$ depending on $\rho_T$ and the modulus of continuity of $f$ such that \begin{equation}\label{domain-gap}\Domain(\tau)\cap\big([a,a+\delta)\cup(b-\delta,b]\big)=\emptyset\quad\text{for every arc }\tau\in S^*(\lambda Q').\end{equation} (Below we only need to know that $\delta>0$.)
To proceed, define collections of arcs and an associated collection of intervals by \begin{align}
\mathcal{A}_y&:=\{\tau\in S(\lambda Q'):U_{Q'}\in\mathcal{O}, y\in\Image(\tau), \Domain(\tau)\subset[a,b]\}\quad\forall y\in T\cap P_x,\\
\mathcal{I}&:=\{\text{connected components $I$ of }\bigcup\{\Domain(\tau): y\in T\cap P_x\text{ and }\tau\in A_y\}\}.\end{align} By definition, $\mathcal{I}$ is pairwise disjoint, and by \eqref{O-cover} and \eqref{domain-gap}, we have $T\cap P_x\subset \bigcup_{I\in\mathcal{I}} f(I)$ and $I\subset[a+\delta,b-\delta]$ for all $I\in\mathcal{I}$. By Lemma \ref{unnecessary-endpoints} and Lemma \ref{unnecessary-overlaps}, for each interval $I=[c,d]\in\mathcal{I}$, either $\Pi_T(f(c)),\Pi_T(f(d))<x$ or $\Pi_T(f(c)),\Pi_T(f(d))>x$, where we identify $[f(a),f(b)]$ with an isometric subset of $\RR$. Modulo applying continuous reparameterizations to the domain and image of the continuous map $\Pi_T\circ f:[a,b]\rightarrow[f(a),f(b)]$, we have built a function $g:[0,1]\rightarrow[0,1]$ such that \begin{quotation}$(\star)$: $g$ is continuous, $g(0)=0$, $g(1)=1$, and there exists a pairwise disjoint collection $\mathcal{J}$ of non-degenerate closed subintervals of $[1/4,3/4]$ such that the preimage $g^{-1}(1/2)\subset\bigcup\mathcal{J}$ and for all intervals $J=[c,d]\in\mathcal{J}$, either $g(c),g(d)<1/2$ or $g(c),g(d)>1/2$.\end{quotation} (Explicitly, send $a\mapsto 0$, $b\mapsto 1$, $a+\delta\mapsto 1/4$, $b-\delta\mapsto 3/4$, $f(a)\mapsto 0$, $f(b)\mapsto 1$, $x\mapsto 1/2$.) By the next lemma, no such function exists. Therefore, our supposition was false, and there exists $U_{Q'}\in \mathcal{S}(T)$ such that $x\in \Pi_T(U_{Q'}\cap T)$.\end{proof}

\begin{lemma} A function $g:[0,1]\rightarrow[0,1]$ with property $(\star)$ does not exist.\end{lemma}

\begin{proof} Suppose that $g$ exists. Let $\mathcal{I}$ denote the connected components of $[0,1]\setminus\bigcup_{J\in\mathcal{J}}J$. Label each interval $I\in\mathcal{I}$ as \emph{left-directed} or \emph{right-directed} depending on whether there is an interval $J=[c,d]\in\mathcal{J}$ such that $\overline{I}\cap J\neq\emptyset$ and $g(c),g(d)<1/2$ \emph{or} $g(c),g(d)>1/2$, respectively. This concept is well-defined by property $(\star)$, in particular by continuity of $g$ and by the stated properties of $\mathcal{J}$. The unique half-open interval of the form $[0,b)\in\mathcal{I}$ is left-directed, because $g(0)=0$; the unique half-open interval of the form $(a,1]\in\mathcal{I}$ is right-directed, because $g(1)=1$. All other intervals in $\mathcal{I}$ are open intervals $(a,b)$ with $g(t)<1/2$ for all $t\in(a,b)$, if $(a,b)$ is left-directed, and $g(t)>1/2$ for all $t\in(a,b)$, if $(a,b)$ is right-directed. The only restriction on values of $g(t)$ for $t\in[c,d]\in\mathcal{J}$ are at the endpoints $t=c$ and $t=d$.

Let $L:=\{t\in[0,1]:t\in I\text{ for some left-directed interval }I\in\mathcal{I}\}$ and let $u:=\sup L$. Then $g(u)\leq 1/2$, and so, $u$ is not contained in a right-directed interval of $\mathcal{I}$. Let's consider the other two possibilities. First, suppose that $u\in I$ for some left-directed $I\in\mathcal{I}$. Since every left-directed interval is open to the right (as the interval containing $1$ is right-directed), this would mean that $u$ cannot be an upper bound on $L$, which is absurd. Next, suppose that $u\in J$ for some $J\in\mathcal{J}$. Then $J=[u,v]$ for some $u<v$ and $g(u)<1/2$. (This used the approximation property of the supremum.) Let $I'=(v,d)\in\mathcal{I}$ be the interval lying immediately to the right of $J$. The interval $I'$ must exist, since the interval containing $1$ belongs to $\mathcal{I}$. Since $v>u$, $I'$ must be right-directed. Hence $g(v)>1/2$. Thus, $g(u)>1/2$, because $[u,v]\in\mathcal{J}$. This contradicts our observation that $g(u)\leq 1/2$. Therefore, there does not exist a function $g:[0,1]\rightarrow[0,1]$ with property $(\star)$.\end{proof}

\section{Proof of Lemma \ref{improve}}\label{proof-I}

\subsection{Stage 1: improving the coarse estimate}

\begin{lemma}[initial improvement of \eqref{e:coarse estimate}] \label{43-estimate} With notation as in Lemma \ref{improve},
\begin{equation}\begin{split} \label{43-improve}
   \diam T-2\rho_T\leq 1.7\,\ell(R_Q&\cap B_{9r_T}(T)) + \sum_{\mathclap{U_{Q''}\in \mathcal{F}}} \diam 2\lambda Q'' + 1.37 \sum_{\mathclap{U_{Q'}\subset B_{9r_T}(T),\;U_{Q'}\not\in\mathcal{N}_\mathcal{F}}} \diam H_{Q'}.
\end{split}\end{equation}\end{lemma}

\begin{proof} Since $T$ is an \hyperref[def:subarc]{efficient subarc}, $\Pi_T(T)=[f(a),f(b)]$. To start, let \begin{equation}J_0=[f(a)+\rho_T,f(b)-\rho_T]\setminus \Big(\Pi_T(R_Q \cap T) \cup \Pi_T(\textstyle\bigcup_{\mathcal{F}}2\lambda Q'')\Big).\end{equation} By subadditivity of measures and the fact that $\Pi_T$ is 1-Lipschitz, \begin{equation}\begin{split} \label{43-initial} \diam T-2\rho_T & \leq \Haus^1(\Pi_T(R_Q \cap T)) + \Haus^1(\Pi_T(\textstyle\bigcup_{\mathcal{F}}2\lambda Q'')) + \Haus^1(J_0) \\ & \leq \ell(R_Q \cap T) + \sum_{\mathcal{F}} \diam 2\lambda Q'' + \Haus^1(J_0).\end{split}\end{equation}  We shall reach \eqref{43-improve} from \eqref{43-initial} by making a sequence of refined estimates on $\Haus^1(J_0)$. More precisely, we inductively define measurable\footnote{If $\XX$ is not separable, pass to a separable subspace of $\XX$ containing the rectifiable curve $\Gamma$ before defining $J_0$ to ensure the projection $\Pi_T(R_Q)$ is universally measurable.} sets $J_0\supset J_1\supset J_2\supset \cdots$ with $\bigcap_{i=0}^\infty J_i=\emptyset$ and ``pay for'' $\Haus^1(J_{i-1}\setminus J_{i})$ for each $i\geq 1$ using a Borel subset $R_i$ of the remainder set $R_Q$ and certain cores $\mathcal{M}_i$ in lying in $B_{9r_T}(T)$. In particular, we will prove that \begin{equation}\label{43-intermediate} \Haus^1(J_{i-1}\setminus J_{i}) \leq 0.7\,\ell(R_i) + 1.37\sum_{\mathclap{U_{Q'}\in \mathcal{M}_i}} \diam H_{Q'}.\end{equation} Naturally, we will arrange things so that $R_i\cap R_j=\emptyset$ and $\mathcal{M}_i\cap\mathcal{M}_j=\emptyset$ for all $i\neq j$. Further, the cores in $\mathcal{M}_i$ will not belong to $\mathcal{N}_\mathcal{F}$, the set of all cores $U_{Q'}$ with $Q'\in\Child(Q)$ such that $U_{Q'}\subset 1.99\lambda Q''$ for some $U_{Q''}\in\mathcal{F}$. Thus, \eqref{43-improve} follows immediately by combining \eqref{43-initial} and \eqref{43-intermediate}.

Let $\mathcal{S}(T)$ be given as in Definition \ref{def:sufficient}. For each $i\geq 1$, inductively define \begin{align}\label{next-S} \mathcal{S}_i &:=\{U_{Q'}\in\mathcal{S}(T):\diam Q'=2^{-KMi}\diam Q,\,\Pi_T(U_{Q'})\cap J_{i-1}\neq\emptyset\},\\
J_i&:=J_{i-1}\setminus \bigcup_{U_{Q'}\in\mathcal{S}_i} \Pi_T(D_{Q'}).\end{align} By Lemma \ref{topological lemma}, every $x\in J_0$ lies in the shadow $\Pi_T(U_{Q'})$ of some core $U_{Q'}\in\mathcal{S}(T)$. Hence $\bigcap_{i=0}^\infty J_i=\emptyset$. Every core $U_{Q'}\in\mathcal{S}_i$ ($i\geq 1$) is locally maximal (see Definitions \ref{def:locally-maximal} and \ref{def:sufficient}), because $\Pi_T(U_{Q'})\cap J_{i-1}\neq\emptyset$ implies that $16Q'_*\not\subset 1.00002Q''_*$ for any locally maximal core $U_{Q''}\in\mathcal{N}(T)$ with $\diam Q''>\diam Q'$. Indeed, the shadows $\Pi_T(D_{Q''})\supset \Pi_T(1.00002Q''_*)$ of all locally maximal $U_{Q''}\in\mathcal{N}(T)$ with $\diam Q''>\diam Q'$ (which belong to $\mathcal{S}(T)$) were already deleted from $J_0,\dots, J_{i-2}$ in the inductive definition $J_{i-1}$.

Our next task is to bound the length of each set $J_{i-1}\setminus J_{i}$. Fix $i\geq 1$. If $J_{i}=J_{i-1}$, then $\Haus^1(J_{i-1}\setminus J_i)=0$. If $J_{i}\neq J_{i-1}$, then by countable subadditivity of measures, the isodiametric inequality $\Haus^1(A)\leq \diam A$ for all $A\subset \RR$, and the fact that $\Pi_T$ is 1-Lipschitz, \begin{equation}\label{J-diff-Haus-est}
\Haus^1(J_{i-1}\setminus J_{i}) \leq
   \sum_{\mathclap{U_{Q'}\in \mathcal{S}_i}} \diam \Pi_T(D_{Q'}).
\end{equation}
For each $U_{Q'}\in\mathcal{S}_i$, define an auxiliary family of cores $\mathcal{M}_{Q'}$ and Borel set $\hat R_{Q'}$ as follows: \begin{itemize}
\item if $U_{Q'}\in\mathcal{N}_1(T)$, define $\mathcal{M}_{Q'}$ to be the family in Lemma \ref{N1-auxiliary} and $\hat R_{Q'}:=R_Q\cap F_{Q'}$;
\item if $U_{Q'}\in\mathcal{N}_{2.1}(T)$, define $\mathcal{M}_{Q'}:=\{U_{Q'}\}$  and $\hat R_{Q'}:=\emptyset$ (cf.~Lemma \ref{N21-auxiliary}); and,
\item if $U_{Q'}\in\mathcal{N}_{2.2}(T)$, define $\mathcal{M}_{Q'}$ to be the family in Lemma \ref{N22-auxiliary} and $\hat R_{Q'}:=R_Q\cap F_{Q'}$.\end{itemize}
By Lemma \ref{locally-maximal-M's}, the set $M_{Q'}:=\hat R_{Q'}\cup\bigcup\mathcal{M}_{Q'}\subset E_{Q'}$ for all $U_{Q'}\in\mathcal{S}_i$. Furthermore, the set $M_{Q'}\subset B_{9r_T}(T)$ and $\mathcal{M}_{Q'}\cap\mathcal{N}_\mathcal{F}=\emptyset$ by Lemma \ref{location-of-E's} and property \hyperref[property-F]{(F)}. Define \begin{equation}
\mathcal{M}_i := \bigcup_{U_{Q'} \in \mathcal{S}_i} \mathcal{M}_{Q'}\quad\text{and}\quad R_i := \bigcup_{U_{Q'} \in \mathcal{S}_i} \hat R_{Q'}.\end{equation} Then \eqref{43-intermediate} follows immediately from \eqref{J-diff-Haus-est}, the estimates Lemma \ref{N1-auxiliary}, Lemma \ref{N21-auxiliary}, and Lemma \ref{N22-auxiliary}, and the second part of Lemma \ref{disjoint-E's}.

Finally, as required, $\mathcal{M}_i\cap\mathcal{M}_j=\emptyset$ and $R_i\cap R_j=\emptyset$ for all $i\neq j$ by the first part of Lemma \ref{disjoint-E's}.\end{proof}

\subsection{Stage 2: iterating the improved estimate}

\begin{lemma}\label{N22-better} Let $Q\in\mathscr{G}$ and let $T\subset T'\in\Gamma^*_{U_Q}$ be an efficient subarc. If $U_{Q'}\in\mathcal{N}_{2.2}(T)$ is locally maximal, then there is a set $\mathcal{M}_{Q'}$ of cores $U_{Q''}$ with $Q''\in\Child(Q)$ and $U_{Q''} \subset E_{Q'}$ such that \begin{equation}\label{N22-auxliary-cores-better-total}
\diam D_{Q'} < 1.2\,\ell(R_Q\cap E_{Q'}) + 0.95 \sum_{U_{Q''}\in\mathcal{M}_{Q'}}\diam H_{Q''}.\end{equation}
\end{lemma}

\begin{proof} Let $\mathcal{Y}$ be given by Lemma \ref{N22-geometry}. We repeat the proof of Lemma \ref{N22-auxiliary}, but use the improved estimate \eqref{43-improve} with $\mathcal{F}=\emptyset$ instead of the coarse estimate. In effect, we are incorporating the existence of cores $U_{Q''}$ that lie nearby, but do not necessarily intersect the subarcs $Y\in\mathcal{Y}$. By Lemma \ref{N22-geometry} and assumption $U_{Q'}$ is locally maximal, for every subarc $Y\in\mathcal{Y}$, we know that $Y\subset F_{Q'}\setminus Q'_*$, $\diam Y\geq 0.00021\diam Q'_*$, $r_Y\leq 2^{-KM}\diam Q'_*\leq 2^{-100}\diam Q'_*$, $\rho_Y \leq 2\lambda A_\mathscr{H}\cdot 2^{12}r_Y \leq 2^{-84}\diam Q'_*$, and $0.99999\diam Y\leq \diam Y-2\rho_Y$. In addition, $\{1.00002Q'_*\}\cup\{B_{9r_Y}(Y):Y\in\mathcal{Y}\}$ is pairwise disjoint. Since $F_{Q'}=15.98Q'_*$, we easily obtain $B_{9r_Y}(Y)\subset 15.981Q'_*\subset 15.99Q'_*=E_{Q'}$ from the estimate on $r_Y$.

Let $\mathcal{M}_{Q'}=\{U_{Q''}:Q''\in\Child(Q)\text{ and } U_{Q''}\subset E_{Q'}\}.$ Now, $\diam Q'_*\leq 2.00002\diam H_{Q'}$, which implies $0.68499\diam Q'_*\leq 1.37\diam H_{Q'}$. Further, for every $Y\in\mathcal{Y}$, \begin{equation*}
0.99999\diam Y\leq \diam Y - 2\rho_Y \leq 1.7\,\ell(R_Q\cap B_{9r_Y}(Y)) + 1.37 \sum_{\mathclap{U_{Q''}\subset B_{9r_Y(Y)}}}\diam H_{Q''}\end{equation*} by \eqref{43-improve} with $\mathcal{F}=\emptyset$. Also, by Lemma \ref{N22-geometry}, $\sum_{Y\in\mathcal{Y}} 0.99999\diam Y\geq 22.45977\diam Q'_*$. Finally, $B_{9r_Y}(Y)\subset E_{Q'}$. Combining these estimates, we obtain
\begin{align*}
    (0.68499+22.45977)\diam Q'_* & \le 1.7\,\ell(R_Q\cap E_{Q'}) + 1.37 \sum_{\mathclap{U_{Q''}\in\mathcal{M}_{Q'}}} \diam H_{Q''}.
\end{align*}
Since $\diam D_{Q'}=16\diam Q'_*$, this estimate yields \eqref{N22-auxliary-cores-better-total}.
\end{proof}

\begin{proof}[Proof of Lemma \ref{improve}]
Repeat the proof of Lemma \ref{43-estimate}, except use Lemma \ref{N22-better} in place of Lemma \ref{N22-auxiliary}. Instead of \eqref{43-intermediate}, the proof gives \begin{equation}\label{b43-intermediate} \Haus^1(J_{i-1}\setminus J_{i}) \leq 1.2\,\ell(R_i) + 1.00016\sum_{\mathclap{U_{Q'}\in \mathcal{M}_i\cap\mathcal{N}_{2.1}(T)}}\diam H_{Q''}+0.95\sum_{\mathclap{U_{Q'}\in \mathcal{M}_i\setminus\mathcal{N}_{2.1}(T)}} \diam H_{Q'}.\end{equation} Therefore, instead of \eqref{43-improve}, we ultimately obtain \begin{equation}\begin{split} \label{b43-improvement}
   \diam T-2\rho_T\leq 2.2\,\ell(R_Q&\cap B_{9r_T}(T)) + \sum_{\mathclap{U_{Q''}\in \mathcal{F}}} \diam 2\lambda Q'' \\ &+ 1.00016 \sum_{\mathclap{U_{Q'}\in \mathcal{N}_{2.1}(T)\setminus\mathcal{N}_\mathcal{F}}}\diam H_{Q'} + 0.95 \sum_{\mathclap{U_{Q'}\not\in\mathcal{N}_{2.1}(T)\cup\mathcal{N}_\mathcal{F}}} \diam H_{Q'}.
\end{split}\end{equation} where the sums in the second line may be further restricted to $U_{Q'}$ contained in $B_{9r_T}(T)$. Replacing the terms $0.95\sum_{U_{Q'}\in\mathcal{N}_{2.2}(T)\setminus\mathcal{N}_\mathcal{F}}\diam H_{Q'}$ with $1.00016\sum_{U_{Q'}\in\mathcal{N}_{2.2}(T)\setminus\mathcal{N}_\mathcal{F}}\diam H_{Q'}$ yields \eqref{consolidated-improvement}. (The purpose of this last step is to let us avoid defining $\mathcal{N}_{2.1}(T)$ in \S\ref{proof-main}.)\end{proof}

\begin{remark}One could continue to iterate estimates for $\mathcal{N}_{2.2}(T)$ cores to further reduce the coefficient 0.95. However, iteration will never let us improve the coefficient $1.00016$ associated to $\mathcal{N}_{2.1}(T)$ cores.\end{remark}

\section{Proof of Lemma \ref{l:B_2 subset}}\label{proof-II}

Assume for the duration of this section that $Q\in\mathscr{G}$ has small remainder in the sense of Definition \ref{remainder-sizes} and few non-$\mathcal{N}_2(G_Q)$ cores in the sense of \eqref{many-N2-cores}.

\subsection{Existence of $\mathcal{A}$ and proof of \eqref{A-estimates 2}} \label{ss:A}

Because $G_Q=f([a_Q,b_Q])$ satisfies \eqref{H-G diam est}, $Q$ has small remainder, and \eqref{many-N2-cores} holds,
\begin{equation}\begin{split} \label{A construction N_2 estimate}
      \Haus^1\left(\Pi_{G_Q}\left(\textstyle\bigcup\mathcal{N}_2(G_Q)\right)\right)  &\geq \diam G_Q - \ell(R_Q) - \sum_{U_{Q'} \not \in \mathcal{N}_2(G_Q)}\diam U_{Q'}\\
    & \ge (0.99993-0.01-0.05)\diam H_Q = 0.93993 \diam H_Q.
\end{split}\end{equation} (To start, write $\diam G_Q=\Haus^1(\Pi_{G_Q}(G_Q))$. Compare to the derivation of \eqref{43-improve}.)

We will construct $\mathcal{A}$ inductively using a greedy algorithm. To begin, we stratify $\mathcal{N}_2(G_Q)$ by size. For each $i\geq 1$, let $\mathcal{U}_i$ denote the set of all cores $U_{Q'}\in\mathcal{N}_2(G_Q)$ such that $\diam Q' = 2^{-KMi}\diam Q$. Each family $\mathcal{U}_i$ consists of finitely many cores, because $\Gamma$ is compact. Some (but not all) of the families may be empty.

Choose $\mathcal{A}_1$ to be a maximal subset of $\mathcal{U}_{1}$ such that $\{2\lambda Q'':U_{Q''}\in\mathcal{A}_1\}$ is pairwise disjoint. Note that $\mathcal{A}_1$ automatically enjoys property \hyperref[property-F]{(F)} with $T=G_Q$, because there are no $Q'\in\Child(Q)$ with $\diam Q'>\diam Q''$. If $\sum_{U_{Q''} \in \mathcal{A}_1} \diam 2\lambda Q'' \ge 0.04 \diam H_Q$, then we halt and define $\mathcal{A}:= \mathcal{A}_1$. Otherwise, we move to the induction step.

Suppose that we have defined $\mathcal{A}_1\subset\cdots\subset \mathcal{A}_{i-1}$ for some $i\geq 2$ so that  $\mathcal{A}_{i-1}$ satisfies property \hyperref[property-F]{(F)} with $T=G_Q$ and  $\sum_{U_{Q''}\in\mathcal{A}_{i-1}}\diam 2\lambda Q''< 0.04 \diam H_Q$.  Choose a maximal family $\mathcal{A}_i'$ from the collection
\begin{equation}\begin{split}\label{U-subcollection}
\{U_{Q''} &\in \mathcal{U}_i :\ 2\lambda Q'' \cap 2\lambda Q'=\emptyset\text{ for all }U_{Q'}\in \mathcal{A}_{i-1},\text{ and}\\
 &2\lambda Q''\not\subset 16.1Q'_*\text{ when }U_{Q'}\in \Child(Q) \text{ and } \diam Q' > \diam Q''\}
\end{split}\end{equation}
such that $\{2\lambda Q'': U_{Q''} \in \mathcal{A}'_i\}$ is pairwise disjoint. If it happened that $2\lambda Q''\cap 16Q'_*\neq\emptyset$ for some $U_{Q''}\in\mathcal{A}'_i$ and $U_{Q'}\in \Child(Q)$ with $\diam Q'>\diam Q''$, then we would also have $2\lambda Q''\subset 16.1Q'_*$ by \eqref{Q'-diam}, which is impossible. Thus, the next family $\mathcal{A}_i:=\mathcal{A}_{i-1}\cup\mathcal{A}'_i$ also satisfies property \hyperref[property-F]{(F)} with $T=G_Q$. If $\sum_{U_{Q''} \in \mathcal{A}_i} \diam 2\lambda Q'' \ge 0.04 \diam H_Q$, then we halt and define $\mathcal{A}:= \mathcal{A}_i$. Otherwise, carry out the next step of the induction.

We claim that the process described above always halts, i.e.~there is an integer $n\geq 1$ such that $\mathcal{A} =\mathcal{A}_n$ has property \hyperref[property-F]{(F)} and $\sum_{U_{Q''} \in \mathcal{A}} \diam 2\lambda Q'' \ge 0.04\diam H_Q$. Suppose for contradiction that the process does not halt. We will construct an overly efficient cover of $\Pi_{G_Q}(\bigcup_{U_{Q''} \in \mathcal{N}_2(G_Q)}U_{Q''})$. Suppose that $U_{Q''}\in\mathcal{U}_j\setminus\mathcal{A}_j$ for some $j\geq 1$. Then, by maximality of the family $\mathcal{A}'_j$, at least one of the following occurs: \begin{itemize}
\item[(i)] $2\lambda Q''\cap 2\lambda Q'\neq\emptyset$ for some $Q'\in\mathcal{A}_j$ with $\diam Q'\geq \diam Q''$;
\item[(ii)] $2\lambda Q''\subset 16.1Q'_*$ for some $U_{Q'}\in \Child(Q)$ with $\diam Q'>\diam Q''$.
\end{itemize} In situtation (i), $2\lambda Q''\subset 6\lambda Q'$ for some $U_{Q'}\in\mathcal{A}_j$. In the event that (ii) holds, there are two alternatives: \begin{itemize}
\item[(iii)] $2\lambda Q'' \subset 16.1Q_*'$ for some $U_{Q'}\not \in\mathcal{N}_2(G_Q)$;
\item[(iv)] $2\lambda Q''\subset 16.1Q_*'\subset 2\lambda Q'$ for some $U_{Q'}\in\mathcal{N}_2(G_Q)$ with $\diam Q'>\diam Q''$, and hence  $U_{Q'}\in\mathcal{U}_{i}$ for some $i<j$.
\end{itemize} It follows that for each $j\geq 1$, \begin{equation}\label{recursive-U} \bigcup_{U_{Q''}\in\mathcal{U}_j} 2\lambda Q'' \subset \bigcup_{U_{Q'}\in\mathcal{A}_j}6\lambda Q' \quad\cup\quad \bigcup_{U_{Q'}\not\in\mathcal{N}_2(G_Q)} 16.1Q_*' \quad\cup\quad \bigcup_{i=1}^{j-1}\bigcup_{U_{Q'}\in\mathcal{U}_{i}} 2\lambda Q'. \end{equation} After recursively applying \eqref{recursive-U} and then letting $j\rightarrow\infty$, we obtain \begin{equation*}\bigcup_{U_{Q''}\in\mathcal{N}_2(G_Q)} U_{Q''}\subset \bigcup_{U_{Q''}\in\mathcal{N}_2(G_Q)} 2\lambda Q'' \subset \bigcup_{i=1}^\infty \bigcup_{U_{Q'}\in\mathcal{A}_i} 6\lambda Q' \quad\cup\quad \bigcup_{U_{Q'}\not\in\mathcal{N}_2(G_Q)} 16.1Q'_*.\end{equation*} In particular, by countable subadditivity of measures and by the now familiar fact that $\Haus^1(\Pi_{G_Q}(A))\leq \diam \Pi_{G_Q}(A)\leq \diam A$ for all Borel sets $A\subset\XX$, \begin{align*} \Haus^1\left(\Pi_{G_Q}\left(\textstyle\bigcup\mathcal{N}_2(G_Q)\right)\right) &\leq 3 \sum_{U_{Q''} \in \mathcal{A}_1\cup\mathcal{A}_2\cup\cdots} \diam 2\lambda Q'' + 16.1\sum_{U_{Q'}\not\in\mathcal{N}_2(G_Q)}\diam U_{Q'}\\
&<(3\cdot 0.04 + 16.1\cdot 0.05)\diam H_Q=0.925\diam H_Q. \end{align*} This contradicts \eqref{A construction N_2 estimate}. Therefore, the process above halts and $\mathcal{A}=\mathcal{A}_n$ for some $n\geq 1$. We remark that $\mathcal{A}$ is finite, because $\mathcal{A}\subset\bigcup_{i=1}^n \mathcal{U}_i$ and each $\mathcal{U}_i$ is finite. This proves \eqref{A-estimates 2}.

\subsection{Proof of \eqref{A-estimates}} \label{ss:large-scale}

The proof of \eqref{A-estimates} leans on techniques developed in \S\ref{sec:prelim}. To begin, we describe the large-scale geometry of $*$-almost flat arcs in balls around $\mathcal{A}$ cores. Recall that every $\mathcal{A}$ core belongs to $\mathcal{N}_2(G_Q)$. The first lemma below (Lemma \ref{A-geometry}) is a variant of Lemma \ref{N22-geometry} in the large-scale window $2\lambda Q''$ instead of the small-scale window $16Q''_*$. The second lemma (Lemma \ref{large-scale-geometry}) modifies the arcs obtained in Lemma \ref{A-geometry} to avoid cores $U_{Q'} \subset 2\lambda Q''$ such that $\diam Q' = \diam Q''$. This is necessary to get good control on \hyperref[e:r_T def]{$\rho_{X}$}  for the subarcs $X$ that we apply Lemma \ref{improve} to in the third lemma below (Lemma \ref{large-scale-better}).

\begin{lemma}\label{A-geometry} If $U_{Q''}\in\mathcal{A}$, then there exists a finite set $\mathcal{Y}$ subarcs of arc fragments in $\Gamma^*_{1.98\lambda Q''}$ such that the neighborhoods $\{B_{2^{-M-35}\diam Q''_*}(Y) : Y\in\mathcal{Y}\}$ are pairwise disjoint, $\diam Y\geq 0.00199 \diam 2\lambda Q''$ for all $Y\in\mathcal{Y}$, and in total $\sum_{Y\in\mathcal{Y}} \diam Y \geq 1.23\diam 2\lambda Q''$. (The cardinality of $\mathcal{Y}$ is 2 or 3.)
\end{lemma}

\begin{proof} Let $\tau=f|_{[a,b]}\in S(\lambda Q'')$ be a wide arc for $U_{Q''}$. By our convention in Remark \ref{general projections}, $f(a)$ lies to the left of $f(b)$. Let $T_1=\tau([c,d])$ be a subarc of $\Image(\tau)\cap 1.98\lambda Q''$, where \begin{align*}
c&:=\sup\{t\in[a,b]:\tau(t)\in P_{U_{Q''}}^-\cap\partial(1.98\lambda Q'')\}\text{ and}\\
d&:=\inf\{t\in[a,b]:\tau(t)\in P_{U_{Q''}}^+\cap\partial(1.98\lambda Q'')\}\end{align*} By \eqref{almost-flat-line-estimate} and \eqref{beta-monotone}, there exists a line $L$ such that $\dist(p,L)\leq 2^{-53}\diam 1.98\lambda Q''$ for all $p\in\Image(\tau)$. Since $\Image(\tau)\cap 1.00002Q''_*\neq\emptyset$, repeating the proof of Lemma \ref{proof:proj-lemma} \emph{mutatis mutandis} informs us that $T_1$ (easily) intersects $1.1Q''_*\subset 2^{-11}(1.98\lambda Q'')$. Further, by mimicking the proof of Lemma \ref{N21-auxiliary}, we find that $$\diam T_1\geq (1-2^{-10}-2^{-52})\diam 1.98\lambda Q''\geq 0.98903\diam 2\lambda Q''.$$ Choose a line $L_\tau$ such that \eqref{line-estimate-with-M} holds for $\tau$, choose a $J$-projection $\Pi_\tau$ onto $L_\tau$, and identify $L_\tau$ with $\RR$. By \eqref{Pi-x}, $|\Pi_\tau(w)-w|\leq 2^{-M-47}\diam 2\lambda Q''$ for all $w\in\Image(\tau)$. Thus, the interval $[s_1,s_2]:=\Pi_\tau(T_1)$ is large in the sense that $$s_2-s_1\geq \diam T_1 - 2^{-M-46}\diam 2\lambda Q''\geq 0.98902\diam 2\lambda Q''.$$ Since $\beta_{S^*(\lambda Q'')}(2\lambda Q'')\geq 2^{-M}$, but the excess of $\Image(\tau)$ over $L_\tau$ is comparatively small, we can locate an arc $\xi\in S^*(\lambda Q'')$ and point $x\in \Image(\xi)$ such that $\dist(x,L_\tau)\geq 2^{-M}\diam 2\lambda Q''$. Let $T_2$ be a subarc of $\Image(\xi)\cap 1.98\lambda Q''$ with one endpoint in $\partial(1.98\lambda Q'')$ and one endpoint in $\partial(\lambda Q'')$. We can do this, because the image of every arc in $\Lambda(\lambda Q'')$ intersects $\lambda Q''$ and $Q''\not\in\mathscr{B}^\lambda_0$. Then
\begin{align}\label{big ball T_2 diameter estimate}
    \diam T_2\geq 0.98\radius \lambda Q''=0.245\diam 2\lambda Q''
\end{align} and $\diam T_1+\diam T_2\geq 1.23403\diam 2\lambda Q''.$ If $B_{2^{-M-35}\diam Q''_*}(T_1)\cap B_{2^{-M-35}\diam Q''_*}(T_2)=\emptyset$, then we may take $\mathcal{Y}=\{T_1,T_2\}$.

Suppose otherwise that $B_{2^{-M-35}\diam Q''_*}(T_1)\cap B_{2^{-M-35}\diam Q''_*}(T_2)\neq \emptyset$. For ease of notation, we switch from scale $\diam Q''_*$ to scale $\diam 2\lambda Q''$, recalling that $2^{-M-35}\diam Q''_* \le 2^{-M-48} \diam 2\lambda Q''$.  Let $L_\xi$ be a line such that \eqref{line-estimate-with-M} holds for $\xi$ and let $\Pi_\xi$ be a $J$-projection onto $L_\xi$. Then $$B_{2^{-M-48}\diam 2\lambda Q''}(T_1) \subset B_{(2^{-M-48} + 2^{-M-54})\diam 2\lambda Q''}(L_{\tau}) \subset B_{2^{-M-47}\diam 2\lambda Q''}(L_{\tau}),$$ $$B_{2^{-M-48}\diam 2\lambda Q''}(T_2) \subset B_{(2^{-M-48}+2^{-M-48})\diam 2\lambda Q''}(L_{\xi}) \subset B_{2^{-M-47}\diam 2\lambda Q''}(L_{\xi}),$$ and $L_{\tau}$ intersects $B_2:=B_{2^{-M-45}\diam 2\lambda Q''}(L_{\xi})$ by the triangle inequality.  Continuing to identify $L_{\tau}$ with $\mathbb{R}$, define $$t_1:=\min\{z:z\in L_\tau \cap B_2\}\quad\text{and}\quad t_2:=\max\{z:z\in L_\tau \cap B_2\}.$$ As in the proof of Lemma \ref{N22-geometry}, there are two cases.

For the easier case, suppose that $t_2\leq s_1+0.002\diam 2\lambda Q''$ or $t_1\geq s_2-0.002\diam 2\lambda Q''$. Choose a subarc $\tilde T_1$ of $T_1$ with $\Pi_\tau(\tilde T_1)=[s_1+0.002\diam 2\lambda Q'',s_2-0.002\diam 2\lambda Q'']$. Then by \eqref{almost-flat-line-estimate} and \eqref{Pi-x} $\tilde T_1$ satisfies $$\tilde T_1 \subset B_{2^{-M-53}\diam 2\lambda Q''}([s_1+0.002\diam 2\lambda Q'',s_2-0.002\diam 2\lambda Q'']),$$ and, by the triangle inequality, $\diam \tilde T_1\geq s_2-s_1 - (0.004 + 2^{-M-52})\diam 2\lambda Q'' \geq 0.98501\diam 2\lambda Q''$.  To verify disjointness, we use the triangle inequality again to calculate \begin{align*}
    \gap(B_{2^{-M-47}\diam 2\lambda Q''}(L_{\tau}), B_{2^{-M-47}\diam 2\lambda Q''}(L_{\xi})) \geq (2^{-M-45} - 2^{-M-46})\diam 2\lambda Q''.
\end{align*} Recalling \eqref{big ball T_2 diameter estimate} we see $\diam \tilde T_1+\diam T_2\geq 1.23\diam 2\lambda Q''$. Therefore, in this case we may take $\mathcal{Y}=\{\tilde T_1,T_2\}$.

For the harder case, suppose that \begin{equation}\label{hard-t} t_2 > s_1+0.002\diam 2\lambda Q''\quad\text{and}\quad t_1<s_2-0.002\diam 2\lambda Q''.\end{equation} Our immediate goal is to show that $t_2-t_1$ is relatively small. Let $y, z \in L_\tau$ be such that $y = t_1$ and $z = t_2$ by our identification of $L_\tau$ with $\mathbb{R}$. Since $y,z \in L_{\tau} \cap B_2$, the points $y_\xi, z_{\xi}:=\Pi_\xi(y)$ satisfy $$\max\{|y-y_\xi|, |z-z_\xi|\} \le  2^{-M-45}\diam 2\lambda Q''.$$

Now, define the line $\tilde L_\xi:= L_{\xi} + (y- y_{\xi})$ parallel to $L_{\xi}$ which intersects $y$. Let $\Pi_{\tilde{\xi}}(v):= \Pi_{\xi}(v) + (y-y_{\xi})$ and note that $\Pi_{\tilde{\xi}}$ is a $J$-projection onto $\tilde L_{\xi}$.  Recall that $x \in \Image \xi$, and define $x_{\tilde{\xi}}:=\Pi_{\tilde{\xi}}(x)$, $x_{\tilde{\xi}\tau}:=\Pi_{\tau}(x_{\tilde{\xi}})$, $z_{\tilde{\xi}}:=\Pi_{\tilde{\xi}}(z)$, and $z_{\tilde{\xi}\tau}:=\Pi_{\tau}(z_{\tilde{\xi}})$. Then, we have: \begin{align}
|z_{\tilde{\xi}}-y| & \geq |z-y| - |z_{\tilde{\xi}}-z_{\xi}| - |z_{\xi}-z|\ge   t_2-t_1 -2^{-M-44}\diam 2\lambda Q'',\\
|z_{\tilde{\xi}}-z_{\tilde{\xi}\tau}| & \leq 2\dist(z_{\tilde{\xi}}, L_\tau) \leq 2|z-z_{\tilde{\xi}}| \le 2|z-z_{\xi}| + 2|y - y_{\xi}| \leq 2^{-M-43}\diam 2\lambda Q'', \\
|x_{\tilde{\xi}}-y| &\leq |x-y|+|x_{\tilde{\xi}}-x|\leq (1 +2^{-M-44})\diam 2\lambda Q'',\text{ and}\\
|x_{\tilde{\xi}}-x_{\tilde{\xi}\tau}| & \geq |x-x_{\tilde{\xi}\tau}|-|x-x_{\tilde{\xi}}| \geq 2^{-M-1}\diam 2\lambda Q''.\end{align} By ``similar triangles'', it follows that $$t_2-t_1 - 2^{-M-44}\diam 2\lambda Q''\leq |z_{\tilde{\xi}}-y| = |x_{\tilde{\xi}}-y|\frac{|z_{\tilde{\xi}}- z_{\tilde{\xi}\tau}|}{|x_{\tilde{\xi}}-x_{\tilde{\xi}\tau}|} <
(2\diam 2\lambda Q'')\frac{2^{-M-43}}{2^{-M-1}}.$$ Rearranging, we see that $t_2-t_1<(2^{-M-44}+2^{-41})\diam 2\lambda Q''< 2^{-40}\diam 2\lambda Q''.$ Together with \eqref{hard-t}, it follows that we may choose $\tilde t_1$ and $\tilde t_2$ such that $$\tilde t_1<t_1<t_2<\tilde t_2$$ and $\tilde t_2-\tilde t_1 \le 2^{-39}\diam 2\lambda Q''$. Let $\tilde T_{1.1}$ and $\tilde T_{1.2}$ be a subarcs of $T_1$ with $\Pi_\tau(\tilde T_{1.1})=[s_1,\tilde t_1]$ and $\Pi_\tau(\tilde T_{1.2})=[\tilde t_2,s_2]$.

To see that $B_{2^{-M-48}\diam 2\lambda Q''}(T_{1.1})$ and $B_{2^{-M-48}\diam 2\lambda Q''}(T_{1.2})$ are disjoint, we calculate
\begin{align*}
    \gap(B_{2^{-M-47}\diam 2\lambda Q''}([s_1, \tilde t_1]), B_{2^{-M-47}\diam 2\lambda Q''}([\tilde t_2, s_2])) \ge (2^{-39} - 2^{-M-46}) \diam 2\lambda Q''>0.
\end{align*}
Similarly, to see that $B_{2^{-M-48}\diam 2\lambda Q''}(\tilde T_{1.1}\cup \tilde T_{1.2})$ and $B_{2^{-M-48}\diam 2\lambda Q''}(T_2)$ are disjoint, we estimate
\begin{align*}
    & \gap(B_{2^{-M-48}\diam 2\lambda Q''}(\tilde T_{1.1}\cup \tilde T_{1.2}), B_{2^{-M-48}\diam 2\lambda Q''}(T_2))\\
    & \qquad \ge \gap(B_{2^{-M-47}\diam 2\lambda Q''}(L_{\xi}), B_{2^{-M-47}\diam 2\lambda Q''}([s_1 , \tilde t_1] \cup [\tilde t_2, s_2]))\\
    & \qquad \ge (2^{-M-45} - 2^{-M-46})\diam 2\lambda Q'' > 0.
\end{align*}

We not turn to estimating the diameters of these subarcs.  By \eqref{line-estimate-with-M} and \eqref{Pi-x}
\begin{align}\label{big ball hard case containment}
    \tilde T_{1.1} & \subset B_{2^{-M-53}\diam 2\lambda Q''}([s_1, \tilde t_1]), \qquad \diam \tilde T_{1.1}  \ge \tilde t_1-s_1-2^{-M-52}\diam 2\lambda Q'',\\
    \tilde T_{1.2} & \subset B_{2^{-M-53}\diam 2\lambda Q''}([\tilde t_2, s_2]), \qquad \diam \tilde T_{1.2}  \ge s_2 -\tilde t_2 - 2^{-M-52}\diam 2\lambda Q''.
\end{align}
Recalling \eqref{hard-t}, $\min\{\diam \tilde T_{1.1},\diam \tilde T_{1.2}\}\geq 0.00199\diam 2\lambda Q''$. Moreover, by \eqref{big ball T_2 diameter estimate} and the fact that $2^{-39}\ll 0.00001$  \begin{align*}&\diam \tilde T_{1.1}+\diam \tilde T_{1.2} + \diam T_2\\
&\qquad\geq s_2-s_1-0.00001\diam 2\lambda Q''- 2^{-M-51}\diam 2\lambda Q'' +\diam T_2\\
& \qquad \geq 1.234\diam 2\lambda Q''.\end{align*} In this case, we may take $\mathcal{Y}=\{\tilde T_{1.1}, \tilde T_{1.2}, T_2\}$.
\end{proof}

\begin{figure}\begin{center}\includegraphics[width=.9\textwidth]{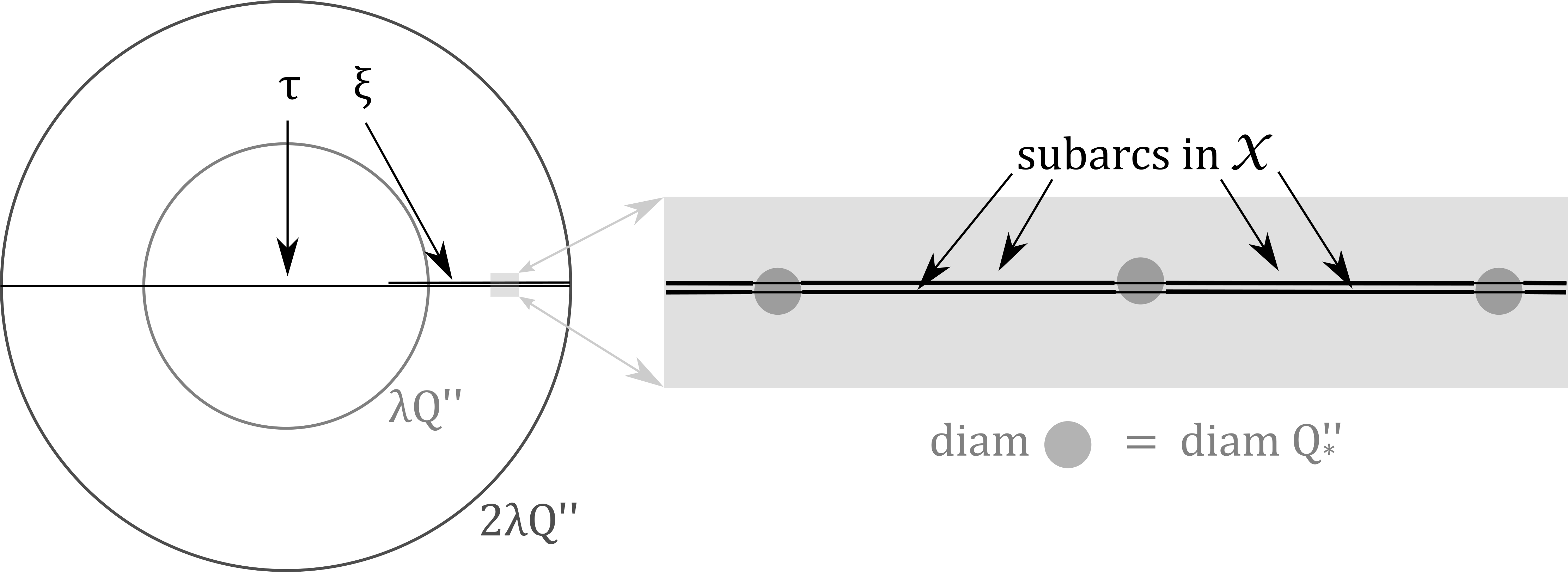}\end{center}\caption[Separated subarcs (large scale geometry)]{Separated subarcs $\mathcal{X}$ associated to a ball $2\lambda Q''$ with $U_{Q''}\in\mathcal{A}$. When $\beta_{S^*(\lambda Q)}(2\lambda Q)$ is sufficiently small, cores $U_{Q'}$ with $\diam Q'_*=\diam Q''_*$ may intersect both of the underlying arcs $\tau$ and $\xi$ used to build $\mathcal{Y}$.}\label{fig:large-scale}\end{figure}

\begin{lemma}\label{large-scale-geometry} If $U_{Q''}\in\mathcal{A}$, then there exists a finite set $\mathcal{X}$ of efficient subarcs of arc fragments in $\Gamma^*_{1.98\lambda Q''}$ such that the set \begin{equation*}\begin{split} \{1.00002Q'_*:&\,Q'\in\Child(Q),\diam Q'=\diam Q''\}\\ &\cup \{B_{2^{-M-35}\diam Q''_*}(X):X\in\mathcal{X}\}\text{ is pairwise disjoint},\end{split}\end{equation*}$\diam X\geq 0.25\diam Q''_*$ for all $X\in\mathcal{X}$, and $\sum_{X\in\mathcal{X}}\diam X \geq 1.11\diam 2\lambda Q''$.
\end{lemma}

\begin{proof} Let $U_{Q''}\in\mathcal{A}$, say $Q''=B(x'',A_\mathscr{H}2^{-k})$, and let $\mathcal{Y}$ be given by the \hyperref[A-geometry]{previous lemma}. Because $\{B_{2^{-M-35}\diam Q''_*}(Y):Y\in\mathcal{Y}\}$ is pairwise disjoint, it suffices to construct a family $\mathcal{X}_Y$ of efficient subarcs $X$ of $Y$ for each $Y\in\mathcal{Y}$ such that $$\{1.00002Q'_*:Q'\in\Child(Q),\;\diam Q'=\diam Q''\}\cup\{B_{2^{-M-35}\diam Q''_*}(X):X\in\mathcal{X}_Y\}$$ is pairwise disjoint, $\diam X\geq \diam Q''_*$ for all $X\in\mathcal{X}_Y$, and in total $\sum_{X\in\mathcal{X}_Y}\diam X\geq 0.904\diam Y$. Then $\mathcal{X}=\bigcup_{Y\in\mathcal{Y}}\mathcal{X}_Y$ satisfies the required properties. In particular, $$\sum_{X\in\mathcal{X}}\diam X\geq 0.904\sum_{Y\in\mathcal{Y}}\diam Y \geq 1.111 \diam 2\lambda Q'',$$ since $\sum_{Y\in\mathcal{Y}}\diam Y\geq 1.23\diam 2\lambda Q''$.

Fix $Y=f([a,b])\in \mathcal{Y}$ and let $\tau\in S^*(\lambda Q)$ be an arc, for which $Y$ is a subarc of $\Image(\tau)\cap 1.98\lambda Q''$. Note that $\diam Y\geq 0.00199\diam 2\lambda Q''> 2^{-9}\diam 2\lambda Q''\geq 2^{4}\diam Q''_*$. Let $L$ be a line such that \eqref{line-estimate-with-M} holds for $\tau$ and let $\Pi_L$ be a $J$-projection onto $L$. By \eqref{Pi-x}, we have $|\Pi_L(x)-x|\leq 2^{-M-47}A_\mathscr{H}^{-1}\diam 2\lambda Q''\leq 2^{-M-38}\diam Y$ for all $x\in \Image(\tau)$. Since $Y$ is compact and connected, $I_0:=\Pi_L(Y)=[c,d]$. Considering  any pair of points $u,v\in Y$ such that $|u-v|=\diam Y$, we see that $$\diam I_0\geq |\Pi_L(u)-\Pi_L(v)| \geq |u-v|-|\Pi_L(u)-u|-|\Pi_L(v)-v|\geq (1-2^{-M-37})\diam Y.$$ Hence $\diam I_0> 0.99999\diam Y>15.999\diam Q''_*$. Form the minimal partition $\mathcal{P}$ of $I_0$ into closed intervals with disjoint interiors that includes the set of intervals \begin{align*}
    \mathcal{J}:=\{I_0\cap \Pi_L(1.00004Q'_*):& Q'\in\Child(Q),\diam Q'=\diam Q'',\\ & 1.00002Q'_*\cap B_{2^{-M-35}\diam Q''_*}(Y)\neq\emptyset\}.
\end{align*} If $\mathcal{J}=\emptyset$, then we may simply take $\mathcal{X}_Y=\{\tilde Y\}$, where $\tilde Y$ is any efficient subarc of $Y$ with $\diam \tilde Y=\diam Y$. Thus, suppose that $\mathcal{J}$ is nonempty. Because every ball in $\XX$ contains a diameter parallel to $L$, for each $J=I_0\cap \Pi_L(1.00004Q'_*)\in\mathcal{J}$, $$\diam J\leq \diam \Pi_L(1.00004Q'_*)=1.00004\diam Q'_*=1.00004\cdot 2^{-k-11}$$ with equality unless $J\cap\{c,d\}\neq\emptyset$. The intervals in $\mathcal{J}$ are uniformly separated. Indeed, for each $J=I_0\cap 1.00004Q'_*$, let $x_J$ denote the center of $Q'_*$, let $y_J\in B_{2^{-M-35}\diam Q''_*}(Y)\cap 1.00002Q'_*$, and let $z_J=\Pi_L(y_j)\in J$; then $\diam J<2^{-k-10}$ and \begin{align*}
    |x_J-z_J|&\leq |x_J-y_J|+|y_J-z_J|\\
    &\leq 1.00002\cdot 2^{-k-12} + 2^{-k-12 -M-35} + 2^{-M-47}A_\mathscr{H}^{-1}\cdot 4\lambda A_\mathscr{H} 2^{-k} < 2^{-k-10}.
\end{align*} Because $\{x_J:J\in\mathcal{J}\}$ is $2^{-k}$-separated, it follows that for all distinct $J_1,J_2\in\mathcal{J}$, \begin{align*}\gap(J_1,J_2)&\geq 2^{-k}-|x_{J_1}-z_{J_1}|-\diam J_1-|x_{J_2}-z_{J_2}|-\diam J_2\\
&\geq (1-2^{-8})2^{-k}=(1-2^{-8})2^{11}\cdot 2^{-k-11}\geq 2^{10}\diam Q''_*.\end{align*} For each interval $I\in\mathcal{P}\setminus\mathcal{J}$, choose an efficient subarc $X_I$ of $Y$ such that $\Pi_L(X_I)\subset I$ and $\diam X_I\geq \diam I$. If $I\in\mathcal{P}\setminus\mathcal{J}$ and $I\cap\{c,d\}\neq\emptyset$, then $I$ lies between two distinct intervals $J_1,J_2\in\mathcal{J}$ and $\diam X_I\geq \diam I\geq \gap(J_1,J_2)\geq 2^{10}\diam Q''_*$. At most two \emph{exceptional} $I\in\mathcal{P}\setminus\mathcal{I}$ contain one of the endpoints of $I_0$; the diameter of an exceptional interval $I$ may be relatively large or small. We assign $$\mathcal{X}_Y:=\{X_I:I\in\mathcal{P}\setminus\mathcal{J}\text{ and }\diam I\geq 0.25\diam Q''_*\},$$ which contains all of the subarcs $X_I$ that we defined with at most two exceptions. (We exclude $X_I$ from $\mathcal{X}_Y$ if exceptionally $I\cap \{c,d\}\neq\emptyset$ \emph{and} $\diam I<0.25\diam Q''_*$.) By design, the $2^{-M-35}\diam Q''_*$-neighborhoods of the subarcs in $\mathcal{X}_Y$ do not intersect $\bigcup\{1.00002Q'_*:Q'\in\Child(Q), \diam Q'=\diam Q''\}$. Furthermore, any pair of distinct $X_{I_1},X_{I_2}\in\mathcal{X}_Y$ enjoy $$\gap(X_{I_1},X_{I_2})\geq \gap(\Pi_L(X_{I_1}),\Pi_L(X_{I_2}))\geq 1.00004\diam Q''_*,$$ because $I_1$ and $I_2$ are separated by an interval in $J\in \mathcal{J}$ that does not intersect $\{c,d\}$.

It remains to estimate the total diameter in $\mathcal{X}_Y$ in terms of $\diam Y$. Let us agree to call an interval $I\in\mathcal{P}\setminus\mathcal{I}$ \emph{short}, \emph{medium}, or \emph{long} if $\diam I<0.25\diam Q''_*$, $0.25\diam Q''_*\leq \diam I<2^{10}\diam Q''_*$, or $\diam I\geq 2^{10}\diam Q''_*$, respectively. Above, we showed that any interval $I\in\mathcal{P}\setminus\mathcal{J}$ lying between two intervals in $\mathcal{J}$ is long. Hence any short or medium interval must contain one of the endpoints of $I_0$. Also, if $I$ is short, then $\diam I<0.25\diam Q''_*<0.016\diam I_0$, because $\diam I_0>15.999\diam Q''_*$ (look above).
After deleting any short intervals from the ends of $I_0$, the remaining interval $$I_{00}:=\overline{I_0\setminus\bigcup\{I\in\mathcal{P}\setminus\mathcal{J}:X_I\not\in\mathcal{X}_Y\}}$$ has $\diam I_{00}\geq 0.968\diam I_0\geq 0.96799\diam Y> 15.486\diam Q''_*$. Now, if $J\in\mathcal{J}$, then $\diam J \le 1.00004\diam Q''_*<0.065\diam I_{00}$. If $J\in\mathcal{J}$ and $I$ is long, then $\diam J\leq 1.00004\diam Q''_*<0.001\diam I$. Since there the number of intervals in $\mathcal{J}$ is at most one more than the number of long intervals, it follows that \begin{equation*}\sum_{X_I\in\mathcal{X}_Y}\diam X_I >  \frac{1-0.065}{1.001}\diam I_{00} > 0.90416\diam Y. \qedhere\end{equation*}
\end{proof}

\begin{lemma}\label{large-scale-better} If $U_{Q''}\in\mathcal{A}$, then there exists a family $\mathcal{L}_{Q''}$ of cores $U_{Q'}\subset 1.99\lambda Q''$ with $Q'\in\Child(Q)$ such that \begin{equation} \label{large-scale-better-total}
\diam 2\lambda Q'' \leq 2\ell(R_Q\cap 1.99\lambda Q'') + 0.91 \sum_{U_{Q'}\in\mathcal{L}_{Q''}} \diam H_{Q'}.\end{equation}
\end{lemma}

\begin{proof} Fix $U_{Q''}\in\mathcal{A}$ and let $\mathcal{X}$ be the family of efficient subarcs of arc fragments in $\Gamma^*_{1.98\lambda Q''}$ given by Lemma \ref{large-scale-geometry}. With the intention to invoke Lemma \ref{improve}, we define $$\mathcal{L}_{Q''}:=\{U_{Q'}:Q'\in\Child(Q)\text{ and }U_{Q'}\cap B_{9\diam Q''_*}(1.98\lambda Q'')\neq\emptyset\}.$$ Property \hyperref[property-F]{(F)} with $\mathcal{F}=\mathcal{A}$ and $T=G_Q$ tells us $\diam Q'\leq \diam Q''$ for all $Q' \in \Child(Q)$ such that $16Q'_* \cap 2\lambda Q'' \not = \emptyset$. This more than ensures $U_{Q'}\subset 1.99\lambda Q''$ for every $U_{Q'}\in\mathcal{L}_{Q''}$.

Let $X\in\mathcal{X}$. By Lemma \ref{large-scale-geometry}, $\diam X\geq 0.25\diam Q''_*$ and \begin{align}\label{section 7 rho estimate in subarcs}\rho_X\leq 2^{-KM}\cdot 2\lambda A_\mathscr{H} \cdot 2^{12}\diam Q''_* \leq 2^{-M-84}\diam Q''_*,
\end{align} since $X\cap 1.00002Q'_*=\emptyset$ whenever $Q'\in\Child(Q)$ and $\diam Q' = \diam Q''$. It follows that $\diam X-2\rho_X\geq 0.99999\diam X$. By Lemma \ref{improve}, with $T=X$ and $\mathcal{F}=\emptyset$,
we obtain \begin{equation}\label{X-estimate} 0.99999\diam X \leq 2.2\,\ell(R_Q \cap B_{9r_X}(X)) + 1.00016 \sum_{U_{Q'}\subset B_{9r_X}(X)}\diam H_{Q'}.\end{equation}

Finally, by Lemma \ref{large-scale-geometry}, the arcs in $\mathcal{X}$ are well-separated from each other compared with \eqref{section 7 rho estimate in subarcs} and have total diameter $\sum_{X\in\mathcal{X}}\diam X\geq 1.11\diam 2\lambda Q''$. Thus, summing \eqref{X-estimate} over all $X\in\mathcal{X}$ and rearranging, we obtain \eqref{large-scale-better-total}.\end{proof}

Because $\{2\lambda Q'':U_{Q''}\in \mathcal{A}\}$ is pairwise disjoint, \eqref{A-estimates} follows by applying \eqref{large-scale-better-total} to each core $U_{Q''}\in \mathcal{A}$. This concludes the proof of Lemma \ref{l:B_2 subset}.

\smallskip

This completes our demonstration of the \hyperref[t:main]{Main Theorem}. In \emph{any} Banach space, a curve of length 1 rarely looks under a magnifying glass like a union of two or more line segments.

\appendix

\section{Unions of overlapping balls in a metric space}

Lemma \ref{union-ball-chains} bounds the radius of a ball containing the union of chains of balls with geometric decay and good separation between balls of similar radii.  Although it can be lowered slightly by increasing the parameter $\xi$, the factor 3 in the lower bound on the gap between balls in level $k$ cannot be made arbitrarily small.

\begin{lemma}[{cf.~\cite[Lemma 2.16]{Schul-AR}}] \label{union-ball-chains} Let $\XX$ be a metric space, let $\xi>6$, and let $r_0>0$. Suppose $\{B(x_i,r_i)\}_{i=1}^{I}$ is a finite ($I<\infty$) or infinite ($I=\infty$) sequence of closed balls in $\XX$ and $(k_i)_{i=1}^I$ is a sequence of integers bounded from below such that \begin{enumerate}
\item[(i)] chain property: for all $j\geq 2$, each pair $(B_1,B_2)$ of balls in the initial segment $\{B(x_i,r_i):1\leq i\leq j\}$ can be connected by a chain of balls from the collection, i.e.~there exists a finite sequence such that the first ball is $B_1$, the last ball is $B_2$, and consecutive balls in the sequence have nonempty intersection;
\item[(ii)] geometric decay: for all $i\geq 1$, we have $r_i\leq \xi^{-k_i}r_0$; and
\item[(iii)] separation within levels: for all $i,j\geq 1$ with $i\neq j$, if $k_i=k_j=k$, then $\gap(B(x_i,r_i),B(x_j,r_j))\geq 3\xi^{-k}r_0$, where $\gap(S,T)=\inf\{\dist(s,t):s\in S,t\in T\}$.
\end{enumerate} Then there exists a unique $M\geq 1$ such that $k_M=\min_{i\geq 1} k_i$, and moreover, \begin{equation}\label{e:union-location} \textstyle\bigcup_{i=1}^I B(x_i,r_i)\subset B\big(x_M,(1+3/\xi)\xi^{-k_M}r_0\big).\end{equation}\end{lemma}

\begin{proof} Let parameters $\xi$ and $r_0$, a sequence $\{B(x_i,r_i)\}_{i=1}^I$, and a sequence $(k_i)_{i=1}^I$ be given with the stated assumptions. Without loss of generality, we may assume that $r_0=1$. Because $\{k_i:i\geq 1\}$ is a set of integers bounded from below, we may choose and fix $M\geq 1$ such that $k_M=\min_{i\geq 1}k_i$. (We prove $M$ is unique later.) Our main task is to prove that for all integers $1\leq n\leq I$, \begin{equation}\label{e:union-goal} \textstyle\bigcup_{i=1}^n B(x_i,r_i) \subset B\big(x_m,(1+2\xi^{-1}+4\xi^{-2}+8\xi^{-3}+\cdots)\xi^{-k_m}\big),\end{equation} where $1\leq m\leq n$ is an index such that $k_m=\min_{i=1}^n k_i$ and $m=M$ whenever $n\geq M$. When $n=1$, there is only one ball and \eqref{e:union-goal} is trivial by (ii). Note that the series in \eqref{e:union-goal} converges, because $\xi>2$. We proceed by strong induction. Let $1\leq N<I$ and suppose that up to relabeling \eqref{e:union-goal} holds for any chain-connected cluster of $N$ or fewer balls satisfying (ii) and (iii). Set $n=N+1$ and choose any index $1\leq m\leq N+1$ such that $k_m=\min_{i=1}^{N+1} k_i$, if $N+1<M$, and set $m=M$, if $N+1\geq M$. Sort the collection $\{B(x_i,r_i):1\leq i\leq N+1\}\setminus \{B(x_m,r_m)\}$ into a finite number of maximal chain-connected components $\mathscr{U}_1,\dots,\mathscr{U}_l$ and note that each $\mathscr{U}_i$ contains at most $N$ balls. See Figure \ref{fig:chain}.

\begin{figure}\begin{center}\includegraphics[width=2in]{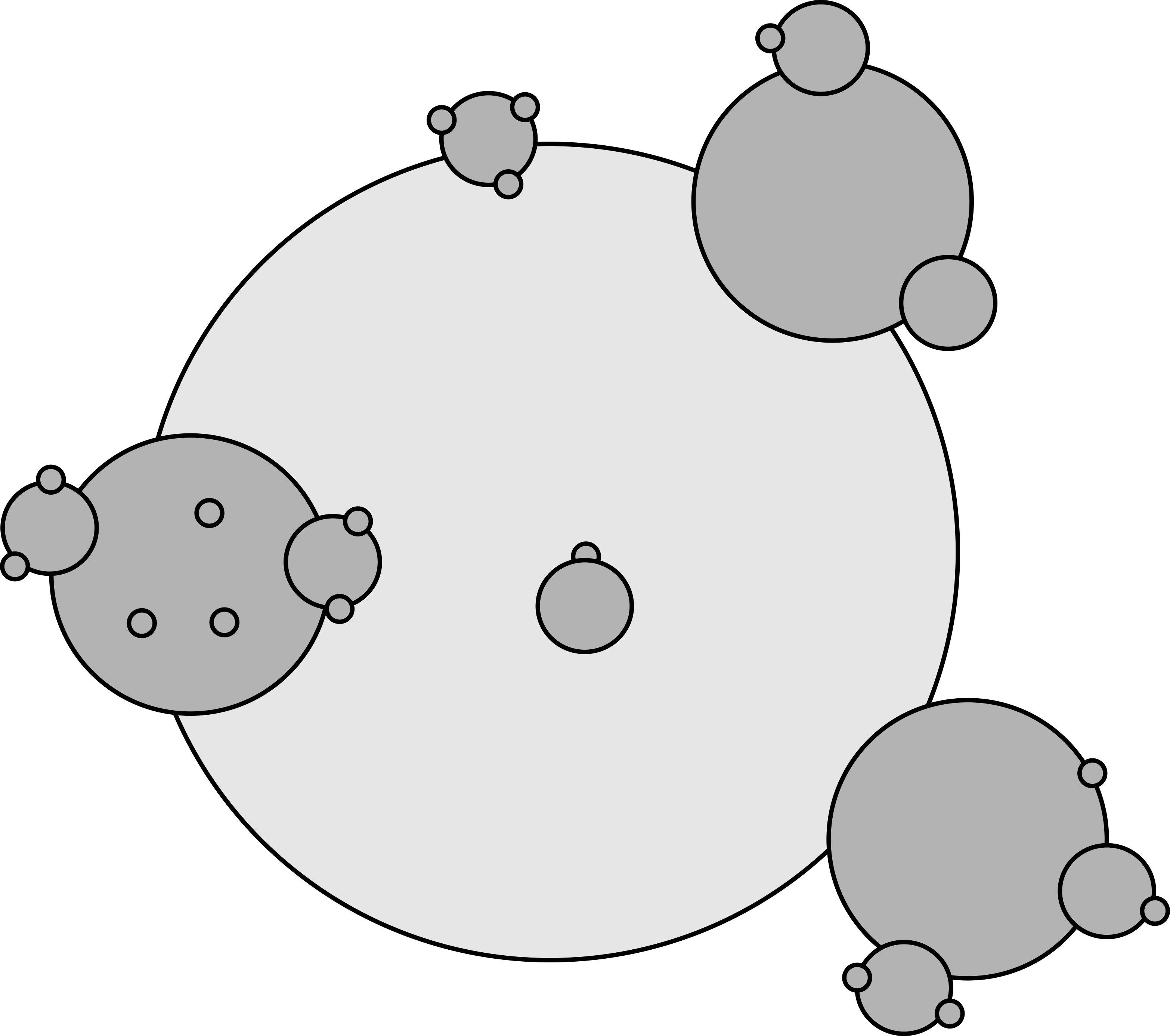}\end{center}\caption[\,Clusters of overlapping balls]{Removing the largest ball (light gray) leaves a finite number of chain-connected ball clusters (dark gray), each of which contains a unique ball of maximal radius.}\label{fig:chain}\end{figure}

Fix a cluster $\mathscr{U}=\mathscr{U}_i$. By the inductive hypothesis, there exists $B(x_j,r_j)\in\mathscr{U}$ so that $$\textstyle\bigcup \mathscr{U}\subset B\big(x_j,(1+2\xi^{-1}+4\xi^{-2}+8\xi^{-3}+\cdots)\xi^{-k_j}\big).$$ Now, $B(x_j,r_j)$ and $B(x_m,r_m)$ both intersect $\bigcup\mathscr{U}$ by (i) and maximality of $\mathscr{U}$. Hence \begin{equation}\label{e:gap-from-above} \gap(B(x_j,r_j),B(x_m,r_m))\leq \diam \textstyle\bigcup \mathscr{U}
\leq (2\xi^{-k_j})/(1-2/\xi)<3\xi^{-k_j}\end{equation} by our requirement that $\xi>6$. By (iii), we conclude that $k_j\neq k_m$. Thus, $k_j\geq k_m+1$, because $k_m$ was chosen to be the minimum level among $k_1,\dots,k_{N+1}$.  Ergo, $$\textstyle\bigcup \mathscr{U} \subset B\big(x_j, \xi^{-1}(1+2\xi^{-1}+4\xi^{-2}+8\xi^{-3}+\cdots)\xi^{-k_m}\big).$$ Thus, by (i) and the triangle inequality, $$\textstyle\bigcup\mathscr{U}\subset B\big(x_m, r_m+2\xi^{-1}(1+2\xi^{-1}+4\xi^{-2}+\cdots)\xi^{-k_m}\big).$$ As this conclusion is true for each family $\mathscr{U}$ and trivially true for $\{B(x_m,r_m)\}$, we obtain  $$\textstyle\bigcup_{i=1}^{N+1}B(x_i,r_i)\subset B\big(x_m, r_m+2\xi^{-1}(1+2\xi^{-1}+4\xi^{-2}+\cdots)\xi^{-k_m}\big).$$  Applying (ii) yields \eqref{e:union-goal} for $n=N+1$. Therefore, by induction, \eqref{e:union-goal} holds for all integers $1\leq n\leq I$. Further, reviewing the inductive step, we conclude that $M$ is the unique index such that $k_M=\min_{i\geq 1} k_i$.

To finish, observe that for any point $z\in \bigcup_{i=1}^I B(x_i,r_i)$, there exists an index $n\geq M$ such that $z\in \bigcup_{i=1}^n B(x_i,r_i)$. By \eqref{e:union-goal}, we have $$z\in B\left (x_M, \left(1+(2\xi^{-1})/(1-2/\xi)\right)\xi^{-k_M}\right).$$ Because $\xi>6$ and $r_0=1$, this yields \eqref{e:union-location}.
\end{proof}

\section{Lipschitz projections onto lines in Banach spaces}\label{sec:smooth}

We now present a class of 1-Lipschitz projections onto a line in a Banach space. Given a real Banach space $\XX$, let $\XX^*$ denote the dual of $\XX$ and let $J: \XX \rightarrow \XX^*$ denote a \textit{normalized duality mapping}, i.e.~ a (nonlinear) map satisfying
\begin{equation}\label{J-properties}
|J(x)|_{\XX^*} = |x| \quad\text{and}\quad \langle J(x), x \rangle = |x|^2\quad\text{for all }x\in \XX,
\end{equation} where $\langle f,x\rangle\equiv f(x)\in\RR$ denotes the natural pairing of $f\in \XX^*$ and $x\in\XX$. Alternatively, $J$ is a subgradient of the convex function $x\in\XX\mapsto (1/2)|x|^2$ (see \cite{Asplund,Kien02}). The norm on any (uniformly) smooth Banach space $\XX$ is Gateaux (uniformly Fr\'echet) differentiable, and thus, $J$ is uniquely determined (see e.g.~\cite[Chapter Two]{Diestel}) when $\XX$ is smooth.

\begin{example}When $\XX=\ell_p$ with $1<p<\infty$, $J(x)= |x|_{\ell_p}^{2-p}y\in \ell_p^*=\ell_{p'},$ where $y=(|x_1|^{p-2}x_1,|x_2|^{p-2}x_2,\dots)$ and $p'$ is the conjugate exponent to $p$.\end{example}

\begin{definition}[{\cite[Definition 3.31]{ENV-Banach}}] \label{def:J-proj}
Let $\XX$ be a Banach space and let $L$ be a one-dimensional linear subspace of $\XX$. Define the \emph{$J$-projection} $\Pi_L$ onto $L$ by
\begin{equation}\Pi_{L}(x) := \langle J(v), x \rangle v\quad\text{for all }x\in\XX,\end{equation} where $J$ is a normalized dual mapping and $v$ is a point in $L$ with $|v|=1$. When $L$ is a one-dimensional affine subspace of $\XX$, define $\Pi_L\equiv p+\Pi_{L-p}(\cdot-p)$ for any choice of $p\in L$.
\end{definition}

\begin{example} Let $\XX=\ell_1^2=(\RR^2,|\cdot|_1)$, let $v=(1,0)$, and let $L=\Span v$ be the $x$-axis. There is a one-parameter family of $J$-projections onto $L$ given as follows. For any $|s|\leq 1/2$, let $w_s=(s,1-|s|)$. With respect to the basis $v,w_s$, $$(x,y)=\Big(x-\frac{s}{1-|s|}y\Big)v+\Big(\frac{1}{1-|s|}y\Big)w_s\quad\text{for all }(x,y)\in\ell^2_1.$$ For any $|s|\leq 1/2$, a $J$-projection onto $L$ is given by $$\Pi_L(x,y)=\Big(x-\frac{s}{1-|s|}y,0\Big)\quad\text{for all }(x,y)\in\ell^2_1.$$ Geometrically, the fibers $\Pi_L^{-1}(x,0)$ are lines parallel to $\Span w_s$ and $\Pi_L^{-1}(v)=v+\Span w_s$ is a supporting line for the unit ball in $\ell_1^2$. When $s=0$, $\Pi_L$ is the orthogonal projection onto $L$. See Figure \ref{fig:cylinder} for an illustration.\end{example}

The following lemma is easily derived from the definition of $\Pi_L$ and \eqref{J-properties}; see \cite[Lemma 2.17]{Badger-McCurdy-1} for sample details.

\begin{lemma}\label{j-proj facts} Let $\XX$ be a Banach space and let $L$ be a line in $\XX$. Every $J$-projection $\Pi_L$ onto $L$ is a 1-Lipschitz projection, i.e.~$\Pi_L(x)\in L$ for all $x$, $\Pi_L(x)=x$ if and only if $x\in L$, and $|\Pi_L(x)-\Pi_L(y)|\leq |x-y|$ for all $x,y$. Moreover, $\dist(x,L)\leq |x-\Pi_L(x)|\leq 2\dist(x,L)$ for every $x\in\XX$.
\end{lemma}

A separated set of points that is sufficiently close to a line admits a canonical ordering (up to choice of orientation) and is locally finite, quantitatively.

\begin{lemma}\label{l:count} Let $\XX$ be a Banach space. Let $\Pi_{L_1}$ and $\Pi_{L_2}$ be $J$-projections onto lines $L_1$ and $L_2$, respectively. If $V\subset \XX$ is a $\delta$-separated set and there exists $0\leq \alpha<1/6$ such that $|v-\Pi_{L_i}(v)|\leq \alpha\delta$ for all $v\in V$ and $i=1,2$, then there exist compatible identifications of $L_1$ and $L_2$ with $\RR$ such that $\Pi_{L_1}(v')\leq \Pi_{L_2}(v'')$ if and only if $\Pi_{L_2}(v')\leq \Pi_{L_2}(v'')$ for all $v',v''\in V$. Moreover, if $v_1,v_2\in V$ and $i=1,2$, then $$|\Pi_{L_i}(v_1)-\Pi_{L_i}(v_2)|\leq |v_1-v_2|\leq (1+3\alpha)|\Pi_{L_i}(v_1)-\Pi_{L_i}(v_2)|.$$ In particular, $V$ is locally finite: $\# V\cap B(x,r\delta)\leq 1+3r$ for every $x\in \XX$ and $r>0$.\end{lemma}

\begin{proof} Repeat the proof of \cite[Lemma 2.1]{Badger-McCurdy-1}, \emph{mutatis mutandis}. (See \cite[Lemma 2.18]{Badger-McCurdy-1} for a related result.) The displayed inequality implies $\Pi_{L_i}|_V$ is injective and $\Pi_{L_i}(V)$ is a $(2/3)\delta$-separated subset of the line $L_i$, whence $V$ is locally finite. To be precise, writing $n\leq\#V\cap B(x,r\delta)$, we have $(2/3)\delta (n-1) \leq \diam \Pi_L(B(x,r\delta))\leq 2r\delta$.\end{proof}

\section{Comments on Lemma 3.28 in [Schul 2007]}\label{appendix328}

In the authors' opinion, the proof of \cite[Lemma 3.28]{Schul-Hilbert} is incorrect and the mistake made in the proof resists a simple fix. The error is in addition to the gap identified in \cite[Remark 3.8]{Badger-McCurdy-1} and is unrelated to the issue of radial versus diametrical arcs discussed in Remarks \ref{r:Schul-1} and \ref{remark:challenges}.

To describe the situation, let us quickly recall the basic setup in \cite{Schul-Hilbert}, which is similar to \S\ref{martingale}, but with some differences. Given a nested sequence $(X_n)_{n=n_0}^\infty$ of $2^{-n}$-nets for a rectifiable curve $\Gamma$ in a Hilbert space $H$, let  $\hat{\mathscr{G}}=\{B(x,A_{\mathscr{G}}):x\in X_n,n\geq n_0\}$ denote the corresponding (truncated) multiresolution family for $\Gamma$. Let $\mathscr{G}$ denote the set of all $Q\in\hat{\mathscr{G}}$ such that $4Q\setminus\Gamma\neq\emptyset$. Choose a Lipschitz continuous parameterization $f:[0,1]\rightarrow \Gamma$ such that $f(0)=f(1)$ and $\#f^{-1}(\{x\})\leq 2$ for $\Haus^1$-a.e.~$x\in\Gamma$. For any ball $Q\in\mathscr{G}$, define $\Lambda(Q)$ to be the set of arcs $\tau=f|_{[a,b]}$ such that $[a,b]$ is a maximal connected component of $f^{-1}(\Gamma\cap Q)$. For each arc $\tau$, define the \emph{arc beta number} $\tilde\beta(\tau)$ by \eqref{e:arc-beta}. Fix parameters $0<\epsilon_1,\epsilon_2\ll_{A_\mathscr{G}} 1$. We say that $\tau$ is \emph{almost flat} and write $\tau\in S(Q)$ if $\tilde\beta(\tau)<\epsilon_2 \beta_\Gamma(Q)$. Fix an integer $J\gg \log_2 A_\mathscr{G}$ and for each $Q\in\mathscr{G}$, define cores $$U_Q:=U_Q^{J,1/64}\quad\text{and}\quad U_Q^x:=U_Q^{J,1/8}$$ using Definition \ref{def-general-cores} above with $\mathscr{G}$ in place of $\mathscr{H}$. For each $Q\in\mathscr{G}$ and $\lambda\in\{1,2,4\}$ such that $\lambda Q\in\mathscr{G}$, choose an arc $\gamma_{\lambda Q}\in\Lambda(\lambda Q)$ containing the center of $Q$. Do this in such a way that $\gamma_{2Q}$ extends $\gamma_Q$ and $\gamma_{4Q}$ extends $\gamma_{2Q}$ whenever the arcs are defined. For each $\lambda\in\{1,2,4\}$, introduce the family  $$\mathscr{G}_2^\lambda:=\{Q\in\mathscr{G}:\gamma_{\lambda Q}\in S(\lambda Q)\text{ and }\beta_{S(\lambda Q)}>\epsilon_1\beta_\Gamma(Q)\}.$$ (Schul's $\mathscr{G}_2^\lambda$ balls correspond to this paper's $\mathscr{B}^\lambda$ balls. Schul also defines $\mathscr{G}_1^\lambda$ and $\mathscr{G}_3^\lambda$ balls, but these are unrelated to Lemma 3.28.) Continuing to follow \cite{Schul-Hilbert}, let us focus on the case $\lambda=1$. Choose a parameter $C_U\gg_{A_\mathscr{G}} 1$ and define  $\Delta_{2.1}$ to be the subfamily of all balls $Q\in\mathscr{G}_2^1$ such that \begin{itemize}
\item almost flat arcs are flatter in $U_Q^x$ than in $Q$: $\beta_{S(Q)}(U_Q^x)\leq C_U^{-1}\beta_{S(\lambda Q)}(Q)$; and,
\item every arc $\tau\in\Lambda(Q)$ such that $\Image(\tau)\cap U_Q\neq\emptyset$ is almost flat: $\tau\in S(Q)$.\end{itemize}
(There are also subfamilies $\Delta_1$ and $\Delta_{2.2}$, which are not relevant here.)

\begin{lemma}[{\cite[Lemma 3.28]{Schul-Hilbert}}] \label{l:328} For every integer $0\leq j\leq J-1$, the family $\Delta'=\{Q\in\Delta_{2.1}:\radius Q=A_\mathscr{G} 2^{-k}\text{ for some }k\equiv j\ (\mathrm{mod}\ J)\}$ satisfies $$\sum_{Q\in\Delta'}\diam Q \lesssim_{A\mathscr{H}} \Haus^1(\Gamma).$$\end{lemma}

Schul's strategy for proving Lemma \ref{l:328} is the one that we described in \S\ref{martingale}. It suffices to construct Borel functions $w_Q:H\rightarrow[0,\infty]$ for each $Q\in\Delta'$, which satisfy the inequalities \eqref{B-mass} and \eqref{B-bounded} with $\Delta'$ in place of $\mathscr{G}$. Build weights $w_Q$ using the cores $U_Q$ as in \S\ref{ss:martingale} with $\diam U_Q$ in place of $\diam H_Q$. (The concept of maximal arc fragments $H_Q$ introduced in Remark \ref{r:maximal-fragments} did not appear in \cite{Schul-Hilbert}, but in any event $\diam H_Q\geq \diam \gamma_Q\approx \diam U_Q$ because $\gamma_Q$ is diametrical for all $Q\in\Delta_{2.1}$.) Define the remainder set $R_Q$ as in \eqref{remainder} and define an auxiliary quantity $s_Q=2\ell(R_Q)+\sum_{Q'\in\Child(Q)}\diam U_{Q'}$. By the argument in \cite[Lemma 3.25, Steps 2--3]{Schul-Hilbert} or Lemma \ref{q-lemma} above, the weights $\{w_Q:Q\in\Delta'\}$ satisfy \eqref{B-mass} and \eqref{B-bounded} so long as there exists a universal constant $0<q<1$ such that \begin{equation}\label{qq-goal} \diam U_Q \leq q s_Q\quad\text{for all } Q\in\Delta'.\end{equation} Unfortunately, the proof of \eqref{qq-goal} in \cite{Schul-Hilbert} contains an error and is incomplete.

\begin{figure}\begin{center}\includegraphics[width=\textwidth]{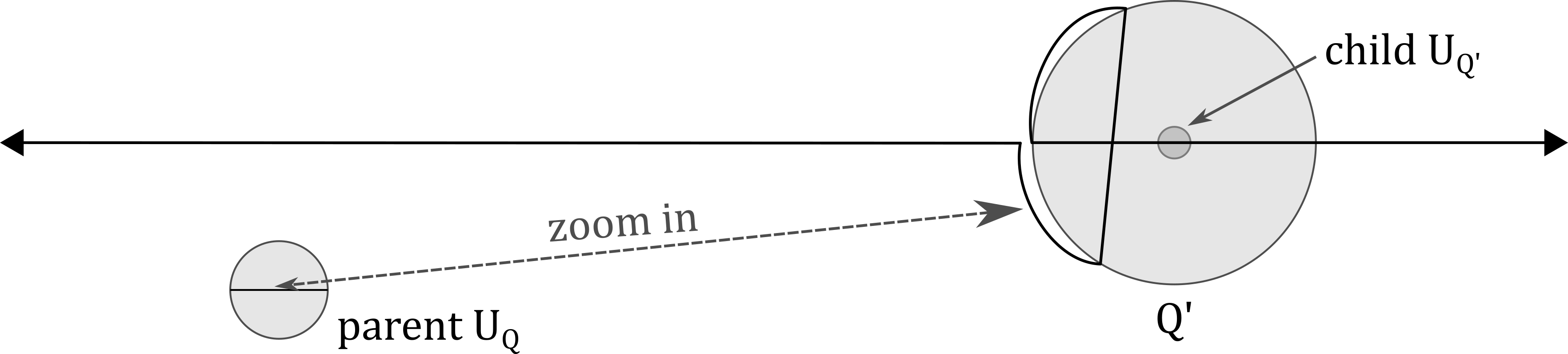}\end{center}\caption[\,Example of an almost flat arc inside of the core of a Schul-type $\Delta_{2.1}$ ball]{Example of an almost flat arc $\tau\in S(Q)$ inside of the core $U_{Q}$ of a Schul-type $\Delta_{2.1}$ ball $Q$. At the resolution of $Q$ or $U_Q$, the portion of $\Image(\tau)$ inside of $U_Q$ is indistinguishable from a line segment. However, zooming in reveals a more complicated picture. The portion of $\Image(\tau)$ inside of the sub-ball $Q'\in\Delta_{2.1}$ is the union of two line segments, only one of which intersects the core $U_{Q'}$. The orthogonal projection $\pi$ from $\Gamma\cap U_Q$ onto the horizontal line through the center of $Q'$ is 1-to-1 when restricted to the cylinder above points in $\pi(\Gamma\cap U_{Q'})$. This shows that \cite[(3.24)]{Schul-Hilbert} used in the proof of Lemma 3.28 is invalid.} \label{fig:L328}\end{figure}

Fix $Q=B(x_Q,A_\mathscr{H}2^{-k})\in\Delta'$. Simplifying the notation from \cite{Schul-Hilbert} slightly, write $Q_*=B(x_Q,(1/64)2^{-k})$. As long as we choose $J$ to be sufficiently large, we have $$Q_*\subset U_Q\subset 1.00001Q_*.$$ Suppose that the central arc $\gamma_Q=f|_{[a,b]}$. Choose an interval $[c,d]\subset[a,b]$ such that $[c,d]$ is a connected component of $\gamma_{Q}^{-1}(0.99999Q_*)$ and $f([c,d])$ has maximal diameter among all such intervals. (This is like extracting $G_Q$ from $H_Q$.) Define $\eta_Q=f|_{[c,d]}$ and let $L$ denote the line passing through $\Edge(\eta_Q)=[f(c),f(d)]$. Because $\gamma_Q$ is almost flat, $\dist(z,L)\lesssim_{A_\mathscr{G}} \epsilon_2 \diam Q_*$ for every $z\in\Image(\eta_Q)$. Finally, let $\pi$ denote the orthogonal projection from $\Gamma\cap 0.99999Q_*$ onto $L$. The first error in the proof is in \cite[(3.24)]{Schul-Hilbert}, which states that for all $x\in\pi(\Gamma\cap 0.99999Q_*)\setminus \pi(R_Q)$, there are at least two points in $\Gamma\cap 0.99999Q_*$ that project onto $x$. In Figure \ref{fig:L328}, we show that this is not the case.

A second (implicit) error appears in the preamble to the proof just before \cite[Remark 3.27]{Schul-Hilbert}. Let $Q'\in\Child(Q)$; in addition to the central arc $\gamma_{Q'}$, the set $S(Q')$ includes at least one other arc $\tau_{Q'}$ with a distinct image. (In the figure, $\gamma_{Q'}$ traces the horizontal line segment and $\tau_{Q'}$ traces the diagonal line segment.) Let $\widehat{\gamma_{Q'}}$ and $\widehat{\tau_{Q'}}$ denote the extensions of the arcs to elements in $\Lambda(Q)$. It is implicitly suggested that the arcs $\widehat{\gamma_{Q'}}$ and $\widehat{\tau_{Q'}}$ are distinct and this together with \cite[(3.24)]{Schul-Hilbert} is what let's one check \eqref{qq-goal}. The example in the figure shows that it is possible for $\Image(\widehat{\gamma_{Q'}})=\Image(\widehat{\tau_{Q'}})$ even though $\Image(\gamma_{Q'})\neq \Image(\tau_{Q'})$. Ultimately, the proof of \eqref{qq-goal} offered in \cite{Schul-Hilbert} is incomplete and unconvincing.

Nevertheless, \eqref{qq-goal} and \cite[Lemma 3.28]{Schul-Hilbert} are correct and this can be shown using the arguments in \S\S3--7. The essential new ingredients that let us wrap up Schul's proof of the Analyst's Traveling Salesman theorem in Hilbert space (Corollary \ref{c:hilbert}) are the classification of cores in Definition \ref{nec-cores}, the case analysis in \S\ref{proof-main}, Lemma \ref{improve}, and Lemma \ref{l:B_2 subset}.

\bibliography{btsp-refs}
\bibliographystyle{amsbeta}

\end{document}